\documentclass[intlim,righttag,10pt]{amsart}
\usepackage{amscd}
\usepackage{amssymb}
\usepackage[all]{xy}
\RequirePackage{mathrsfs}

\def\jcdot{\scriptscriptstyle\bullet}

\def\invlim{\mathop{\vtop{\ialign{##\crcr$\hfill{\lim}\hfil$\crcr
\noalign{\kern1pt\nointerlineskip}\leftarrowfill\crcr\noalign
{\kern -3pt}}}}\limits}
\def\dirlim{\mathop{\vtop{\ialign{##\crcr$\hfill{\lim}\hfil$\crcr
\noalign{\kern1pt\nointerlineskip}\rightarrowfill\crcr\noalign
{\kern -3pt}}}}\limits}
\def\lomapr#1{\smash{\mathop{\relbar\joinrel\longrightarrow}\limits^{#1}}}
 \def\verylomapr#1{\smash{\mathop{\relbar\joinrel\relbar\joinrel\relbar\joinrel\longrightarrow}\limits^{#1}}}
\def\veryverylomapr#1{\smash{\mathop{\relbar\joinrel\relbar\joinrel\relbar
\joinrel\relbar\joinrel\relbar\joinrel\longrightarrow}\limits^{#1}}}
\def\phi{\varphi}
\def\epsilon{\varepsilon}
\let\mathcal\mathscr

\numberwithin{equation}{section}
\newtheorem{theorem}[equation]{Theorem}
 \newtheorem{lemma}[equation]{Lemma}
 \newtheorem{proposition}[equation]{Proposition}
 \newtheorem{corollary}[equation]{Corollary}

\theoremstyle{definition}
\newtheorem{definition}[equation]{Definition}
\newtheorem{remark}[equation]{Remark}

\newtheorem{example}[equation]{Example}
\newtheorem*{acknowledgments}{Acknowledgments}

\let\goth\mathfrak
\newcommand{\ovk}{\overline{K} }
\newcommand{\gp}{\operatorname{gp} }

\newcommand{\dr}{\operatorname{dR} } 
  \newcommand{\hk}{\operatorname{HK} }   
  \newcommand{\LL}{\mathrm {L} }
 \newcommand{\holim}{\operatorname{holim} }
 \newcommand{\eet}{\operatorname{\acute{e}t} }
 \newcommand{\an}{\operatorname{an} }
 \newcommand{\Spec}{\operatorname{Spec} }
 \newcommand{\Spf}{\operatorname{Spf} }
 \newcommand{\Hom}{\operatorname{Hom} }
 
 \newcommand{\Gal}{\operatorname{Gal} }
 \newcommand{\tr}{ \operatorname{tr} }
\newcommand{\can}{ \operatorname{can} }
\newcommand{\id}{ \operatorname{Id} }
\newcommand{\synt}{ \operatorname{syn} }
 \newcommand{\Res}{\operatorname{Res} }
\newcommand{\Lie}{ \operatorname{Lie} }
\newcommand{\Sp}{ \operatorname{Sp} }
 
 \newcommand{\Cone}{\operatorname{Cone} }
\newcommand{\st}{\operatorname{st} }
 \newcommand{\cont}{\operatorname{cont} } 
 \newcommand{\crr}{\operatorname{cr} }
 \newcommand{\gr}{\operatorname{gr} }
\newcommand{\im}{\operatorname{Im} }
 \newcommand{\kr}{^{\jcdot }}

 \newcommand{\sh}{{\mathcal{H}}}
\newcommand{\stt}{{\mathcal{T}}}
 \newcommand{\su}{{\mathcal{U}}}
 
 \newcommand{\so}{{\mathcal O}}
 \newcommand{\sj}{{\mathcal J}}
 \newcommand{\sa}{{\mathcal{A}}}
 \newcommand{\szz}{{\mathcal{Z}}}
 \newcommand{\sx}{{\mathcal{X}}}
 \newcommand{\sss}{{\mathcal{S}}}
  
\newcommand{\sd}{{\mathcal{D}}} 
\newcommand{\sm}{{\mathcal{M}}}

 \newcommand{\wh}{\widehat}
  \newcommand{\Z}{ {\mathbf Z} }
    \newcommand{\laz}{ {\mathcal  L}\hskip-.05cm{ az} }
   \newcommand{\Q}{ {\mathbf Q}}
   \newcommand{\N}{{\mathbf N}}
   \newcommand{\C}{{\mathbf C}}
   
\def\R{{\mathrm R}}
\def\O{{\mathcal O}} 
\def\epsilon{\varepsilon}

\def\E{{\bf E}} \def\A{{\bf A}} 
\def\B{{\bf B}}

\def\bst{{\bf B}_{{\rm st}}}
\def\bstp{{\bf B}_{{\rm st}}^+}
\def\bdr{{\bf B}_{{\rm dR}}}
\def\bdrp{{\bf B}_{{\rm dR}}^+}
\def\acris{{\bf A}_{{\rm cr}}}
\def\vst{{\bf V}_{\rm st}}
\def\ve{v_{\E}}

\begin{document}
 \title[Syntomic complexes and $p$-adic nearby cycles. ]
 {Syntomic complexes and $p$-adic nearby cycles. }
 \author{Pierre Colmez} 
\address{CNRS, IMJ-PRG, Universit\'e Pierre et Marie Curie, 4 place Jussieu,
75005 Paris, France}
\email{pierre.colmez@imj-prg.fr} 
\author{Wies{\l}awa Nizio{\l}}
\address{CNRS, UMPA, \'Ecole Normale Sup\'erieure de Lyon, 46 all\'ee d'Italie, 69007 Lyon, France}
\email{wieslawa.niziol@ens-lyon.fr}
 \date{\today}
\thanks{The authors' research was supported in part by the grant ANR-14-CE25.}
 \begin{abstract}
We compute syntomic cohomology of semistable affinoids in terms of cohomology of $(\varphi,\Gamma)$-modules which,
thanks to work of Fontaine-Herr, Andreatta-Iovita, and Kedlaya-Liu, is known to compute Galois cohomology of these affinoids.
For a semistable scheme over a mixed characteristic local ring this implies a comparison isomorphism, up to some universal constants, 
between truncated sheaves of $p$-adic nearby cycles and  syntomic cohomology sheaves. 
This generalizes the comparison results of Kato, Kurihara, and Tsuji for small Tate twists  (where no constants are necessary) as well as the comparison result of 
Tsuji that holds over the algebraic closure of the field. As an application, we combine this local comparison isomorphism with the theory of finite dimensional 
Banach Spaces and finiteness of \'etale cohomology of rigid analytic spaces 
proved by Scholze to  prove a Semistable conjecture for formal schemes with semistable reduction. 
 \end{abstract}
 \maketitle
 \tableofcontents
 \section{Introduction}
Let $\so_K$ be a complete discrete valuation ring with fraction field
$K$ of characteristic 0 and with perfect
residue field $k$ of characteristic $p$. 
Let $\so_F=W(k)$ and $F=\so_F[\frac{1}{p}]$ so that $K$ is a totally ramified
extension of $F$; let $e=[K:F]$ be the absolute ramification index of $K$.
Let $\so_{\overline K}$ denote the integral closure of $\so_K$ in $\ovk$. 
Set $G_K=\Gal(\overline {K}/K)$, and 
let $\phi=\phi_{W(\overline{k})}$ be the absolute
Frobenius on $W(\overline {k})$.   For a log-scheme $X$ over $\so_K$, $X_n$ will denote its reduction mod $p^n$, $X_0$ will denote its special fiber.

\subsection{Statement of the main results}
\subsubsection{The Fontaine-Messing map}
Let $X$ be a  fine and saturated  log-scheme log-smooth over $\so_K$ equipped with the log-structure coming from the closed point. Denote by $X_{\tr}$ the locus where the log-structure is trivial. This is an open dense subset of the generic fiber of $X$.   For $r\geq 0$, let $\sss_n(r)_{X}$ be the (log) syntomic sheaf modulo $p^n$ on $X_{0,\eet}$. It can be thought of as a derived Frobenius and filtration eigenspace of crystalline cohomology or as a relative Fontaine functor.  Fontaine-Messing \cite{FM}, Kato \cite{K1} have constructed  period morphisms  ($i: X_0\hookrightarrow X, j: X_{\tr}\hookrightarrow X$)
$$\alpha^{\rm FM}_{r,n}:  \quad \sss_n(r)_{X}  \rightarrow i^*Rj_*{\mathbf Z}/p^n(r)^{\prime}_{X_{\tr}},\quad r\geq 0.
$$
from logarithmic syntomic cohomology  to logarithmic $p$-adic nearby cycles. Here  we set  $\Z_p(r)^{\prime}:=\tfrac{1}{p^{a(r)}}\Z_p(r)$,  for $r=(p-1)a(r)+b(r),$ $0\leq b(r)\leq p-1$.

    Assume now that $X$ has semistable reduction over $\so_K$ or is a base change of a scheme with semistable reduction over the ring of integers of a subfield of $K$. That is, locally, $X$ can be written as $\Spec(A)$  for a ring $A$ \'etale over 
  $$
 \so_K[X_1^{\pm 1},\cdots,X_a^{\pm 1}, X_{a+1},\cdots, X_{a+b},X_{a+b+1},\cdots,X_{d},X_{d+1}]/(X_{d+1}X_{a+1}\cdots X_{a+b}-\varpi^h), \quad 1\leq h\leq e.
  $$
  If we put $D:=\{X_{a+b+1}\cdots X_d=0\}\subset \Spec(A)$ then the log-structure on  $\Spec A$  is associated to the special fiber and to the divisor $D$. We have $\Spec(A)_{\tr}=\Spec(A_K)\setminus D_K$. 

  We prove  in this paper that the Fontaine-Messing period map $\alpha^{\rm FM}_{r,n}$, after a suitable truncation,  is essentially a quasi-isomorphism. More precisely, we prove the following theorem.
 \begin{theorem}
 \label{main0}
For   $0\leq i\leq r$,  consider the period map 
\begin{equation}
\label{maineq1}
\alpha^{\rm FM}_{r,n}:\quad  \sh^i(\sss_n(r)_{X}) \rightarrow i^*R^ij_*{\mathbf Z}/p^n(r)'_{X_{\tr}}.
\end{equation}

{\rm (i)} If $K$ has enough roots of unity\footnote{See Section (\ref{def1}) for what it means for a field to contain enough roots of unity. 
For any $K$, the field $K(\zeta_{p^n})$, for $n\geq c(K)+3$, where $c(K)$ is  the conductor of $K$, contains enough roots of unity.} then the kernel  and cokernel of this map are annihilated by $p^{Nr+c_p}$ for a universal constant $N$ {\rm (not depending on $p$, $X$, $K$, $n$ or $r$)} and a constant $c_p$ depending only on $p$ (and $d$ if $p=2$).

{\rm(ii)} In general, the kernel  and cokernel of this map are annihilated by $p^N$ for an integer $N=N(e,p,r)$, which depends on $e$, $r$, but not on $X$ or $n$.
\end{theorem}
For $i\leq r\leq p-1$, it is known that a stronger statement is true: the period map
\begin{equation}
\label{period2}\alpha^{\rm FM}_{r,n}:\quad  \sh^i(\sss_n(r)_{X}) \stackrel{\sim}{\rightarrow} i^*R^ij_*{\mathbf Z}/p^n(r)_{X_{\tr}}.
\end{equation}
is an isomorphism for $X$ a log-scheme log-smooth over a henselian discrete valuation ring $\so_K$ of mixed characteristic. 
This was proved by 
 Kato \cite{K,K1}, Kurihara \cite{Ku}, and Tsuji \cite{Ts,Ts1}. In \cite{Tsc} Tsuji generalized this result  to some \'etale local systems.  As Geisser has shown \cite{Gei}, in the case without log-structure, the isomorphism (\ref{period2}) allows to approximate  the (continuous) $p$-adic motivic cohomology (sheaves) of $p$-adic varieties by their syntomic cohomology; hence to relate $p$-adic algebraic cycles to differential forms. This was used to study algebraic cycles in mixed characteristic \cite{SS}, geometric class field theory \cite{Ku1}, Beilinson's Tate conjecture \cite{AS}, variational Hodge conjecture \cite{BEK}, $p$-adic regulators and special values of $p$-adic $L$-functions \cite{Som}. 
 
 We hope that 
 the ``isomorphism" (\ref{maineq1}) that generalizes (\ref{period2}) will allow to extend the above mentioned applications. Actually, 
(\ref{maineq1}) was already used to approximate motivic cohomology in mixed characteristic. More precisely, in \cite{EN} it is shown that the result of Geisser generalizes to all Tate twists and to the "open" case (i.e., with a possible horizontal divisor at infinity). One gets a higher cycle class map from continuous $p$-adic motivic cohomology to log-syntomic cohomology that is a quasi-isomorphism (by an application of the $p$-adic Beilinson-Lichtenbaum conjectures on the special and the generic fibers). This allows to define well-behaved integral  universal Chern classes to log-syntomic cohomology \cite{NO}.
 Along similar lines, using (\ref{maineq1}), it is shown in \cite{NU}  that the $p$-adic $K$-theory sheaves localized in the log-syntomic topology coincide with log-syntomic Tate twists (up to a direct factor). This is  
analogous to what happens $\ell$-adically \cite{Th,NL}.

  As an application of Theorem~\ref{main0}, one can obtain the following generalization of the
Bloch-Kato's exponential map \cite{BK1}. Let $\sx$ be a quasi-compact formal, semistable scheme over $\so_K$
(for example a semi-stable affinoid). For $i\geq 1$, consider the composition
   $$
   \alpha_{r,i}: \quad H^{i-1}_{\rm dR}(\sx_{K,\tr})\to H^i(\sx,\sss(r))_{\Q}\lomapr{\alpha^{FM}_r}H^i_{\eet}(\sx_{K,\tr},\Q_p(r)).
   $$
If $X$ is  a proper semistable  scheme  $X$ over $\so_K$, and $1\leq i\leq r-1$,
then the $G_K$-representation $V_{i-1}=H^{i-1}_{\eet}(X_{\overline K},\Q_p(r))$ is finite dimensional over $\Q_p$,
$H^{i-1}_{\rm dR}(X_K)$ is finite dimensional over $K$, and $H^{i-1}_{\rm dR}(X_K)=D_{\rm dR}(V_{i-1})$.
The map $\alpha_{r,i}$ for the formal scheme $\sx$ associated to $X$  is then the Bloch-Kato's map \cite{NN}
$$D_{\rm dR}(V_{i-1})\to H^1(G_K,V_{i-1})=H^i_{\eet}(X_K,\Q_p(r)).$$

Easy comparison between de Rham and syntomic cohomologies, together
with Theorem~\ref{main0}, yield the following result.
\begin{corollary}\label{dret}
For $i\leq r-1$, the map
$$
\alpha_{r,i}: H^{i-1}_{\rm dR}(\sx_{K,\tr})\to H^i_{\eet}(\sx_{K,\tr},\Q_p(r))
$$ is an isomorphism. The map
$ \alpha_{r,r}: H^{r-1}_{\rm dR}(\sx_{K,\tr})\rightarrow  H^r_{\eet}(\sx_{K,\tr},\Q_p(r))$ is injective,
but its cokernel can be very large if the dimension of $\sx_K$ is~$\geq 1$.
\end{corollary}
 
 Recall how one shows that the period map $\alpha^{\rm FM}_{r,n}$ from (\ref{period2}) is an isomorphism.  Under the stated assumptions one can do d\'evissage and reduce to $n=1$. 
Then one passes to the tamely ramified extension obtained by attaching the $p$'th root of unity $\zeta_p$. There both sides of the period map (\ref{period2}) are invariant under twisting by $t\in\A_{\crr}$ and $\zeta_p$, respectively, so one reduces to the case $r=i$. 
This is the Milnor case: both sides  compute  Milnor $K$-theory modulo $p$. To see this, one uses symbol maps from Milnor $K$-theory  to the groups on both sides (differential on the syntomic side and Galois on the \'etale side). Via these maps all groups can be filtered compatibly in a way similar to the filtration of the unit group of a local field.  Finally,  the quotients can be computed explicitly by symbols \cite{BK}, \cite{H1}, \cite{Ku}, \cite{Tsc} and they turn out to be isomorphic. This approach to the computation of $p$-adic nearby cycles goes back to the work of Bloch-Kato \cite{BK} who treated the case of good reduction and whose approach was later generalized to semistable reduction by Hyodo \cite{H1}.

\subsubsection{Galois cohomology of semistable affinoid algebras}
 Our proof is of very different nature: 
we construct another local (i.e., on affinoids of a special type, see below) period map, that we call $\alpha^{\laz}_r$. 
Modulo some $(\phi, \Gamma)$-modules theory reductions, it is a version of an integral 
 Lazard isomorphism between Lie algebra cohomology and continuous group cohomology.  We prove directly that it is
a quasi-isomorphism and coincides with Fontaine-Messing's map
up to constants as in Theorem~\ref{main0}.
The (hidden) key input is the purity theorem of Faltings~\cite{FH}, Kedlaya-Liu~\cite{KL}, and Scholze~\cite{Sch1}: it
shows up in the computation of Galois cohomology in terms of $(\varphi,\Gamma)$-modules~\cite{AI,KL}.

 More precisely, let $R$ be the $p$-adic completion of an \'etale algebra over 
 \begin{align*}
R_\Box :=
\O_F\{X^{\pm 1}_1,\cdots, X^{\pm 1}_a, X_{a+1},\cdots,X_{d+1}\}/(X_{d+1}X_{a+1}\cdots X_{a+b}-\varpi),
\end{align*}
 where $a,b$ are integers and $d\geq a+b$. 
We equip the associated formal schemes with the log-structure induced by the ``divisor at infinity'': 
$X_{a+b+1}\cdots X_d=0$ and the special fiber. These are  formal log-schemes with semistable reduction over $\so_K$. 

To compute the syntomic cohomology of $R$, we need to choose good crystalline coordinates for $R$, i.e., we need to write it as
a quotient of a log-smooth $\so_F$-algebra $R^+_{\varpi}$. The easiest way to do that is to add one variable. We start with $\so_K$. Take the algebra  $r^+_\varpi:=\O_F[[X_0]]$
equipped with the log-structure associated to $X_0$.  Sending $X_0$ to $\varpi$ induces
a surjective morphism $r_\varpi^+\to \O_K$.

Let now $R^+_{\varpi,\Box}$ be the completion
of $\O_F[X_0,X^{\pm 1}_1,\cdots, X^{\pm 1}_a, X_{a+1},\cdots,X_{d+1},\frac{X_0}{X_{a+1}\cdots X_{a+b}}]$
for the $(p,X_0)$-adic topology. We add $X_0$ to the log-structure induced from $R_{\Box}$.
Sending $X_0$ to $\varpi$ induces a surjective morphism
$R^+_{\varpi,\Box}\to R_{\Box}$ whose kernel is generated by
$P=P_\varpi(X_0)$.
This provides a closed embedding
of ${\rm Spf}\,R_\Box$ into a formal log-scheme $\Spf R^+_{\varpi,\Box}$ that is log-smooth over $\O_F$.
Let $R^+_\varpi$ be the unique \'etale lift of $R$ over $R^+_{\varpi,\Box}$ complete for the
$(p,P)$-adic topology (which is also the $(p,X_0)$-adic topology). We equip it with the log-structure induced from $R^+_{\varpi,\Box}$. 
Sending $X_0$ to $\varpi$ induces a surjective morphism
$R^+_{\varpi}\to R$, whose kernel is generated by $P=P_\varpi(X_0)$. We denote by $R^{\rm PD}_{\varpi}$ the $p$-adic PD-envelope of $R$ in $R^+_{\varpi}$. We endow it with Frobenius $\phi_{\rm Kum}$ induced by $X_i\to X_i^p$, $0\leq i\leq  d+1$. This is our PD-coordinate system of~$R$. 
 
 The syntomic cohomology of $R$ can then be computed by the complex 
 \begin{equation}
 \label{def2}
    {\rm Syn}(R,r):=\Cone(\xymatrix@C=1.6cm{F^r\Omega_{R^{\rm PD}_{\varpi}}\kr\ar[r]^-{p^r-p\kr\phi_{\rm Kum}}&\Omega_{R^{\rm PD}_{\varpi}}\kr})[-1],
 \end{equation}
 where $\Omega\kr_{R^{\rm PD}_{\varpi}}:=R^{\rm PD}_{\varpi}\otimes_{R^{+}_{\varpi}}\Omega\kr_{R^{+}_{\varpi}/\so_F}$:
we have $H^i_{\rm syn}(R,r)=H^i({\rm Syn}(R,r))$.

Now, let $\overline R$ be the ``maximal extension of $R$ unramified outside the divisor
$X_{a+b+1}\cdots X_d=0$ in characteristic~$0$ (i.e.,~after inverting $p$)''.
Let $G_R={\rm Gal}(\overline R/R)$.  Modulo the identification of $\alpha^{\rm FM}_r$ and
$\alpha^{\laz}_r$, claim (i) of Theorem~\ref{main0} is a consequence of the
following more precise statement that relates Galois cohomology of $G_R$ (and \'etale cohomology
of the associated rigid space) with values
in $\Z_p(r)$ and syntomic cohomology in degrees~$\leq r$. This is the first main result of our paper.
\begin{theorem}\label{main00}
If $K$ contains enough roots of unity then the maps 
\begin{align*}
\alpha_r^{\laz}: & \quad \tau_{\leq r}{\rm Syn}(R,r)\to \tau_{\leq r}{\rm R}\Gamma_{\rm cont}(G_R,\Z_p(r)),\\
\alpha_{r,n}^{\laz}: & \quad  \tau_{\leq r}{\rm Syn}(R,r)_n \to \tau_{\leq r}{\rm R}\Gamma_{\rm cont}(G_R,\Z/p^n(r))\to  \tau_{\leq r}\R\Gamma((\Sp R[1/p])_{\tr,\eet},\Z/p^n(r))
\end{align*}
are quasi-isomorphisms up to $p^{Nr}$, for a universal constant $N$.
\end{theorem}
This statement is more precise than that of Theorem \ref{main0} because we do not require \'etale localization. 
Claim (ii) of Theorem \ref{main0}  follows by a simple descent argument from claim (i). The constants are very crude and no doubt can be improved upon. 
\begin{remark}
The same descent argument would allow to remove the condition ``$K$ contains enough roots of unity'' in
Theorem~\ref{main00}, at the cost of introducing constants depending on $K$.  It seems likely that one could
make the constants depend only on the valuation of the different of $K$ and not on $K$ itself.

In fact, it does not seem unreasonnable to think that a statement analogous to
Corollary~\ref{dret}, describing \'etale cohomology in terms of differential forms,
 should be valid for a general affinoid algebra, smooth
in characteristic~$0$, with constants depending on the ``defect of smoothness in mixed
characteristic'' (i.e.,~the $p$-adic valuation of suitable Jacobians).
The point is that one can use the purity theorems to pass to a finite
cover with very small defect of smoothness where one could try to
apply the methods of this paper (suitably modified).

We hope to come back to these questions in a future work.
\end{remark}

\subsubsection{Period isomorphisms in the semistable case}
  Theorem \ref{main0} holds also for base changes of semistable schemes and implies that we have a quasi-isomorphism (up to $p^{Nr}$ for a universal constant $N$)
  \begin{equation}
\label{maineq}
\alpha^{\rm FM}_{r,n}: \quad \sh^i(\sss_n(r)_{X_{\so_{\overline K}}}) \rightarrow \overline i^*R^i\overline j_*{\mathbf Z}/p^n(r)'_{X_{\overline K,\tr}},\quad i\leq r,
\end{equation}
where $\overline i :X_{0,\overline k}\hookrightarrow X_{\so_{\overline K}}$, $\overline j: X_{\overline K,\tr}\hookrightarrow X_{\so_{\overline K}}$. 
Hence an isomorphism
 \begin{equation}
 \label{etale}
 \alpha^{\rm FM}_{r}: \quad H^i_{\synt}(X_{\so_{\ovk}},r)_{\Q}\stackrel{\sim}{\to} H^i(X_{\ovk,\tr},\Q_p(r)),\quad i\leq r.
 \end{equation}
 This recovers Tsuji's result \cite[Theorem  3.3.4]{Ts} that he proves by similar techniques as his results over $K$. Namely, in the case when $\zeta_{p^n}\in K$ and $r=i$,  Tsuji notices that twisting the nearby cycles by $\zeta_{p^n}$  allows  to perform  d\'evissage and to reduce to $n=1$, where 
one again can use explicit computations by symbols. This allows him to prove Theorem \ref{main0} in this case  \cite[Theorem 3.3.2]{Ts} up to $p^{N}$ with a  constant $N$ that depends only on $p$ and $r$. To pass to all $i\leq r$ he needs to show that twisting by $t$ does not change the syntomic cohomology sheaves. This he is able   to do over $\ovk$ \cite[Theorem 2.3.2]{Ts} and up to a constant $N$ that depends only on $p$, $r$, and $i$. 

   The isomorphism (\ref{etale}) is used traditionally to prove $p$-adic comparison theorems by the syntomic method \cite {FM}, \cite{K1}, \cite{Ts}, \cite{Tso}, \cite{YY}. Recall how the argument goes in the case of a trivial divisor at infinity, i.e., 
when $X_{\tr}=X_K$. One composes the map $\alpha^{\rm FM}_r$ with the natural map 
$$H^i(X_{\ovk},\Q_p(r))\to H^i_{\hk}(X)\otimes_{F}\bst \{r\},
$$
where $H^i_{\hk}(X)$ is the Hyodo-Kato cohomology of $X$,
 to obtain the period map
$$\alpha:\quad H^i(X_{\ovk},\Q_p)\otimes_{\Q_p}\bst \simeq H^i_{\hk}(X)\otimes_{F}\bst .
$$
This map is shown to be compatible with Poincar\'e duality. Hence it is an isomorphism. 

  We have realized that we can prove that the period map $\alpha$ is an isomorphism without evoking Poincar\'e duality by 
 techniques supplied by the theory 
of finite dimensional Banach Spaces \cite{CF}. This reproves the classical comparison theorem for semistable schemes via a modified syntomic method (including the case of non-trivial divisor at infinity treated in Tsuji \cite{Tso} and Yamashita-Yasuda \cite{YY}). Recall that the semistable comparison theorem for schemes was proved also by different methods in \cite{Fa0}, \cite{Fa}, \cite{Ni0}, \cite{Ni}, \cite{Bh}, \cite{BE2} (see \cite{NE} for the proof that most of these methods yield the same period map). They all use Poincar\'e duality. The only proof of a general comparison theorem that does not use Poincar\'e duality is the proof of the de Rham conjecture for rigid analytic spaces by Scholze \cite{Sch}.

   The fact that we do not need Poincar\'e duality anymore allows us to combine the comparison isomorphism (\ref{etale}) with finiteness of \'etale cohomology of proper rigid analytic spaces proved by Scholze \cite{Sch} to prove
 the second main result of this paper -- a comparison theorem for semistable formal schemes.
\begin{corollary}
\label{comp20}{\rm (Semistable conjecture)}
 Let $\sx$ be a proper semistable formal scheme  over $\so_K$. 
There exists a natural $\bst$-linear Galois equivariant period isomorphism
$$
\alpha:\quad H^i(\sx_{\ovk,\tr},\Q_p)\otimes_{\Q_p}\bst \simeq H^i_{\hk}(\sx)\otimes_{F}\bst ,
$$
that preserves the Frobenius and the monodromy operators, and induces a filtered  isomorphism 
$$
\alpha:\quad H^i(\sx_{\ovk,\tr},\Q_p)\otimes_{\Q_p}\bst
 \simeq H^i_{\dr}(\sx_{K,\tr})\otimes_{K}\bdr.
$$
 \end{corollary}
\subsection{Sketch of the proofs  of the main results}We will now sketch the proofs of Theorem \ref{main00} and Corollary \ref{comp20}. 
\subsubsection{Local computations}
 We will start with Theorem \ref{main00}. 
If $v>u>0$, let 
$r_\varpi^{[u]}$ (resp.~$r_\varpi^{[u,v]}$) be the ring
of analytic functions over $F$ convergent on the  disk $v_p(X_0)\geq u/e$ (resp.~the  annulus  $v/e\geq v_p(X_0)\geq u/e$), where $e=[K:F]$,
and let 
$R_\varpi^{[u]}=r_\varpi^{[u]}\widehat\otimes_{r_\varpi^+}R_\varpi^+$ (resp.~$R_\varpi^{[u,v]}=r_\varpi^{[u,v]}\widehat\otimes_{r_\varpi^+}R_\varpi^+$).
Our ring $R^{\rm PD}_{\varpi}$ is very close to $R_\varpi^{[u]}$, $u=\frac{1}{p-1}$.

Set $u=\frac{p-1}{p}$ and $v=p-1$ in what follows (for $p=2$, one has to modify slightly the arguments and take $u=\frac{3}{4}$, $v=\frac{3}{2}$,
but we will ignore this for the introduction).  To define the period map
$$
\alpha_r^{\laz}:  \quad \tau_{\leq r}{\rm Syn}(R,r)\to \tau_{\leq r}{\rm R}\Gamma_{\rm cont}(G_R,\Z_p(r)),
$$
  using $(\varphi,\partial)$-modules techniques,
we  produce a string of ``quasi-isomorphisms'' (a {\it ``quasi-isomorphism''} is a map of complexes
whose associated map on cohomology has kernels and cokernels killed by $p^{Cr}$ for
some absolute constant~$C$). We start with quasi-isomorphisms

\begin{equation}
\label{qis1}
{\rm Syn}(R,r)\simeq{\rm Kum}(R_\varpi^{[u]},r)
\stackrel{\tau_{\leq r}}{\simeq}{\rm Kum}(R_\varpi^{[u,v]},r),
\end{equation}
where the Kummer complexes ${\rm Kum}(\cdot,r)$ are defined by formulas analogous to  (\ref{def2}).
Here we are forced to truncate the morphism   ${\rm Kum}(R_\varpi^{[u]},r)
\to {\rm Kum}(R_\varpi^{[u,v]},r)$ because it is a quasi-isomorphism up to  too large constants in degrees~$>r$. The second quasi-isomorphism in (\ref{qis1}) is proved using the $\psi$ operator  -- left inverse to $\varphi$ -- and acyclicity of the $\psi=0$ eigencomplexes. 

   We called the complexes in (\ref{qis1}) Kummer because they are related to the Kummer extension
$K(\varpi^{1/p^\infty})$ of $K$.
Let us explain what we mean by that.
Let $\E_{\overline R}^+$ be the  tilt of $\overline R$ and set $\A_{\overline R}^+=W(\E_{\overline R}^+)$. 
Choose, inside $\overline{R}$, elements $X_i^{p^{-n}}$, for
$i=1,\dots,d$ and $n\in\N$, satisfying the obvious relations
(i.e. $X_i^{p^{-0}}=X_i$ and $(X_i^{p^{-(n+1)}})^p=X_i^{p^{-n}}$ if $n\geq 0$).
If $i=1,\dots,d$, let $x_i=(X_i,X_i^{1/p},\dots)\in \E^+_{\overline{R}}$.

Sending $X_0$ to $[\varpi^\flat]$, where $\varpi^{\flat}$ is a sequence of $p$'th roots of $\varpi$,  and $X_i$ to $[x_i]$, if $i=1,\dots,d$,
induces an embedding $\iota_{\rm Kum}:R_{\varpi}^+\hookrightarrow 
\A_{\overline{R}}^+$ which commutes with Frobenius  $\varphi$
and is compatible with filtrations (with filtration on $\A_{\overline{R}}^+$ by powers of the
Kernel of the natural map $\theta:\A_{\overline{R}}^+\to \widehat{\overline{R}}$).  
By continuity, this extends to
embeddings 
$$
\iota_{\rm Kum}:R_\varpi^{[u]}\hookrightarrow \A^{[u]}_{\overline R},\quad
\iota_{\rm Kum}:R_\varpi^{[u,v]}\hookrightarrow \A^{[u,v]}_{\overline R},$$
which commute with Frobenius and filtration. 

  But, {\it if $K$ contains enough roots of unity}, the ring $R_\varpi^{[u,v]}$ can also be embedded into period rings via a cyclotomic embedding
(i.e.~using the cyclotomic extension of $K$ instead of its Kummer extension).
 This will however change the Kummer Frobenius into the cyclotomic Frobenius $\varphi_{\rm cycl}$. 
Let us sketch how this is done. 
  Let $i(K)$ be the largest integer $i$ such that $K$ contains $\zeta_{p^i}$ and, if $n\in\N$, let $K_n=K(\zeta_{p^{n+i(K)}})$.
The assumption that $K$ contains enough roots of unity implies (thanks to the field of norms theory and
extra work~\cite{CF}) the existence of $\pi_K\in\A_{\overline K}$, fixed
by ${\rm Gal}(\overline K/K_\infty)$, such that $\varphi(\pi_K)=f(\pi_K)$, with
$f(X_0)$ analytic and bounded on the annulus $0<v_p(X_0)\leq v/e$, and such that
modulo $(p,[p^\flat]^{1/p})$ we have $\overline \pi_K=(\varpi_{K_n})_{n\in\N}\in\E^+_{\overline K}$ -- a tower of uniformizers of the fields $K_n$.
We can embed $r_\varpi^+$ into $\A_{\overline K}$, sending $X_0$ to $\pi_K$ and this extends
to the completion $r_\varpi$ of $r_\varpi^+[X_0^{-1}]$ for the $p$-adic topology.  The image\footnote{This is the image 
by $\varphi^{-i(K)}$ of the usual
$\A_K$ from the theory of $(\varphi,\Gamma)$-modules.}  $\A_K$ is stable under the action
of $\varphi$, hence $r_\varpi$ inherits a Frobenius that we note $\varphi_{\rm cycl}$.
This Frobenius does not preserve $r_\varpi^+$, 
but we have $\varphi_{\rm cycl}(X_0)\in r_{\varpi}^{(0,v]+}$ - the ring of analytic functions over $F$ with integral values on the annulus $0 < v_p(X_0)\leq v/e$.

   For a general $R$, first we extend $\varphi_{\rm cycl}$ to a Frobenius on $R_{\varpi}$ by setting 
   $$
   \varphi_{\rm cycl}(X_i)=X_i^p, \quad 1\leq i\leq d; \quad \quad \varphi_{\rm cycl}(\tfrac{X_0}{X_{a+1}\cdots X_{a+b}})=(\tfrac{X_0}{X_{a+1}\cdots X_{a+b}})^p\tfrac{\varphi_{\rm cycl}(X_0)}{X_0^p}
   $$
 Then, for $n\in\N$,
we consider $R_{n}$  --  the subalgebra of $\overline{R}$
generated by $R$
and $\O_{K_n}$, $X_i^{p^{-n}}$ for $i=1,\dots,d$,
and $\frac{\varpi_{K_n}}{(X_{a+1}\cdots X_{a+b})^{p^{-n}}}$.
Set $R_{\infty}=\cup_{n\in\N}R_{n}$.
The ring $R_{\infty}[\frac{1}{p}]$ is a Galois extension of
$R[\frac{1}{p}]$, with Galois group
$\Gamma_R$ which is the semi-direct product
$$1\to\Gamma'_R\to\Gamma_R\to\Gamma_K\to 1,$$
where 
\begin{align*}
\Gamma'_R={\rm Gal}(R_\infty[\tfrac{1}{p}]/K_\infty\cdot R[\tfrac{1}{p}])\simeq\Z_p^d,\quad 
\Gamma_K={\rm Gal}(K_\infty/K)\simeq 1+p^{i(K)}\Z_p,
\end{align*}
and $a\in 1+p^{i(K)}\Z_p$
acts on $\Z_p^d$ by multiplication by~$a$.

We define the cyclotomic embedding\footnote{Here, and everywhere in the paper,
``deco'' stands for ``decoration'', and is one of PD, $[u]$, $[u,v]$, $(0,v]+$, \dots}
 $\iota_{\rm cycl}:
R_\varpi^{\rm deco}\to\A_{\overline R}^{\rm deco}$ using the embedding $\iota_{\rm cycl}:
r_\varpi^{\rm deco}\to\A_{\overline K}^{\rm deco}$ and 
sending $X_j$ to $[x_j]$ if $1\leq j\leq d$. It is compatible with Frobenius and with filtration. 
We denote by $\A_R^{\rm deco}$  the image
of $R_\varpi^{\rm deco}$ 
by $\iota_{\rm cycl}$.
Then, the rings
$\A_R, \A_R^{[u,v]}, \A_R^{(0,v]+}$ are stable
by $G_R$ which acts through $\Gamma_R$. 

  Coming back to cohomology, using standard crystalline techniques, we show that change of Frobenius does not affect syntomic cohomology. That is,  we produce quasi-isomorphisms
$${\rm Kum}(R_\varpi^{[u,v]},r)\simeq{\rm Syn}((R_\varpi^{[u,v]}\widehat{\otimes} R_\varpi^{[u,v]})^{\rm PD},
\varphi_{\rm Kum}\otimes\varphi_{\rm cycl},r)\simeq {\rm Cycl}(R_\varpi^{[u,v]},r),$$
where the last complex is defined as in (\ref{def2}) using the ring $R_\varpi^{[u,v]}$ and the cyclotomic Frobenius $\varphi_{\rm cycl}$.

    Next, choosing a basis of $\Omega^1$ and using the isomorphism $R_\varpi^{[u,v]}\simeq \A_R^{[u,v]}$,
we change ${\rm Cycl}(R_\varpi^{[u,v]},r)$ into a Koszul complex:
$${\rm Cycl}(R_\varpi^{[u,v]},r)\simeq{\rm Kos}(\varphi,\partial,F^r\A_R^{[u,v]}).$$
Then, multiplying by suitable powers of $t$, we can get rid of the filtration
(in degrees~$\leq r$; this is the only place where the truncation is absolutely necessary),
which changes the derivatives into the action of the Lie algebra of $\Gamma_R$, to obtain:
$$\tau_{\leq r}{\rm Kos}(\varphi,\partial,F^r\A_R^{[u,v]})\simeq 
\tau_{\leq r}{\rm Kos}(\varphi,\Lie\Gamma_R,\A_R^{[u,v]}(r)).$$
Standard analytic arguments \`a la Lazard change this into a Koszul complex for the group:
$${\rm Kos}(\varphi,\Lie\Gamma_R,\A_R^{[u,v]}(r))\simeq {\rm Kos}(\varphi,\Gamma_R,\A_R^{[u,v]}(r)).$$
Then, using $(\varphi,\Gamma)$-module techniques, we get  ``quasi-isomorphisms''
$${\rm Kos}(\varphi,\Gamma_R,\A_R^{[u,v]}(r))
\simeq {\rm Kos}(\varphi,\Gamma_R,\A_R^{(0,v]+}(r))
\simeq {\rm Kos}(\varphi,\Gamma_R,\A_R(r)).$$
Here we use the operator $\psi_{\rm cycl}$ -- the left inverse to $\phi_{\rm cycl}$ and argue by acyclicity of the $\psi=0$ eigencomplex.

  Finally, general nonsense about Koszul complexes gives us a quasi-isomorphism
$${\rm Kos}(\varphi,\Gamma_R,\A_R(r))\simeq[\xymatrix{{\rm R}\Gamma_{\rm cont}(\Gamma_R,\A_R(r))
\ar[r]^-{1-\varphi}&{\rm R}\Gamma_{\rm cont}(\Gamma_R,\A_R(r))}];$$
and general relative $(\varphi,\Gamma)$-module theory
gives quasi-isomorphisms
\begin{align*}
 & [\xymatrix{{\rm R}\Gamma_{\rm cont}(\Gamma_R,\A_R(r))
\ar[r]^-{1-\varphi}&{\rm R}\Gamma_{\rm cont}(\Gamma_R,\A_R(r))}]\simeq \\
 & [\xymatrix{{\rm R}\Gamma_{\rm cont}(G_R,\A_{\overline R}(r))
\ar[r]^-{1-\varphi}&{\rm R}\Gamma_{\rm cont}(G_R,\A_{\overline R}(r))}]\simeq 
  \R\Gamma_{\rm cont}(G_R,\Z_p(r)).
\end{align*}
The first quasi-isomorphism  is proved by the almost \'etale and decompletion techniques developed in the relative setting by Andreatta-Iovita \cite{AI} and by Kedlaya-Liu \cite{KL}; the second one follows from the relative Artin-Schreier theory
(i.e.~the exact sequence $0\to \Z_p\to \xymatrix{\A_{\overline R}\ar[r]^-{1-\varphi}&\A_{\overline R}}\to 0$). This finishes the definition of the quasi-isomorphism $\alpha^{\laz}_r$ from Theorem \ref{main00}.

  Since, by the $p$-adic $K(\pi,1)$-Lemma \cite{Sch} and by Abkhyankar's Lemma, we have 
  $$
   \R\Gamma_{\rm cont}(G_R,\Z/p^n(r))  \simeq \R\Gamma((\Sp R[1/p])_{\tr,\eet},\Z/p^n(r)), 
   $$ the above sequence of quasi-isomorphisms constructs a quasi-isomorphism
  $$
 \alpha_{r,n}^{\laz}:\quad  \tau_{\leq r}{\rm Syn}(R,r)_n\simeq  \tau_{\leq r}\R\Gamma((\Sp R[1/p])_{\tr,\eet},\Z/p^n(r)).
    $$

\subsubsection{Finite dimensional Banach Spaces and semi-stable conjecture} To prove Corollary \ref{comp20},
first we show  (this is a simplification for the sake of the introduction)   
  that we have the long exact sequence
\begin{align}
\label{long}
\to  (H^{i-1}_{\dr}(\sx_{K,\tr})\otimes_K\bdrp)/F^r & \to H^i_{\synt}(\sx_{\so_{\ovk}},r)_{\Q}\to  (H^i_{\hk}(\sx)\otimes_{F}\bstp)^{\phi=p^r,N=0} \\
  &     \to (H^i_{\dr}(\sx_{K,\tr})\otimes_K\bdrp)/F^r \to \notag
\end{align}
For $i < r$, the above long exact sequence yields short exact sequences
$$
0\to H^i_{\synt}(\sx_{\so_{\ovk}},r)_{\Q}\to  (H^i_{\hk}(\sx)\otimes_{F}\bstp)^{\phi=p^r,N=0}\to     (H^i_{\dr}(\sx_{K,\tr})\otimes_K\bdrp)/F^r)\to 0
$$
To prove this, we observe that 
$f_i:(H^i_{\hk}(\sx)\otimes_{F}\bstp)^{\phi=p^r,N=0}\to     (H^i_{\dr}(\sx_{K,\tr})\otimes_K\bdrp)/F^r)$ is the evaluation on $C=\overline K^{\wedge}$ of a map
of finite dimensional Banach Spaces~\cite{CB}.
Recall that these Spaces are to be thought of
as finite dimensional $C$-vector spaces up to finite dimensional
$\Q_p$-vector spaces, and come equipped with a Dimension, which is a pair of numbers $(a,b)$, $a\in \N$, $b\in \Z$, where $a$ is the $C$-dimension, 
and $b$ is the $\Q_p$-dimension. Dimension is additive on short exact sequences. 

But, $H^i_{\synt}(\sx_{\so_{\ovk}},r)_{\Q}$, $i\leq r$, 
is a  finite dimensional $\Q_p$-vector space:
we have the quasi-isomorphism (\ref{etale}) with \'etale cohomology and Scholze proved finite dimensionality of the latter \cite{Sch}. 
This implies that the cokernel of $f_i$, viewed as a map of Banach Spaces,
 is of Dimension $(0,d_i)$.
On the other hand, the Space $(H^i_{\dr}(\sx_{K,\tr})\otimes_K\bdrp)/F^r$ is a successive extension of $C$-vector spaces. 
The theory of Banach Spaces implies that the map $(H^{i-1}_{\dr}(\sx_{K,\tr})\otimes_K\bdrp)/F^r\to 
{\rm Coker}\,f_i$ is zero, hence ${\rm Coker}\,f_i=0$,
as wanted. 

 Now, since we have the Hyodo-Kato isomorphism $$H^i_{\hk}(\sx)\otimes_{K_0}K\simeq H^i_{\dr}(\sx_{K,\tr}),$$ the pair $(H^i_{\hk}(\sx), H^i_{\dr}(\sx_{K,\tr}))$ 
is a $(\varphi,N)$-filtered module (in the sense of Fontaine). The above short exact 
sequence and a ``weight" argument shows that ${\vst}(H^i_{\hk}(\sx), H^i_{\dr}(\sx_{K,\tr}))\simeq H^i(X_{\ovk,\tr},\Q_p)$. Here ${\vst}(\cdot)$ is Fontaine's functor from filtered Frobenius modules to Galois representations. 
 The short exact sequence and dimension count give also that $t_N(H^i_{\hk}(\sx))=t_H(H^i_{\dr}(\sx_{K,\tr}))$, where $t_N(D)=v_p(\det \phi)$ and $t_H(D)=\sum_{i\geq 0}i\dim_K( F^iD/F^{i+1}D)$. 
 The theory of Banach Spaces and finite dimensionality of ${\vst}(H^i_{\hk}(\sx), H^i_{\dr}(\sx_{K,\tr}))$ imply now that the pair $(H^i_{\hk}(\sx), H^i_{\dr}(\sx_{K,\tr}))$ is weakly admissible and this proves Corollary \ref{comp20}. 
\begin{acknowledgments}
Parts of this article were written during our visit to MSRI for the program ``New geometric
Methods in Number Theory and Automorphic Forms".
 We would like to thank the Institute for  its support and hospitality. 
 
 We thank Peter Scholze for pointing out an inaccuracy  in one of our  arguments. 
We thank Piotr Achinger for helping us with the proof of Lemma \ref{Abh}, and 
Fabrizio Andreatta, Kiran Kedlaya, Ruochuan Liu, Shanwen Wang 
for helpful conversations related to the subject of this paper. 
Last but not least, we thank the referee for a careful reading and 
her/his demands concerning the exposition. 
\end{acknowledgments}

\subsubsection{Notation and Conventions}
\begin{definition}
\label{1saint}
Let $N\in {\mathbf N}$. For a morphism $f: M\to M^{\prime}$ of ${\mathbf Z}_p$-modules, we say that $f$ is 
{\it $p^N$-injective} (resp. {\it $p^N$-surjective}) if its kernel (resp. its cockernel) is annihilated by $p^N$ 
and we say that $f$ is {\it $p^N$-isomorphism} if it is $p^N$-injective and $p^N$-surjective. 
We define in the same way the notion of {\it $p^N$-exact sequence} or $p^N$-acyclic complex 
(complex whose cohomology groups are annihilated by $p^N$) as well as the notion of {\it $p^N$-quasi-isomorphism}
 (map in the derived category that induces a $p^N$-isomorphism on cohomology). 
\end{definition}

We will use a shorthand for certain homotopy limits. Namely,  if $f:C\to C'$ is a map  in the dg derived category of an abelian category, we set
$$[\xymatrix{C\ar[r]^f&C'}]:=\holim(C\to C^{\prime}\leftarrow 0).$$ 
We also set
$$
\left[\begin{aligned}
\xymatrix{C_1\ar[d]\ar[r]^f & C_2\ar[d]\\
C_3\ar[r]^g & C_4
}\end{aligned}\right]
:=[[C_1\stackrel{f}{\to} C_2]\to [C_3\stackrel{g}{\to} C_4]],
$$ 
where the diagram in the brackets  is a commutative diagram in the dg derived category. If this diagram is strictly commutative this  simply amounts to taking the total complex of the associated double complex. In this paper this is the case everywhere but in Section 5.

\section{Formal log-schemes and  period rings}
This is a preliminary section introducing all the rings that we are going
to use in the next three sections.  Most of its content consists
of variations on standard techniques from the theory of $(\varphi,\Gamma)$-modules~\cite{Hr,CC,CF,AA,AB,AI,KL}.

\subsection{The implicit function theorem}
Let $\lambda:\Lambda_1\to\Lambda_2$ be a continuous morphism
of topological rings.
Let $\Lambda'_1=\Lambda_1\{Z\}/(Q)$, where\footnote{If $\Lambda$ is a topological
ring, and $Z=(Z_1,\dots,Z_s)$, we let $\Lambda\{Z\}$ design the ring of
power series $\sum_{{\bf k}\in\N^s}a_{\bf k}Z^{\bf k}$, where $a_{\bf k}\in\Lambda$
goes to $0$ when ${\bf k}\to\infty$.}
$Z=(Z_1,\dots,Z_s)$ and $Q=(Q_1,\dots,Q_s)$.
We would like to extend $\lambda$ to $\Lambda'_1$; this amounts to solving
the equation $Q^\lambda(Y)=0$ in $\Lambda_2$, where, if
$F\in\Lambda_1\{Z\}$, we note $F^\lambda\in\Lambda_2\{Z\}$ the series obtained
by applying $\lambda$ to the coefficients of $F$.

Let $J=(\frac{\partial Q_j}{\partial Z_i})_{1\leq i,j\leq s}\in 
{\bf M}_s(\Lambda_1\{Z_1,\dots,Z_s\})$.

\begin{proposition}\label{extr6}
Assume there exists:

$\bullet$  $z\in\Lambda_2$, 

$\bullet$ an ideal $I$ of $\Lambda_2$ such
that $z^{-2}I\subset \Lambda_2$ and 
$\Lambda_2$ is separated and complete with respect to the $z^{-2}I$-adic topology,

$\bullet$ $Z_\lambda=(Z_{1,\lambda},\dots,Z_{s,\lambda})\in\Lambda_2^s$
and $H_\lambda\in z^{-1}{\bf M}_s(\Lambda_2)$, 

\noindent such that:

$\bullet$ the entries of $Q^\lambda(Z_\lambda)$ belong to $I$,

$\bullet$ the entries of
$H_\lambda J^\lambda(Z_\lambda)-1$
belong to $z^{-1}I$.

\noindent Then the equation $Q^\lambda(Y)=0$ has a unique solution
in $Z_\lambda+ (z^{-1}I)^s$.
\end{proposition}
\begin{proof}
Consider $f:\Lambda_2^s\to(\Lambda_2[z^{-1}])^s$ defined by
$$f(x)=x-H_\lambda Q^\lambda(x+Z_\lambda).$$
We have $f(0)\in (z^{-1}I)^s$.
Now, we can write $f(x)-f(y)$ as:
\begin{align*}
(1-H_\lambda J^\lambda(Z_\lambda))(x-y)& \\ +
H_\lambda \Big(\big(Q^\lambda(y+Z_\lambda)&-Q^\lambda(x+Z_\lambda)-J^\lambda(x+Z_\lambda)(y-x)\big)
  +
\big(J^\lambda(x+Z_\lambda)-J^\lambda(Z_\lambda)\big)(y-x)\Big).
\end{align*}
The hypothesis implies that, if $x-y\in (z^{-1}I)^s$,
then $f(x)-f(y)\in z^{-2}I\cdot (z^{-1}I)^s$.
Hence $f$ is contracting on $(z^{-1}I)^s$ and has a unique fixed
point in this module.  As $x$ is a fixed point if and only if
$Q^\lambda(x+Z_\lambda)=0$, this concludes the proof.
\end{proof}
\begin{remark}\label{extr6.1}
If $\Lambda'_1$ is \'etale over $\Lambda$, there exists $H\in {\bf M}_s(\Lambda_1\{Z_1,\dots,Z_s\})$
such that $HJ-1$ has its entries in $(Q_1,\dots,Q_s)$, hence
if $Q^\lambda(Z_\lambda)$ has coordinates in an ideal $I$,
then $H^\lambda J^\lambda-1$ has entries in $I$. 
If $\Lambda_2$ is separated and complete for the $I$-adic topology, we can take
$z=1$ and $H_\lambda=H^\lambda(Z_\lambda)$ in the above proposition, hence the equation
$Q^\lambda(Y)$ has a unique solution in $Z_\lambda+I^s$.
\end{remark}

\subsection{Kummer theory}
\subsubsection{Local fields}\label{def1}
Let $F\subset K$ be as in the introduction, so that
$K$ is a finite, totally ramified extension of $F$. 
Choose a uniformiser $\varpi$ of $K$, and let $P_\varpi$ be
its minimal polynomial over $F$ (hence $P_\varpi$ is an Eisenstein polynomial
of degree $e=[K:F]$).

Choose a compatible system $(\zeta_{p^n})_{n\in\N}$ of
roots of unity, with $\zeta_{p^0}=1$, $\zeta_p\neq 1$, and $\zeta_{p^n}^p=\zeta_{p^{n-1}}$,
and define $F_n=F(\zeta_{p^n})$,
and $F_\infty=\cup_n F_n$.

Let $i(K)$ be the largest integer $n$ such that $K$ contains
$\zeta_{p^n}$.
If $n\in\N$, we set $K_n=K(\zeta_{p^{n+i(K)}})$, so that $K_0=K$, but
$K$ is strictly contained in $K_n$ if $n\geq 1$, even if $K$ contains
some roots of unity. Set $K_\infty=\cup_{n\in\N}K_n$.
Let $$\delta_K=ev_p({\goth d}_{K/F_{i(K)}})\in\N.$$
We say that $K$ {\it contains enough roots of unity} if 
$$v_p({\goth d}_{K/F_{i(K)}})<\tfrac{1}{2p}-\tfrac{1}{[F_{i(K)}:F]}\quad{\text{ or, equivalently,}}\quad 
\delta_K+[K:F_{i(K)}]<\tfrac{e}{2p}.$$
If we fix $K$,
then $K(\zeta_{p^n})$ contains enough roots of unity for all
$n$ big enough (this is a restatement of the fact that
$K_\infty/F_\infty$ is almost \'etale). More precisely,
 by \cite[Proposition ~4.5]{CF}, it suffices to take 
$n\geq c(K)+2$ (if $p=2$, one needs $n\geq c(K)+3$), where
$c(K)$ is the conductor of $K$ (i.e.~$t\geq c(K)$ if and only if
$K$ is fixed by the higher ramification subgroup $G_F^t$ of $G_F={\rm Gal}(\overline F/F)$).

\subsubsection{Formal log-schemes}
\label{formalscheme}
Let $h\in[1,e]$ be an integer (this $h$ is the multiplicity that appears when we base
change a semi-stable scheme over $\O_{K'}$, with $K'\subset K$, to $\O_K$; hence it is~$\leq e$
in applications, although almost everywhere it could as well be arbitrary).

Let $a,b,c$ be integers and $d=a+b+c$. Set $X=(X_1,\dots,X_{d})$ and define
\begin{align*}
R_\Box &=\O_K\{X,\frac{1}{X_1\cdots X_a},\frac{\varpi^h}{X_{a+1}\cdots X_{a+b}}\}\\ &=
\O_K\{X,X_{d+1},X_{d+2}\}/(X_{d+2}X_1\cdots X_a-1,X_{d+1}X_{a+1}\cdots X_{a+b}-\varpi^h).
\end{align*}
We endow $R_\Box$ with the spectral norm.

Let $R$ be the $p$-adic completion of an \'etale algebra over $R_\Box$,
so that $R_\Box$ provides a {\it system of coordinates} for $R$
(or ${\rm Spf}\,R_\Box$ a {\it frame} for ${\rm Spf}\,R$). We equip the
associated formal schemes with the log-structure induced by the divisor at infinity: 
$X_{a+b+1}\cdots X_d=0$ and the special fiber. We obtain that way formal log-schemes that are log-smooth over $\so_K^{\times}$ 
(in fact, both schemes have semistable reduction over $\so_K$).

\subsubsection{Adding an arithmetic variable}
To compute the syntomic cohomology of $R$, we need to write it as
a quotient of a log-smooth algebra over $\O_F$ (it is already log-smooth
over $\O_K^{\times}$).  The cheapest way to do so is to add an extra variable.

We denote by $r^+_\varpi$ and $r_\varpi$ the algebras $\O_F[[X_0]]$
and $\O_F[[X_0]]\{X_0^{-1}\}$ with the log-structure associated to $X_0$.  Sending $X_0$ to $\varpi$ induces
a surjective morphism $r_\varpi^+\to \O_K$.

Let $X'=(X_0,X)$ and let $R^+_{\varpi,\Box}$ be the completion
of $\O_F[X',\frac{1}{X_1\cdots X_a},\frac{X_0^h}{X_{a+1}\cdots X_{a+b}}]$
for the $(p,X_0)$-adic topology. We add $X_0$ to the log-structure induced from $R_{\Box}$.
Sending $X_0$ to $\varpi$ induces a surjective morphism
$R^+_{\varpi,\Box}\to R_{\Box}$ whose kernel is generated by
$P=P_\varpi(X_0)$.
This provides a closed embedding
of ${\rm Spf}\,R_\Box$ into a formal log-scheme $\Spf R^+_{\varpi,\Box}$ that is log-smooth over $\O_F$.

We can write $R$ as
$$R=R_\Box\{Z_1,\dots,Z_t\}/(Q_1,\dots,Q_t),$$
with $\det(\frac{\partial Q_i}{\partial Z_j})$ invertible in $R$.
Choose liftings $\tilde Q_j$ of the $Q_j$'s in $R^+_{\varpi,\Box}$,
and let $R^+_\varpi$ be the quotient
of the completion of $R^+_{\varpi,\Box}[Z_1,\dots,Z_t]$ for the
$(p,P)$-adic topology (which is also the $(p,X_0)$-adic topology)
by $(\tilde Q_1,\dots,\tilde Q_t)$. We equip it with the log-structure induced from $R^+_{\varpi,\Box}$. 
We have that  $\det(\frac{\partial \tilde Q_i}{\partial Z_j})$ is invertible in $R^+_\varpi$
(since it is modulo $P$). Hence  $R^+_\varpi$ is \'etale over $R^+_{\varpi,\Box}$
and  log-smooth over $\O_F$.

Sending $X_0$ to $\varpi$ induces a surjective morphism
$R^+_{\varpi}\to R$, whose kernel is generated by $P=P_\varpi(X_0)$.

\subsubsection{Divided powers and localizations on smaller balls or annuli} 
We let $R_\varpi^{\rm PD}$ be the $p$-adic completion of 
$R^+_\varpi[\frac{P^k}{k!},\,k\in\N]$.  As $P=X_0^e$ modulo $p$,
we have $R^+_\varpi[\frac{P^k}{k!},\,k\in\N]=R^+_\varpi[\frac{X_0^k}{[\frac{k}{e}]!},\ k\in\N]$. We equip $R_\varpi^{\rm PD}$ with the log-structure induced from $R^+_{\varpi}$ and $X_0$. We have defined the following diagram of formal log-schemes.
\begin{equation}
 \label{diagram}
 \xymatrix@R=20pt@C=15pt{
  & \Spf R^{\rm PD}_{\varpi}\ar[rd] & \\
 \Spf R\ar[d]\ar@{^{(}->}[ru]\ar@{^{(}->}[rr]\ar@{}[rr] & &\Spf R^{+}_{\varpi}\ar[d]\\
\Spf R_{\Box}\ar[d]\ar@{^{(}->}[rr] & & \Spf R^+_{\varpi,\Box}\ar[d]\\
  \Spf \so_K\ar[d]\ar@{^{(}->}[rr] & & \Spf r^+_{\varpi}\ar[lld]\\
 \Spf \so_F 
 }
 \end{equation}
Crystalline cohomology, hence syntomic cohomology, is defined by means of
$R_\varpi^{\rm PD}$, but we will show that, up to absolute constants
depending only on $u,v,r$, we can replace $R_\varpi^{\rm PD}$
by the rings $R_\varpi^{[u]}$ or $R_\varpi^{[u,v]}$ below.
So, if $0<u\leq v$, let\footnote{We will need $R_\varpi^{[u,v]}$, with
$u,v$ satisfying a number of inequalities: $\frac{p-1}{p}\leq u\leq\frac{v}{p}<1<v<pv_K$
and $\big(1+\frac{p+1}{2p(p-1)})u>\frac{1}{p-1}$.  For $p\geq 3$, we can take
$u=\frac{p-1}{p}$ and $v=p-1$.  For $p=2$, we can take $u=\frac{3}{4}$ and $v=\frac{3}{2}$.}

$\bullet$ $R_\varpi^{(0,v]+}$ be the $p$-adic completion of
$R_\varpi^+[\frac{p^{\lceil vi/e\rceil}}{X_0^i}, i\in\N]$,
and $R_\varpi^{(0,v]}=R_\varpi^{(0,v]+}[\frac{1}{X_0}]$,

$\bullet$ $R_\varpi^{[u]}$ be the $p$-adic completion of
$R_\varpi^+[\frac{X_0^i}{p^{[ui/e]}},\ i\in\N]$,

$\bullet$ $R_\varpi^{[u,v]}$ be the $p$-adic completion of
$R_\varpi^+[\frac{X_0^i}{p^{[ ui/e]}},\frac{p^{\lceil vi/e\rceil}}{X_0^i},\ i\in\N]$,

$\bullet$ $R_\varpi$ be the $p$-adic completion of
$R_\varpi^+[\frac{1}{X_0}]$.
                                                         
\smallskip
We have $R^{[u]}_\varpi\subset R_\varpi^{\rm PD}$ if $u\leq \frac{1}{p}$
and $R_\varpi^{\rm PD}\subset R_\varpi^{[u]}$ if $u\geq\frac{1}{p-1}$. 
If $\frac{1}{p}<u < \frac{1}{p-1}$ then there exists a constant $C(u)$ such that $p^{C(u)}R^{[u]}_\varpi\subset R_\varpi^{\rm PD}$. 

\smallskip
We note $v^{(0,v]}$, $v^{[u]}$, $v^{[u,v]}$ the spectral valuations
on $R_\varpi^{(0,v]}$, $R_\varpi^{[u]}$ and $R_\varpi^{[u,v]}$ respectively.

\begin{remark}
Denote by $r_\varpi^{\rm deco}$ the ring $R_\varpi^{\rm deco}$ corresponding to
$R=\O_K$.
One can describe the rings $r_\varpi^{\rm deco}$ as rings of Laurent series
with coefficients satisfying growth conditions.  Namely:

$\bullet$ $r_\varpi^{\rm PD}$ is the set of $f=\sum_{i\in\N}a_i\frac{X_0^i}{[\frac{i}{e}]!}$,
with $a_i\in\O_F$ going to $0$ when $i\to\infty$.

$\bullet$ $r_\varpi^{[u]}$ is the set of $f=\sum_{i\in\N}a_i\frac{X_0^i}{p^{[\frac{iu}{e}]}}$,
with $a_i\in\O_F$ going to $0$ when $i\to\infty$.
It is the ring of analytic functions over $F$ with integral values on the disk $v_p(X_0)\geq u/e$.

$\bullet$ $r_\varpi^{[u,v]}$ is the set of $f=\sum_{i\in\Z}a_iX_0^i$,
with $a_i\in F$ and $v_p(a_i)\geq (-iv)/e$ if $i\leq 0$ and $v_p(a_i)\geq (-iu)/e$ if $i\geq 0$, and with the differences going to $+\infty$ with $i\to \pm\infty$. 
It is the ring of analytic functions over $F$ with integral values on the annulus $\frac{u}{e}\leq
v_p(X_0)\leq \frac{v}{e}$.

We have $r_\varpi^{[u]}\subset r_\varpi^{[u,v]}$
 and the quotient $r_\varpi^{[u,v]}/r_\varpi^{[u]}$ involves only negative powers of $X_0$.

$\bullet$ $r_\varpi^{(0,v]+}$ is the set of $f=\sum_{i\in\Z}a_iX_0^i$, 
with $a_i\in \O_F$ and $v_p(a_i)\geq (-iv)/e$ if $i\leq 0$, the differences going to $+\infty$ with $i\to -\infty$. 
It is the ring of analytic functions over $F$ with integral values on the annulus $0<
v_p(X_0)\leq \frac{v}{e}$.

$\bullet$ $r_\varpi$ is the set of $f=\sum_{i\in\Z}a_iX_0^i$, 
with $a_i\in \O_F$ and $a_i\to 0$ when $i\to -\infty$.     
\end{remark}

Note that going from $R^+_\varpi$ to $R_\varpi^{\rm PD}$,
$R_\varpi^{[u]}$, $R_\varpi^{[u,v]}$ or $R_\varpi$ involves only the ``arithmetic''
variable $X_0$ (going from
$R^+_\varpi$ to $R_\varpi^{[u]}$ (resp.~$R_\varpi^{[u,v]}$)
amounts to localization of 
$X_0$ on the disk $v_p(X_0)\geq\frac{u}{e}$ 
(resp.~on the annulus $\frac{v}{e}\geq v_p(X_0)\geq\frac{u}{e}$). 
This can be translated into the following isomorphisms:
\begin{align*}
R_\varpi^{\rm PD}=r_\varpi^{\rm PD}\widehat\otimes_{r_\varpi^+}R_\varpi^+,&\quad
R_\varpi=r_\varpi\widehat\otimes_{r_\varpi^+}R_\varpi^+,\\
R_\varpi^{[u]}=r_\varpi^{[u]}\widehat\otimes_{r_\varpi^+}R_\varpi^+,\quad
R_\varpi^{(0,v]+}=r_\varpi^{(0,v]+}&\widehat\otimes_{r_\varpi^+}R_\varpi^+,\quad
R_\varpi^{[u,v]}=r_\varpi^{[u,v]}\widehat\otimes_{r_\varpi^+}R_\varpi^+,
\end{align*}
where the $\widehat\otimes$ is the tensor product completed for
the $p$-adic topology.

\begin{remark}\label{LIFT0}
If $\Lambda$ is a topological ring, and $X=\N,\Z$, let $\ell_0(X,\Lambda)$ denote
the space of sequences $(x_n)_{n\in X}$ of elements of $\Lambda$ such that
$x_n\to 0$ when $n\to\infty$.

It is plain from the definitions that:

\quad $\bullet$ $(x_n)_{n\in\N}\mapsto\sum_{n\in\N}x_nX_0^{-n}$ induces a surjection
$\ell_0(\N,R_\varpi^+)\to R_\varpi$,

\quad $\bullet$ $(x_n)_{n\in\N}\mapsto\sum_{n\in\N}x_n p^{\lceil nv/e\rceil}X_0^{-n}$ induces a surjection
$\ell_0(\N,R_\varpi^+)\to R_\varpi^{(0,v]+}$,

\quad $\bullet$ $(x_n)_{n\in\Z}\mapsto\sum_{n\geq 0}x_n p^{\lceil nv/e\rceil}X_0^{-n}+
\sum_{n<0}x_n p^{\lceil nu/e\rceil}X_0^{-n}$ induces a surjection
$\ell_0(\Z,R_\varpi^+)\to R_\varpi^{[u,v]}$.

If $N\geq 1$ and $u\leq v < \frac{e}{N}$, the same is true with
$r_\varpi^+$ replaced by $R_\varpi^+[[\frac{p}{X_0^N}]]$.  

Now, let $\lambda:R_\varpi^+\to R_\varpi^+[[\frac{p}{X_0^N}]]$ be a continuous morphism
such that $\frac{\lambda(X_0)}{X_0^A}$ is a unit in $R_\varpi^+[[\frac{p}{X_0^N}]]$, with
$A\geq 1$. The above
descriptions
of $R_\varpi^{\rm deco}$ allow to extend $\lambda$ by continuity
to morphisms $R_\varpi\to R_\varpi$ and $R_\varpi^{(0,v]+}\to R_\varpi^{(0,v/A]+}$,
$R_\varpi^{[u,v]}\to R_\varpi^{[u/A,v/A]}$, 
if $u\leq v < \frac{e}{N}$.
\end{remark}

\subsubsection{Filtration}
We filter all the above rings $S$ by order of vanishing at $X_0=\varpi$.

$\bullet$ If $P$ is invertible in $ S[\frac{1}{p}]$ (this is the case if
$ S=R_\varpi^{(0,v]+}$ or $ S=R_\varpi^{(0,v]}$ and
$v<1$, or if $ S=R_\varpi^{[u,v]}$ and $1\notin[u,v]$, or if $ S=R_\varpi$),
 put the trivial
filtration on $ S$ (i.e. ${F}^r S= S$ for all $r$).

$\bullet$ If $P$ is not invertible in $ S[\frac{1}{p}]$
(i.e.~in all the cases that are not listed above), we
have a natural embedding $ S\to R[\frac{1}{p}][[P]]=R[\frac{1}{p}][[X_0-\varpi]]$ by completing
$ S[\frac{1}{p}]$ for the $P$-adic topology
(the completion of $\O_F[X_0,\frac{1}{p}]$ for the $P$-adic topology
is the complete discrete valuation ring $K[[P]]$, and $P$ or $X_0-\varpi$
are uniformizers). We use this embedding to endow $ S$
with the natural filtration of $R[[P]]$ (by powers of $P$ or, equivalently, $X_0-\varpi$).

 \begin{remark} \label{0wiesia} 
 An element $f\in r_\varpi^{\rm PD}$ (resp.~$f\in r_\varpi^{[u]}$)
can be written (uniquely)
 in the form $f=f^++f^-$ with $f^+\in F^rr_\varpi^{\rm PD}$ 
and $f^-$ of degree $\leq re-1$ (this implies
$f^-\in \frac{1}{(r-1)!}\O_F[X_0]$, resp.~$f^-\in \frac{1}{p^{[ru]}}\O_F[X_0]$).

It follows that we can write any $f\in R_\varpi^{\rm PD}$ as $f_1+f_2$ with
$f_1\in F^rR_\varpi^{\rm PD}$ and $f_2\in \frac{1}{(r-1)!}R_\varpi^+$
and we have the same statement for $f\in R_\varpi^{[u]}$, with $f_1\in F^rR_\varpi^{[u]}$
and $f_2\in \frac{1}{p^{[ru]}}\R_\varpi^+$.
 \end{remark}

\subsubsection{Frobenius}
A Frobenius $\varphi$ on $R_\varpi$ is a ring homomorphism lifting
$x\mapsto x^p$ on $R_\varpi/p$ (in particular, $\varphi$ restricts to
the absolute Frobenius on $\O_F$).
We say that $\varphi$ is {\it admissible} if:
\begin{itemize}
\item  $\varphi(r_\varpi)\subset r_\varpi$,
\item 
 there exist $v_0>1$ such that
$\varphi(R_\varpi^{(0,v_0]+})\subset R_\varpi^{(0,v_0/p]+}$,
and $v^{(0,v_0/p]}\big(\frac{\varphi(X_j)}{X_j^p}-1\big)>\inf(\frac{1}{p-1},\tfrac{1}{2})$, if
$0\leq j\leq d$.
\end{itemize}
(The second condition ensures that $\varphi(x)-x^p$ has divided powers if $p\neq 2$. If $p=2$,
it ensures that it has higher divided powers: $\frac{(\varphi(x)-x^p)^k}{[k/2]!}$ is defined.)
If $\varphi$ is admissible then, for all $v\leq v_0$ ($v_0$ as above),
we have $\varphi(R_\varpi^{(0,v]+})\subset R_\varpi^{(0,v/p]+}$,
and $\varphi$ extends by continuity to
$R_\varpi^{[u,v]}$, for all $0<u\leq v\leq v_0$.
Moreover, the above minoration is valid with $v^{(0,v_0/p]}$ replaced by
$v^{(0,v/p]}$ or $v^{[u/p,v/p]}$, accordingly.

\medskip
We are going to use two admissible Frobenii: $\varphi_{\rm Kum}$ and
$\varphi_{\rm cycl}$.  
The two maps $\varphi_{\rm Kum}$ and $\varphi_{\rm cycl}$ are very similar
(they differ only on the arithmetic variable).
The advantage of $\varphi_{\rm Kum}$ is that it is already defined on $R_\varpi^+$
(which is not the case of $\varphi_{\rm cycl}$) and can be used
to compute crystalline and syntomic cohomology. The reason for considering
$\varphi_{\rm cycl}$ is that it is the Frobenius coming from the theory
of $(\varphi,\Gamma)$-modules and it is closely related to \'etale cohomology
(see~Proposition \ref{KLiu}), through the works of Herr \cite{Hr} (in dimension~$0$),
Andreatta-Iovita \cite{AI} (in the case of good reduction and without divisor at infinity), and Kedlaya-Liu \cite{KL} (for the general case).
It is not inconceivable that one could use Breuil-Kisin modules \cite{Ki} or
Caruso's $(\varphi,\tau)$-modules \cite {Ca} instead of
$(\varphi,\Gamma)$-modules (and hence $\varphi_{\rm Kum}$ instead
of $\varphi_{\rm cycl}$) to compute \'etale cohomology.

\subsubsection{The Frobenius $\varphi_{\rm Kum}$ and its left inverse $\psi_{\rm Kum}$}
We define $\varphi_{\rm Kum}$
on $R^+_{\varpi,\Box}$ by $\varphi_{\rm Kum}(X_i)=X_i^p$
if $0\leq i\leq d$.
As $R^+_{\varpi}$ is \'etale over $R^+_{\varpi,\Box}$, $\varphi_{\rm Kum}$
extends, uniquely, to $R^+_\varpi$ (use Remark~\ref{extr6.1} with $\Lambda_1=R^+_{\varpi,\Box}$,
$\Lambda'_1=\Lambda_2=R_\varpi^+$, $\lambda=\varphi_{\rm Kum}$, $I=(p)$ and $Z_\lambda=Z^p$).
Finally, if $ S=R_{\varpi,\Box},R_\varpi$,
$\varphi_{\rm Kum}$ extends, by continuity, to ring endomorphisms of
$ S$, $ S^{\rm PD}$, $ S^{[u]}$ and to ring morphisms
$ S^{[u,v]}\to  S^{[u,v/p]}$.

\medskip
If $0\leq j\leq d$, set
$$\partial_{{\rm Kum},j}=X_j\tfrac{\partial}{\partial X_j}.$$
These are well defined on $R_{\varpi,\Box}^+$, hence also on $R_\varpi^+$ by \'etaleness.
We extend them by the obvious rule on $R_\varpi^+[X_0^{-1}]$, and by linearity and continuity
to $R_\varpi^{\rm deco}$, for ${\rm deco}\in\{\ ,{\rm PD}, [u], (0,v]+, [u,v]\}$.
The resulting operators commute pairwise.

If $\alpha=(\alpha_0,\dots,\alpha_d)\in\{0,\dots,p-1\}^{[0,d]}$, set
$$u_{{\rm Kum},\alpha}=X_0^{\alpha_0}\cdots X_d^{\alpha_d}.$$
We have 
$$\partial_{{\rm Kum},j}u_{{\rm Kum},\alpha}=\alpha_j\,u_{{\rm Kum},\alpha}
\quad{\rm and}\quad
\varphi^{\rm cycl}(u_{{\rm Kum},\alpha})=u_{{\rm Kum},\alpha}^p.$$

\begin{lemma}\label{ku1}
{\rm (i)} Any $x\in R_\varpi/p$ can be written uniquely as $x=\sum_{\alpha}c_{{\rm Kum},\alpha}(x)$, with
$\partial_{{\rm Kum},j}c_{{\rm Kum},\alpha}(x)=\alpha_jc_{{\rm Kum},\alpha}(x)$, if $0\leq j\leq d$.

{\rm (ii)} There exists a unique $x_\alpha\in R_\varpi/p$ such that $c_{{\rm Kum},\alpha}(x)=x_\alpha^pu_{{\rm Kum},\alpha}$.

{\rm (iii)} If $x\in R_\varpi^+/p$, then $c_{{\rm Kum},\alpha}(x)\in R_\varpi^+/p$, $x_\alpha\in X_0^{-h-1}R_\varpi^+/p$,
and $x_\alpha\in R_\varpi^+/p$, if $\alpha_i=0$ for $i\neq 0$;
in particular, $x_0\in R_\varpi^+/p$.
\end{lemma}
\begin{proof}
Let $S=R_\varpi/p$, $S^+=R_\varpi^+/p$, and $\partial_j=\partial_{{\rm Kum},j}$, for $0\leq j\leq d$.

If $0\leq j\leq d$, then $\partial_j(\partial_j-1)\cdots(\partial_j-(p-1))$ is identically $0$
on $R_{\varpi,\Box}/p$, hence also on $S$ by \'etaleness.  It follows that $\partial_j$ is diagonalizable
for all $j$ and, since these operators commute pairwise, that we can decompose $S$ and $S^+$
into the direct sum of common eigenspaces.  This proves (i) as well as the the fact that
$c_{{\rm Kum},\alpha}(x)\in S^+$ if $x\in S^+$.

Now, $X_0,\dots,X_d$ is a basis of the module of differentials of
$R_{\varpi,\Box}/p$, hence also of $S$, since $S$ is obtained as the
completion of an \'etale algebra over $R_{\varpi,\Box}/p$.  It follows from~\cite{tyc}
that $X_0,\dots,X_d$ is a $p$-basis of $S$ which can be rephrased by saying that
any element $x$ of $S$ can be written uniquely as
$x=\sum_\alpha x_\alpha^p u_{{\rm Kum},\alpha}$.
Since $\partial_j(x_\alpha^p u_{{\rm Kum},\alpha})=\alpha_j x_\alpha^p u_{{\rm Kum},\alpha}$,
this proves (ii).

Finally, 
$(X_0\cdots X_dx_\alpha)^p=X_0^{p-\alpha_0}\cdots X_d^{p-\alpha_d}c_{{\rm Kum},\alpha}(x)\in S^+$,
hence $X_0\cdots X_dx_\alpha\in S^+$
(because $S^+$ is integrally closed),
 which implies
$X_0^{h+1}X_{a+b+1}\cdots X_d x_\alpha\in S^+$.
But $x_\alpha\in S=S^+[X_0^{-1}]$ and $X_0$ is not a zero divisor
in $S^+/(X_{a+b+k})$, if $1\leq k\leq d-a-b$; hence $x_\alpha\in X_0^{-h-1}S^+$, as wanted. 

In the case $\alpha_i=0$ if $i\neq 0$, we get $(X_0x_\alpha)^p=X_0^{p-\alpha_0}c_{{\rm Kum},\alpha}(x)\in S^+$,
hence $X_0x_\alpha\in S^+$.  But $(X_0x_\alpha)^p$ is actually divisible par $X_0$,
since $\alpha_0\leq p-1$,
hence $X_0x_\alpha\in X_0S^+$ and $x_\alpha\in S^+$.
\end{proof}

\begin{corollary}\label{ku2}
{\rm (i)} Any $x\in R_\varpi$ can be written uniquely as $\sum_\alpha\varphi_{\rm Kum}(x_\alpha)u_{{\rm Kum},\alpha}$,
with $x_\alpha\in R_\varpi$.

{\rm (ii)} 
If $x\in R_\varpi^+$, then $x_0\in R_\varpi^+$ and, if 
$c_{{\rm Kum},\alpha}(x)=\varphi_{\rm Kum}(x_\alpha)u_{{\rm Kum},\alpha}$, then
$c_{{\rm Kum},\alpha}(x)\in R_\varpi^+$ for all $\alpha$ and 
$$\partial_{{\rm Kum},j} c_{{\rm Kum},\alpha}(x)-\alpha_jc_{{\rm Kum},\alpha}(x)\in pR_\varpi^+,
\quad{\text{ if $0\leq j\leq d$.}}$$
\end{corollary}

We define $\psi_{\rm Kum}$ on $R_\varpi$ and $R_\varpi^+$ by the formula:
$$\psi_{\rm Kum}(x)= \varphi_{\rm Kum}^{-1}(c_{{\rm Kum},0}(x)).$$
A more conceptual definition of $\psi_{\rm Kum}$ is as follows:
$R_\varpi^+[X_0^{-1}]$ is an
extension of degree $p^{d+1}$ of $\varphi_{\rm Kum}(R_\varpi^+[X_0^{-1}])$,
with basis the $u_{{\rm Kum},\alpha}$'s. and since
$\varphi_{\rm Kum}(u_{{\rm Kum},\alpha})=u_{{\rm Kum},\alpha}^p$,
we have
$$\psi_{\rm Kum}(x)=p^{-(d+1)}\varphi_{\rm Kum}^{-1}\big({\rm Tr}_{ S/\varphi_{\rm Kum}( S)}(x)\big).$$
Note that $\psi_{\rm Kum}$ is not a ring morphism; it is a left inverse to $\varphi_{\rm Kum}$
and, more generally, we have $\psi_{\rm Kum}(\varphi_{\rm Kum}(x)y)=x\psi_{\rm Kum}(y)$.
We have
$$
\partial_{{\rm Kum},i}\circ\varphi_{\rm Kum}=p\,\varphi_{\rm Kum}\circ \partial_{{\rm Kum},i}\ \ 
\partial_{{\rm Kum},i}\circ\psi_{\rm Kum}=p^{-1}\,\psi_{\rm Kum}\circ \partial_{{\rm Kum},i}
\quad{\text{if $i=0,1,\dots,d$}}$$

The above explicit formula for
$\psi_{\rm Kum}$ extends, by continuity, to 
$ R_\varpi^{(0,v]+}$, $ R_\varpi^{(0,v]}$ and to maps
$ R_\varpi^{[u]}\to  R_\varpi^{[pu]}$,
$ R_\varpi^{[u,v]}\to  R_\varpi^{[pu,pv]}$ (surjective in all cases since $\psi_{\rm Kum}\circ\varphi_{\rm Kum}={\rm id}$).
The maps $x\mapsto c_{{\rm Kum},\alpha}(x)$ also extend and lead
to decompositions
$S=\oplus_\alpha S_\alpha$, for 
$S=R_{\varpi}^{\rm deco}$, with ${\rm deco}\in\{\ ,+,[u],(0,v]+,[u,v]\}$.
Since $\psi_{\rm Kum}(x)=\varphi_{\rm Kum}^{-1}(c_{{\rm Kum},0}(x))$, we have
$$S^{\psi_{\rm Kum}=0}=\oplus_{\alpha\neq 0}S_\alpha.$$

\begin{lemma}\label{prepsi}
If $S=R_{\varpi}^{\rm deco}$, and ${\rm deco}\in\{\ ,+,[u],(0,v]+,[u,v]\}$,
then $\partial_{{\rm Kum},j}=\alpha_j$ on
$S_\alpha/p S_\alpha^{\rm deco}$.
\end{lemma}
\begin{proof}
If ${\rm deco}\in\{\ ,+\}$, this is part of Corollary~\ref{ku2}.
If ${\rm deco}\in\{[u],(0,v]+,[u,v]\}$, elements of $S^{\rm deco}$ are those
of the form $\sum_{k\in\Z}p^{r_k}X_0^ka_k$, where $a_k\in S^+$ goes to $0$ when
$k\to\infty$ and $r_k$ is determined
by ``deco''.  Now, if $x=\sum_{k\in\Z}p^{r_k}X_0^ka_k$, then
$c_{{\rm Kum},\alpha}(x)=\sum_{k\in\Z}p^{r_k}X_0^kc_{{\rm Kum},(\alpha_0-k,\alpha_1,\dots,\alpha_d)}(a_k)$, 
where $\alpha_0-k$ is to be understood
as its representative modulo $p$ between $0$ and $p-1$.
So, if $x\in S_\alpha$, we can assume that
$a_k\in S^+_{(\alpha_0-k,\alpha_1,\dots,\alpha_d)}$, for all $k$. 
A direct computation,
using the result for $S^+$, shows then that $\partial_{{\rm Kum},j}(X_0^ka_k)-\alpha_j X_0^ka_k\in pS^+$,
which allows to conclude. 
\end{proof}

\subsection{Cyclotomic theory}
Suppose that $K$ contains enough roots of unity throughout this subsection.
\subsubsection{The cyclotomic variable~$T$}
Let 
$$i=i(K), \quad e_0=[F_i:F]=(p-1)p^{i(K)-1},
\quad f=[K:F_i]=\tfrac{e}{e_0}, \quad {\rm and}\quad \zeta=\zeta_{p^i}.$$
 As $\zeta-1$ is a uniformiser of $F_i$,
we can find polynomials $A_0,\dots,A_{f-1}\in\O_F[T]$
with $v_T(A_i)\geq 1$ and $v_T(A_0)=1$, such that, if
$$Q(X_0,T)=X_0^f+A_{f-1}(T)X_0^{f-1}+\cdots+A_0(T),$$
then $Q(X_0,\zeta-1)$ is the minimal polynomial
of $\varpi$ over $F_i$ (of course the $A_i$ are not uniquely determined by these requirements).

Let $T\in r_\varpi^+$ be the solution of $Q(X_0,T)=0$.
Since $A_0=\alpha T+\cdots$, with $\alpha$ a unit in $\O_F^*$, and $v_T(A_i)\geq 1$
if $i\geq 1$, one gets (by successive approximations) that $T=-\alpha^{-1}X_0^f+\alpha_{f+1}X_0^{f+1}+\cdots
\in r_\varpi^+$.  In particular, {\it $X_0^{-f}T$ is a unit in $r_\varpi^+$}.
We have
$$r_\varpi^+=\O_F[[T,X_0]]/Q(X_0,T)=\oplus_{i=0}^{f-1}\O_F[[T]] X_0^i,$$
and $r_\varpi$ is an \'etale extension of $\O_F[[T]]\{T^{-1}\}$. This is a consequence
of the following, more precise, lemma.
\begin{lemma}\label{new1}
Let $Q'(X_0,T)=\frac{\partial Q}{\partial X_0}(X_0,T)$.

{\rm (i)} There exists $U\in\O_F[X_0]$ which is a unit in $\O_F[[X_0]]$,
$V,W\in\O_F[X_0,T]$, such that $Q'(X_0,T)=X_0^{\delta_K}U+pV+Q(X_0,T)W$.

{\rm (ii)} The image
$\Delta$ of $Q'(X_0,T)$ in $r_\varpi^+$ is invertible in $r_\varpi$
and 
$X_0^{-\delta_K}\Delta$ is a unit in
$\O_F[[X_0,\frac{p}{X_0^{\delta_K}}]]$.
\end{lemma}
\begin{proof}
We have $v_p({\goth d}_{K/F_i})=v_p(Q'(\varpi,\zeta-1))$.
  Now, since
$v_p({\goth d}_{K/F_i})<1$ by assumption,
we have 
$$\delta_K=e\,v_p(Q'(\varpi,\zeta-1))
=\inf_{(k,p)=1}(fv_T(\overline A_k)+(k-1))=
v_{X_0}\big(\overline Q'(X_0,T)\big),$$
where $\overline{\phantom{A}}$ is the reduction modulo~$p$.
Hence $Q'(X_0,T)$ is an element of $r_\varpi^+$
congruent to $\alpha X_0^{\delta_K}$
mod~$(p,X_0^{\delta_K+1}, Q)$, where $\alpha$ is a unit in $\O_F$.
The lemma follows.
\end{proof}

\subsubsection{Filtration}
The morphism $r_\varpi^+\to \O_K$ sends $T$ to $\zeta-1$. Hence
the subring $\O_F[[T]]$ of $r_\varpi^+$ can be thought of as $r_{\zeta-1}^+$
and $\O_F[[T]]\{T^{-1}\}$ as $r_{\zeta-1}$.
In particular, the filtrations on $r_\varpi^+$ and its siblings $r_\varpi^{\rm PD}$, etc.
are also
given by the
order of vanishing at $T=\zeta-1$, which means
that {\it we can use $$P_0=\frac{(1+T)^{p^i}-1}{(1+T)^{p^{i-1}}-1}$$ instead of $P$
to define the filtrations.}

\medskip
  Set $t=p^i\log(1+T)$.
\begin{lemma}\label{19saint}
If $\frac{p-1}{p}\leq u\leq \frac{v}{p}<1<v$ and\footnote{This is an
additional condition only for $p=2$.}
 $\frac{1}{p}<u$, then:

{\rm (i)} $t$ belongs to $p\,r_\varpi^{[u,v]}$ and to $p\,r_\varpi^{[u,v/p]}$,

{\rm (ii)} $x\mapsto t^rx$ induces 
a $p^{r}$-isomorphism $r_\varpi^{[u,v]}\simeq F^rr_\varpi^{[u,v]}$
and a $p^{2r}$-isomorphism
$r_\varpi^{[u,v/p]}\simeq r_\varpi^{[u,v/p]}$.
\end{lemma}
\begin{proof}

For (i), we start from $t=\sum_{k\geq 1}\frac{(-1)^{k-1}p^i}{k}T^k$,
so we have to estimate $\inf_{s\in[u/e,v/e]}\inf_{k\geq 1}(i+ks-v_p(k))$.
The minimum is reached for $s=u/e$ and $k=p^i$, and is $\frac{pu}{p-1}\geq 1$,
which means that each of the terms of the series belongs to
$p\,r_\varpi^{[u,v]}$ (and, a fortiori, to $p\,r_\varpi^{[u,v/p]}$).

For (ii), we first check that
$f=\big(\frac{(1+T)^{p^i}-1}{(1+T)^{p^{i-1}}-1}\big)^{-1}t$ does not vanish
on $\frac{u}{e}\leq v_p(T)\leq\frac{v}{e}$ and
$t$ does not vanish on $\frac{u}{e}\leq v_p(T)\leq\frac{v}{pe}$
(this is ensured by the inequalities put on $u$ and $v$, as the valuations
of zeroes of $t$ are of the form $p^k\frac{1}{e}$ for $k\in\Z$, $k\leq i-1$, or $+\infty$).
This implies that $f$ is a unit in $r_\varpi^{[u,v]}[\frac{1}{p}]$
and $t$ is a unit in $r_\varpi^{[u,v/p]}[\frac{1}{p}]$.
To conclude we need to show that $f^{-1}\in p^{-1}r_\varpi^{[u,v]}$
and $t^{-1}\in p^{-2}r_\varpi^{[u,v/p]}$.
Equivalently, we need to check that, if $s\in[u/e,v/e]$
(resp.~$s\in[u/e,v/pe]$), $\inf_{v_p(T)=s}v_p(f(T)^{-1})\geq -1$
(resp.~$\inf_{v_p(T)=s}v_p(t^{-1})\geq -2$).
By additivity of the valuation on a circle, this amounts to checking
that $\inf_{v_p(T)=s}v_p(f(T))\leq 1$
(resp.~$\inf_{v_p(T)=s}v_p(t)\geq 2$). 

$\bullet$ For $t$ on $s\in[u/e,v/pe]$, we have to estimate
$\max_{s\in[u/e,v/pe]}\inf_{k\geq 1}(i+ks-v_p(k))$.
The maximum is reached for $s=\frac{v}{pe}$ and, 
taking $k=p^i$, we see that it is $\leq \frac{p^iv}{pe}<\frac{p}{p-1}\leq 2$.

$\bullet$ For $f$ on $s\in[u/e,v/e]$, we use the formula
$f=((1+T)^{p^{i-1}}-1)\prod_{n\geq i+1}\frac{(1+T)^{p^{n}}-1}{p((1+T)^{p^{n-1}}-1)}$.
Again the maximum of $\inf_{v_p(T)=s}v_p(f(T))$ is obtained for $s=\frac{v}{pe}$
and, since $1>\frac{v}{p}>\frac{1}{p}$, we have
$$\inf_{v_p(T)=s}v_p\big(\frac{(1+T)^{p^{n}}-1}{p((1+T)^{p^{n-1}}-1)}\big)=0,\quad{\text{if $n\geq i+1$}}$$
(the constant term, i.e.~$1$, has the minimum valuation because
$\frac{v}{p}>\frac{1}{p}$),
and 
(because $\frac{v}{p}<1$)
$$\inf_{v_p(T)=s}v_p\big((1+T)^{p^{i-1}}-1\big)\leq p^{i-1}s\leq \frac{1}{p-1}\leq 1.$$
This allows to conclude.
\end{proof}

\subsubsection{Extension of morphisms}

Let $R^+_{\zeta-1,\Box}=\O_F[[T]]\{X_1,\dots,X_d\}$. We have
$$R_{\varpi,\Box}^+=
R^+_{\zeta-1,\Box}\{X_0,X_{d+1},X_{d+2}\}/\big(Q(X_0,T),X_{d+1}X_{a+1}\cdots X_{a+b}-X_0^h,
X_{d+2}X_1\cdots X_a-1\big).$$

Let ${\mathcal M}$ be the monoid generated by $X_0,X_1,\dots,X_{d+2}$,
and let\footnote{One could consider setting
$\delta_R=\lceil\frac{\delta_K}{f}\rceil$ if $b=0$
and $\delta_R=\lceil\frac{\delta_K}{f}\rceil+1$ if $b\neq 0$ in order to differentiate
between good and semi-stable reduction.  Such a distinction would only be useful when $K$
is unramified, so we will not bother.}
$\delta_R=\lceil\frac{\delta_K}{f}\rceil$, so that
$$f\delta_R-\delta_K\geq 0\quad{\rm and}\quad
\delta_R+1<\tfrac{e_0}{2p}=\tfrac{1}{2}(p-1)p^{i-2},\ {\text{since $K$ has enough roots of unity.}}$$
\begin{proposition}\label{LIFT}
Let:

$\bullet$ $\lambda:R_{\zeta-1,\Box}^+\to\Lambda$, a continuous ring morphism,
such that $\lambda(T)^{2\delta_R}$ divides $p$ and 
$\Lambda$ is separated and complete for the $\big(\frac{p}{\lambda(T)^{2\delta_R}}\big)$-adic topology.

$\bullet$ $\beta:{\mathcal M}\to \Lambda$, a morphism of monoids, with $\beta(X_i)=u_i\lambda(X_i)$,
if $1\leq i\leq d$, $\beta(X_0)^f=\lambda(T)u$,
where $u_i,u\in\Lambda^*$,

$\bullet$ $\mu\in r_{\zeta-1}^+$ such that $\frac{\lambda(\mu)}{\lambda(T)^{2\delta_R}}\in\Lambda$
and is topologically nilpotent in $\Lambda$.

\noindent Suppose that:

$\bullet$  $\lambda(\mu)$ divides $Q(\beta(X_0),\lambda(T))$,

$\bullet$ There exists an ideal $J$ of $\Lambda$, containing 
$\frac{\lambda(\mu)}{\lambda(T)^{1+\delta_R}}$, such that $\Lambda$ is separated and complete
for the $J$-adic topology, and such that
$\lambda:R_{\zeta-1,\Box}^+\to\Lambda/J$ extends to $\lambda_\varpi:R_\varpi^+\to
\Lambda/J$,
with $\lambda_\varpi=\beta$ on ${\mathcal M}$.

Then $\lambda$ admits a unique extension to $R_\varpi^+$, lifting $\lambda_\varpi$,
with $\lambda(X_0)\in \beta(X_0)+\frac{\lambda(\mu)}{\lambda(T)^{\delta_R}}\Lambda$.
\end{proposition}
\begin{proof}
We will proceed in three steps, extending (uniquely) $\lambda$ first from $r_{\zeta-1}^+$ to
$r_\varpi^+$ (this boils down to defining~$\lambda(X_0)$), then from
$r_\varpi^+\{X_1,\dots,X_d\}$ to $R_{\varpi,\Box}^+$, and finally from $R_{\varpi,\Box}^+$
to $R_\varpi^+$.

For the first step, we use Proposition ~\ref{extr6}, with:

$\bullet$ $\Lambda_1=r_{\zeta-1}^+$, $\Lambda'_1=r_\varpi^+$ (and $J=Q'(X_0,T)$), $\Lambda_2=\Lambda$,

$\bullet$ $z=\lambda(T)^{\delta_R}$, $Z_\lambda=\beta(X_0)$, 
$H_\lambda=J^\lambda(Z_\lambda)^{-1}=
Q'(\beta(X_0),\lambda(T))^{-1}$, and $I=(\lambda(\mu))$.

(We have $Q^\lambda(Z_\lambda)=Q(\beta(X_0),\lambda(T))\in I$ by assumption,
and $H_\lambda J^\lambda(Z_\lambda)-1=0$, so the requirements of that proposition
are satisfied provided we show that $H_\lambda\in z^{-1}\Lambda$. But
$H_\lambda=Q'(\beta(X_0),\lambda(T))^{-1}$ and there exists a unit $U$ of $\O_F[[X_0]]$ and
$V,W\in\O_F[[X_0,T]]$ such that $Q'(X_0,T)=X_0^{\delta_K}U+pV+Q(X_0,T)W$ (Lemma~\ref{new1}).
  Hence, setting $\alpha= \beta(X_0)^{f\delta_R-\delta_K}u^{\delta_R}\in\Lambda$, one can write
$zH_\lambda$ as 
$$zH_\lambda=\alpha\Big(U(\beta(X_0))+
\alpha
\big(\tfrac{p}{\lambda(T)^{\delta_R}}V(\beta(X_0),\lambda(T))+
\tfrac{Q(\beta(X_0),\lambda(T))}{\lambda(T)^{\delta_R}}
W(\beta(X_0),\lambda(T))\big)\Big)^{-1},$$
and the expression in the big parenthesis is indeed a unit in $\Lambda$ since
$U(\beta(X_0))$ is, and $\tfrac{Q(\beta(X_0),\lambda(T))}{\lambda(T)^{\delta_R}}$ 
and $\tfrac{p}{\lambda(T)^{\delta_R}}$ are 
divisible by $\frac{\lambda(\mu)}{\lambda(T)^{2\delta_R}}$ and $\tfrac{p}{\lambda(T)^{2\delta_R}}$
which are assumed to be topologically
nilpotent in $\Lambda$.) In particular,
$\lambda(X_0)\in \beta(X_0)+\frac{\lambda(\mu)}{\lambda(T)^{\delta_R}}\Lambda$.

For the second step, we must (and can) extend $\lambda$ by setting:
$$\lambda(X_{d+1})=\big(\tfrac{\lambda(X_0)}{\beta(X_0)}\big)^h
(u_{a+1}\cdots u_{a+b})^{-1}\beta(X_{d+1}),
\quad
\lambda(X_{d+2})=(u_1\cdots u_a)^{-1}\beta(X_{d+2}).$$
(Note that $\beta(X_0)$ divides $\lambda(T)$ by assumption, hence
$\tfrac{\lambda(X_0)}{\beta(X_0)}\in 1+\frac{\lambda(\mu)}{\lambda(T)^{1+\delta_R}}\Lambda\subset\Lambda$ makes sense.)

For the third step, we use Remark~\ref{extr6.1} with
$\Lambda_1=R_{\varpi,\Box}^+$, $\Lambda'_1=R_\varpi^+$, $\Lambda_2=\Lambda$,
$Z_\lambda=(Z_{\lambda,1},\dots,Z_{\lambda,t})$ where $Z_{\lambda,i}$ is an arbitrary
lift of $\lambda_\varpi(Z_i)$, and $I=J$.

\end{proof}

\subsubsection{The Frobenius $\varphi_{\rm cycl}$}
We endow $R^+_{\zeta-1,\Box}$ with the Frobenius $\varphi_{\rm cycl}$
sending $T$ to $(1+T)^p-1$ and $X_i$ to $X_i^p$, if $1\leq i\leq d$.
Then $\varphi_{\rm cycl}(x)-x^p\in p R^+_{\zeta-1,\Box}$, if $x\in R^+_{\zeta-1,\Box}$.
Let
 $$v_R=\tfrac{e_0}{2p\delta_R}=\tfrac{e}{2pf\delta_R}\quad{\rm and}\quad
R_\varpi^{(0,v_R)}=R_\varpi^+[[\tfrac{p}{T^{2p{\delta_R}}}]]=R_\varpi^+[[\tfrac{p}{X_0^{2pf\delta_R}}]],
\ R_\varpi^{(0,pv_R)}=R_\varpi^+[[\tfrac{p}{T^{2{\delta_R}}}]]=R_\varpi^+[[\tfrac{p}{X_0^{2f\delta_R}}]].$$
Note that $v_R>1$, by assumption.
\begin{proposition}\label{extr8}
$\varphi_{\rm cycl}$ admits a unique extension to $R_\varpi^+$ such that
$$
\varphi_{\rm cycl}(X_0)-X_0^p\in\tfrac{p}{T^{p\delta_R}}R_\varpi^{(0,v_R)}
\quad{\rm and}\quad
\varphi_{\rm cycl}(x)-x^p\in\tfrac{p}{T^{p(1+{\delta_R})}}R_\varpi^{(0,v_R)},
\ {\text{if $x\in R_\varpi^+$.}}$$
\end{proposition}
\begin{proof}
This is a direct consequence of Proposition~\ref{LIFT}, applied to
$\Lambda=R_\varpi^{(0,v_R)}$, $\lambda=\varphi_{\rm cycl}$,
$\beta(x)=x^p$, $\mu=p$ and $\lambda_\varpi(x)=x^p$, 
$J=\big(\tfrac{p}{T^{p(1+{\delta_R})}}\big)$, taking into account that
$\frac{\varphi_{\rm cycl}(T)}{T^p}=\frac{(1+T)^p-1}{T^p}$ is a unit in $\Lambda$.
(We have
$Q(X_0^p,(1+T)^p-1)\equiv Q(X_0,T)^p=0$ mod~$p$, hence $\mu$ divides
$Q(\beta(X_0),\lambda(T))$, 
which shows that the assumptions of Proposition~\ref{LIFT} are indeed fullfilled.)
\end{proof}

Now, $\frac{\varphi_{\rm cycl}(X_0)}{X_0^p}$ is a unit in $R_\varpi^{(0,v_R)}$.
Hence we can use
Remark~\ref{LIFT0} to extend $\varphi_{\rm cycl}$ to morphisms (still noted $\varphi_{\rm cycl}$)
$R_\varpi\to R_\varpi$, $R_\varpi^{(0,v]+}\to R_\varpi^{(0,v/p]+}$ and
$R_\varpi^{[u,v]}\to R_\varpi^{[u/p,v/p]}$, if $u\leq v< v_R$.
This Frobenius is admissible.

\begin{lemma}\label{patch2}
If $v<pv_R$, and if $x\in R_\varpi$ is such that $\varphi_{\rm cycl}(x)\in R_\varpi^{(0,v/p]+}$,
then $x\in R_\varpi^{(0,v]+}$.
\end{lemma}
\begin{proof}
Write $\varphi$ for $\varphi_{\rm cycl}$.
We have $R_\varpi^{(0,v]+}=\sum_{n\geq 0}\frac{p^n}{X_0^{[ne/v]}}R_\varpi^+$.
Now, if $\varphi(x)\in R_\varpi^{(0,v/p]+}$, this implies, in particular, that the
image $\overline x$ of $x$ in $R_\varpi/p$ is such that $\overline x^p\in R_\varpi^+/p$; hence 
$\overline x\in R_\varpi^+/p$.  This means that we can find $a_0\in R_\varpi^+$ such
that $x-a_0\in p R_\varpi$.  But, $\varphi(a_0)\in R_\varpi^{(0,v/p]+}$ thanks
to the assumption $v<pv_R$, which means that $\varphi(x-a_0)\in\sum_{n\geq 1}\frac{p^n}{X_0^{[ne/v]}}R_\varpi^+$.
It follows that, if we write $x=a_0+\frac{p}{X_0^{[e/v]+1}}x_1$, the image
of $\varphi(x_1)$ in $R_\varpi/p$ belongs to $X_0R_\varpi^+/p$ (since $p\big([e/v]+1\big)-[pe/v]\geq 1$),
hence the image of $x_1$ belongs to $X_0R_\varpi^+/p$, and we can find $a_1\in R_\varpi^+$
such that $x_1-X_0a_1\in pR_\varpi$.  But then
$\varphi\big(x-a_0-\frac{p}{X_0^{[e/v]}}a_1\big)\in \sum_{n\geq 2}\frac{p^n}{X_0^{[ne/v]}}R_\varpi^+$,
and we can iterate the process, writing $x=a_0+\frac{p}{X_0^{[e/v]}}a_1+\frac{p^2}{X_0^{[2e/v]+1}}x_2$,
and deducing that $x_2=X_0a_2+py$, with $a_2\in R_\varpi^+$ and $y\in R_\varpi$.
Going to the limit, we find $x=\sum_{n\in\N}a_n\frac{p^n}{X_0^{[ne/v]}}$, with $a_n\in R_\varpi^+$,
for all $n$, which concludes the proof.
\end{proof}

\subsubsection{The operator $\psi_{\rm cycl}$}
 Let
$$u_{{\rm cycl},\alpha}=(1+T)^{\alpha_0}X_1^{\alpha_1}\cdots
X_d^{\alpha_d},\quad{\text{ if $\alpha\in\{0,1,\dots,p-1\}^{[0,d]}$.}}$$
\begin{proposition}\label{patch1}
{\rm (i)}
Any $x\in R_\varpi$ can be written uniquely $x=\sum_\alpha c_{{\rm cycl},\alpha}(x)$, with
$c_{{\rm cycl},\alpha}(x)\in\varphi_{\rm cycl}(R_\varpi)u_{{\rm cycl},\alpha}$.

{\rm (ii)} If $v<v_R$, and $x\in R_\varpi^{(0,v]+}$, then
$x_{{\rm cycl},\alpha}\in X_0^{-p\delta_K} R_\varpi^{(0,v]+}$, for all $\alpha$.
\end{proposition}
Before turning to the proof of this proposition, let us draw some consequences.
We define the left inverse $\psi_{\rm cycl}$ of $\varphi_{\rm cycl}$,
 on $R_\varpi$, by the formula
$$\psi_{\rm cycl}(x)=\varphi_{\rm cycl}^{-1}(c_{{\rm cycl},0}(x)).$$
Since $\varphi_{\rm cycl}(u_{{\rm cycl},\alpha})=u_{{\rm cycl},\alpha}^p$
for all $\alpha$, we have ${\rm Tr}_{R_\varpi/\varphi_{\rm cycl}(R_\varpi)}(u_{{\rm cycl},\alpha})=0$
if $\alpha\neq 0$, and we can define $\psi_{\rm cycl}$, intrinsically,
by the formula
$$\psi_{\rm cycl}=p^{-(d+1)}\varphi_{\rm cycl}^{-1}\circ{\rm Tr}_{R_\varpi/\varphi_{\rm cycl}(R_\varpi)}.$$
\begin{proposition}\label{extr10}
If $v<pv_R$, then:

{\rm (i)}
$\psi_{\rm cycl}\big(T^{-pN}R_\varpi^{(0,v/p]+})\subset T^{-N-{\delta_R}}R_\varpi^{(0,v]+}$.

{\rm (ii)} If $\ell_R=\lceil\frac{p}{p-1}\delta_R\rceil$, then $T^{-\ell_R}R_\varpi^{(0,v]+}$
is stable by $\psi_{\rm cycl}$.

{\rm (iii)} If $v\leq p$, the natural map $\oplus_{\alpha\neq 0}\varphi_{\rm cycl}(R_\varpi^{(0,v]+})
u_{{\rm cycl},\alpha}\to (R_\varpi^{(0,v/p]+})^{\psi_{\rm cycl}=0}$ is a $p^{p+1}$-isomorphism.
\end{proposition}
\begin{proof}
(i) follows from (ii) of Proposition ~\ref{patch1}, taking into account that $\psi(\varphi(T)^{-N}x)=T^{-N}\psi(x)$,
that $\frac{\varphi(T)}{T^p}$ is a unit in $R_\varpi^{(0,v/p]+}$, and
that $X_0^{\delta_K}$ divides $T^{\delta_R}$.
 (ii) is a direct consequence of (i) and
the inclusion $R_\varpi^{(0,v]+}\subset R_\varpi^{(0,v/p]+}$.

Finally, if $x\in (R_\varpi^{(0,v/p]+})^{\psi_{\rm cycl}=0}$, we can write,
using Proposition ~\ref{patch1}, $x=\sum_{\alpha\neq 0}\varphi(a_\alpha) u_\alpha$,
with $\varphi(a_\alpha) u_\alpha\in X_0^{-p\delta_K}R_\varpi^{(0,v/p]+}$.
But $u_\alpha$ divides $X_0^{p(h+1)}\prod_{i=a+b+1}^dX_i^p$, hence
$\varphi(a_\alpha)\in X_0^{-p(h+1+\delta_K)}R_\varpi^{(0,v/p]+}$.
(See the proof of Lemma~\ref{ku1} for a similar argument.)
This implies, thanks to Lemma~\ref{patch2}, that $a_\alpha\in X_0^{-(h+1+\delta_K)}R_\varpi^{(0,v]+}$.
This proves (iii) since $X_0^{h+1+\delta_K}$ divides $p^{p+1}$
in $R_\varpi^{(0,v]+}$, as $h\leq e$, $\delta_K+1\leq\frac{e}{p}$, and $v\leq p$.
\end{proof}

We let $$\partial_{{\rm cycl},0}=(1+T)\frac{\partial}{\partial T}
\quad{\rm and}\quad
\partial_{{\rm cycl},i}=X_i\frac{\partial}{\partial X_i},\ {\text{if $1\leq i\leq d$.}}$$
Then we have
\begin{align*}
\partial_{{\rm cycl},i}\circ\partial_{{\rm cycl},j}=&\ \partial_{{\rm cycl},j}\circ\partial_{{\rm cycl},i}\quad{\rm and}\quad
\partial_{{\rm cycl},i}\circ\varphi_{\rm cycl}=p\,\varphi_{\rm cycl}\circ \partial_{{\rm cycl},i},\\
\partial_{{\rm cycl},i}\circ\psi_{\rm cycl}=&\ p^{-1}\,\psi_{\rm cycl}\circ \partial_{{\rm cycl},i},
\quad{\text{if $i,j=0,1,\dots,d$.}}
\end{align*}

\begin{lemma}
If $v<2pv_R$, then $\partial_{{\rm cycl},0}X_0\in X_0^{-\delta_K}r_\varpi^{(0,v]+}$ and
$v^{(0,v]}(\partial_0X_0)\geq -1$.
\end{lemma}
\begin{proof}
Write $\partial_0$ for $\partial_{{\rm cycl},0}$.
As $Q(X_0,T)=0$, we have $\frac{\partial Q}{\partial X_0}(X_0,T)
\partial_0X_0+\frac{\partial Q}{\partial T}(X_0,T)
\partial_0T=0$, hence
$$\partial_0X_0=-(1+T)\big(\frac{\partial Q}{\partial X_0}(X_0,T)\big)^{-1}
\frac{\partial Q}{\partial T}(X_0,T).$$
Since $\frac{\partial Q}{\partial T}(X_0,T)\in r_\varpi^+$, the result follows from Lemma~\ref{new1}.
\end{proof}

\subsubsection{Proof of Proposition~\ref{patch1}}
Write $\varphi$ for $\varphi_{\rm cycl}$ in all that follows.
We will first prove the existence of a decomposition $x=\sum_\alpha x_\alpha$ as in Proposition ~\ref{patch1},
but with $x_\alpha\in \varphi(R_\varpi) u_{{\rm Kum},\alpha}$, and then use the relationship (Lemma~\ref{ku3})
between the $u_{{\rm Kum},\alpha}$'s and the $u_{{\rm cycl},\alpha}$'s to conclude.

\smallskip

Let $I=\{a+1,\dots,a+b,d+1\}$, so that $\prod_{i\in I}X_i=X_0^h$.
If $i\in I$, let $R_{\varpi,\Box,i}^+$, $R_{\varpi,i}^+$, be the rings obtained
by localizing on the polycircle where $X_j$ is a unit if $j\in I-\{i\}$ (hence
$X_i$ is $X_0^h$ times a unit). Concretely,
if $S=R_{\varpi,\Box}^+,R_\varpi^+$, then
$S_i=S\{Y\}/(1-Y\prod_{j\in I-\{i\}}X_i)$.
Then $R_{\varpi,i}^+$ is \'etale over $R_{\varpi,\Box,i}^+$,
and $R_{\varpi,\Box,i}^+$ is smooth over $\O_F$ (the equation $\prod_{i\in I}X_i=X_0^h$
becomes $X_i=YX_0^h$, which allows to remove $X_i$ from the variables and
the associated space is just a product of closed circles of radius $1$ and unit balls).

From $R_{\varpi,i}^+$ we can construct rings $R_{\varpi,i}^{\rm deco}$
by the same process we used to construct $R_\varpi^{\rm deco}$ from $R_\varpi^+$.
Note that $\varphi$ extends naturally to $R_{\varpi,i}$, and induces
morphisms $R_{\varpi,i}^{(0,pv_R)}\to R_{\varpi,i}^{(0,v_R)}$, 
and $R_{\varpi,i}^{(0,pv]+}\to R_{\varpi,i}^{(0,v]+}$, if $v<v_R$.

\begin{lemma}\label{patch3}
If $x\in R_\varpi$, 
the following conditions are equivalent:

{\rm (i)} $x\in R_\varpi^{(0,r_R)}$,

{\rm (ii)} $x\in R_{\varpi,i}^{(0,r_R)}$, for all $i\in I$.
\end{lemma}
\begin{proof}
There is only the implication (ii)$\Rightarrow$(i) to prove.
Let us first show that the divisibility of $x$ by $X_0$ in $R_{\varpi,i}^+/p$ for all $i\in I$
implies its divisibility in $R_\varpi^+/p$.
Geometrically, this is equivalent to $\cup_{i}Z_i$ Zariski dense in
$Z$, where $Z_i={\rm Spec}(R_{\varpi,i}^+/(p,X_0))$ and $Z={\rm Spec}(R_\varpi^+/(p,X_0))$.
The same statement for $R_{\varpi,\Box}$ instead of $R_\varpi$ is immediate as
each irreducible component of $Z_\Box$ contains exactly one of the
$Z_{\Box,i}$'s (with obvious notations), as a Zariski open subset.
The statement for $R_\varpi$ follows by \'etaleness of $Z\to Z_{\Box}$
($Z_i$ is the inverse image of $Z_{\Box,i}$ in $Z$).
Since $R_\varpi/p=(R_\varpi^+/p)[X_0^{-1}]$,
one infers from the above statement that if $x\in R_\varpi/p$ is such
that $x\in R_{\varpi,i}^+/p$ for all $i\in I$, then $x\in R_\varpi^+/p$.

Choose a section $s:R_\varpi^+/p\to R_\varpi^+$ of the reduction modulo~$p$
such that $s(0)=0$ and $s(X_0x)=X_0s(x)$ so that $s$ is continuous.
If $x\in R_\varpi$ belongs to $R_{\varpi,i}^{(0,r_R)}$ for all $i$, 
then its image $a_0$ in $R_\varpi/p$ belongs to $R_{\varpi,i}^+/p$, for all $i$,
and hence $a_0\in R_\varpi^+/p$.
Now, by definition of $R_{\varpi,i}^{(0,r_R)}$, if $M=2pf\delta_R$,
the image $a_1$ of $\frac{X_0^{M}}{p}(x-s(a_0))$ in $R_\varpi/p$
belongs to $R_{\varpi,i}^+/p$ for all $i$, hence $a_1\in R_\varpi^+/p$.
One can iterate, and get that the image $a_2$ of $\frac{X_0^{2M}}{p^2}(x-s(a_0)-\frac{p}{X_0^{M}}s(a_1))$
belongs to $R_{\varpi,i}^+/p$ for all $i$, hence $a_2\in R_\varpi^+/p$.
This process gives us a sequence $(a_n)_{n\in\N}$ of elements
of $R_\varpi^+/p$, 
such that $x=\sum s(a_n)\frac{p^n}{X_0^{nM}}$.
This concludes the proof.
\end{proof}

\begin{lemma}\label{boring}
Let $S=r_\varpi, R_{\varpi,i}$, and let $(u_j)_{j\in J}$ be a family of elements
of $S^+$ such that:

\quad $\bullet$  any 
$x\in S/p$ can be written, uniqueley, $x=\sum_j c_j(x)^pu_j$, with $c_j(x)\in S$, for all $j\in J$,

\quad $\bullet$ there exists $N\leq 2f\delta_R$ such that, if $x\in S^+/p$, then $c_j(x)\in X_0^{-N}S^+/p$, for all $j$.

Then any $x\in S$ can be written, uniquely,
$x=\sum_j\varphi(c_j(x))\,u_j$, with $c_j(x)\in S$, and if
$x\in S^{(0,v_R)}$,
then 
$c_j(x)\in X_0^{-N}S^{(0,pv_R)}$ for all $j$.
\end{lemma}
\begin{proof}
The $c_j(x)$'s are obtained by the following algorithm:

$\bullet$ Choose a section $s:S^+/p\to S^+$ of the reduction $y\mapsto \overline y$ modulo~$p$, such that $s(X_0y)=X_0y$,
and extend $s$ to $S/p$ using this functional equation.

$\bullet$ 
Define $f:S^{(0,v_R)}[X_0^{-1}]\to S^{(0,v_R)}[X_0^{-1}]$ by
$f(y)=\tfrac{1}{p}\big(y-\sum_j\varphi(s(c_j(\overline y)))\,u_j\big)$ (that $f$ exists rests upon
the fact that $\varphi(S^+)\subset S^{(0,pv_R)}$ and the observation that
$f(\varphi(X_0)^{-N}y)=\varphi(X_0)^{-N}f(y)$ thanks to the functional equation
$s(X_0y)=X_0s(y)$). 

$\bullet$ Set $x_0=x$ and let $x_{n+1}=f(x_n)$, if $n\geq 0$, so that
$x=p^{n+1}x_{n+1}+\sum_j\varphi\big(\sum_{i=0}^np^is(c_j(\overline{x_i}))\big)\,u_j$.

An easy induction, using the fact that $N\leq 2f\delta_R$, shows that $x_n\in X_0^{-2pnf\delta_R}
S^{(0,v_R)}$ and $s(c_j(\overline x_n))\in X_0^{-N-2nf\delta_R}S^{(0,pv_R)}$.
Hence $c_j(x)=\sum_{i\geq 0}p^is(c_j(\overline{x_i}))\in X_0^{-N}S^{(0,pv_R)}$.
The result follows.
\end{proof}

If $i\in I$, let 
$$A_i=\{\beta=(\beta_0,\dots,\beta_{d+1})\in\{0,1,\dots,p-1\}^{[0,d+1]},\ \beta_i=0\}.$$
If $\beta\in A_i$, let $u_{i,\beta}=X_0^{\beta_0}\cdots X_{d+1}^{\beta_{d+1}}$.
Note that $u_{d+1,\beta}=u_{{\rm Kum},(\beta_0,\dots,\beta_d)}$.
\begin{lemma}\label{extr9}
Let $i\in I$.

{\rm (i)}
An element $x$ of $R_{\varpi,i}$ can be written uniquely
$x=\sum_{\beta\in A_i}x_{i,\beta}$,
with $x_{i,\beta}\in \varphi(R_{\varpi,i}) u_{i,\beta}$ 
for all~$\beta$.

{\rm (ii)}
If $x\in R_{\varpi,i}^{(0,v_R)}$, then
$x_{i,\beta}\in R_{\varpi,i}^{(0,v_R)}$ for all $\beta$.
\end{lemma}
\begin{proof}
If $S=R_{\varpi,\Box,i}^+/p$, 
the $X_j$'s, for $j\neq i$, form a basis of $\Omega^1_S$.  By \'etaleness, this is
also the case if $S=R_{\varpi,i}^+/p$.  Hence, according to~\cite{tyc}, the
$u_{i,\beta}$'s, for $\beta\in A_i$, form 
a basis of $S$ over $\varphi(S)$, if $S=R_{\varpi,i}^+/p$ and, as a consequence, if $S=R_{\varpi,i}/p$.
It follows that we can use Lemma~\ref{boring} (with $N=0$) to conclude.
\end{proof}

If $\beta\in A_{d+1}$, let $\alpha=(\beta_0,\dots,\beta_d)$ be the corresponding element
of $\{0,\dots,p-1\}^{[0,d]}$, and set $x_\alpha=x_{d+1,\beta}$, so that the 
decomposition $x=\sum_{\beta\in A_{d+1}}x_{d+1,\beta}$
becomes simply 
$x=\sum_\alpha x_\alpha$.
\begin{lemma}\label{patch4}
{\rm (i)} If $x\in R_\varpi$, then
$x_\alpha\in\varphi(R_\varpi)\,u_{{\rm Kum},\alpha}$.

{\rm (ii)}
If $v<v_R$, and if $x\in R_\varpi^{(0,v]+}$, then
$x_\alpha\in R_\varpi^{(0,v]+}$, for all $\alpha$.
\end{lemma}
\begin{proof}
(i) is a direct consequence of the analogous statement modulo~$p$ (Lemma~\ref{ku1}, we have
$\varphi(x)=x^p$ modulo~$p$).
To prove (ii),
an adaptation of Remark~\ref{LIFT0} shows that it is enough to prove the same 
statement for $R_\varpi^{(0,v_R)}$: write an element
of $S^{(0,v]+}$ as $\sum_{n\in\N}x_n\frac{p^n}{\varphi(X_0)^{[ne/v]-d_n}}$, with $x_n\in S^{(0,v_R)}$
and $d_n\to +\infty$,
and use the result for $R_\varpi^{(0,v]+}$.

According to Lemma~\ref{patch3}, it is enough to show that $x_\alpha\in R_{\varpi,i}^{(0,v_R)}$,
for all $i\in I$.  For this purpose it is enough, considering Lemma~\ref{patch4}, to show
that $x_\alpha$ is a linear combination of the $x_{i,\beta}$'s, with coefficients
in $R_{\varpi,i}^{(0,v_R)}$, and we will in fact exhibit such a combination
with coefficients in $r_\varpi^{(0,v_R)}$.

First, let us decompose elements of $r_\varpi^{(0,v_R)}$.
If $N\in\Z$, if $x\in X_0^Nr_\varpi^{(0,v_R)}$, and
if $N=pq+r$, with $0\leq r\leq p-1$, one can write $x$, modulo $X_0^{N+1}$,
as $X_0^r\varphi(X_0)^q\sum_{n}\varphi(a_n)\frac{p^n}{\varphi(X_0)^{2f\delta_R}}$,
with $a_n\in\O_F$.  Hence, there exist $x_{N+1}\in X_0^{N+1}r_\varpi^{(0,v_R)}$,
and $a_N\in r_\varpi^{(0,pv_R)}$ such that $x=X_0^r\varphi(X_0^qa_N)+x_{N+1}$.  Iterating this process,
one can write $x$ as $\sum_{pq+r\geq N, 0\leq r\leq p-1}X_0^r\varphi(X_0^qa_{pq+r})$, with
$a_{pq+r}\in r_\varpi^{(0,pv_R)}$.  Setting $c_r(x)=\sum_{pq+r\geq N}X_0^r\varphi(X_0^qa_{pq+r})$,
we obtain a decomposition 
$$x=\sum_{r=0}^{p-1}c_r(x),\quad {\text{with $c_r(x)\in X_0^r\varphi(r_\varpi)\cap X_0^Nr_\varpi^{(0,v_R)}$.}}$$
In particular, if $N\in\N$, we can write $X_0^N=\sum_{r=0}^{p-1}X_0^{N}c_{N,r}$, with $c_{N,r}\in r_\varpi^{(0,v_R)}$, 
in such a way
that $X_0^{N}c_{N,r}\in X_0^r\varphi(r_\varpi)$, for all $r$.

Next, notice that $u_{i,\beta}=X_0^{\beta_0+\beta_{d+1}h}X_1^{\beta_1}\cdots X_a^{\beta_a}
X_{a+1}^{\beta_{a+1}-\beta_{d+1}}\cdots X_{a+b}^{\beta_{a+b}-\beta_{d+1}}X_{a+b+1}^{\beta_{a+b+1}}\cdots
X_d^{\beta_d}$.
This implies that $$(x_{i,\beta})_\alpha=\begin{cases}
c_{\beta_0+\beta_{d+1}h,\alpha_0}x_{i,\beta}, 
&{\text{if $(\alpha_1,\dots,\alpha_d)=(\beta_1,\dots,\beta_a,\beta_{a+1}-\beta_{d+1},\dots,\beta_{a+b}-\beta_{d+1}, \beta_{a+b+1},\dots,\beta_d)$ modulo~$p$,}}\\
0, & {\text{otherwise.}}
\end{cases}$$
Summing over $\beta\in A_i$ gives the desired expression of $x_\alpha$ as a combination
of the $x_{i,\beta}$'s.
\end{proof}

\begin{lemma}\label{ku3}
If $k\geq0$, there exists $a_{k,i}\in T^{-\delta_R}r_\varpi^{(0,pv_R)}$ such that
$X_0^k=\sum_{i=0}^{p-1}\varphi(a_{k,i})\,(1+T)^i$.
\end{lemma}
\begin{proof}
Let us denote by ${\rm Tr}$ the trace from $r_\varpi$ to $r_{\zeta-1}$.
If $0\leq j\leq f-1$, let $f_j=X_0^j$.  The
$f_j$'s are a basis of $r_\varpi/p$ over $r_{\zeta-1}/p$;
let $(g_j)_{0\leq j\leq f-1}$ be the dual basis for the bilinear form
${\rm Tr}(xy)$.  Now, if $\Delta$ is as in Lemma~\ref{new1},
we have ${\rm Tr}(\Delta^{-1}X_0^j)=0$ if $0\leq j\leq f-2$ and
${\rm Tr}(\Delta^{-1}X_0^{f-1})=1$.  It follows that
$g_j=\Delta^{-1}R_j$, where $R_j\in k[X_0]$ is unitary, of degree~$f-1-j$;
in particular, $g_j\in X_0^{-\delta_K}r_\varpi^+/p$.

Now the $f_j^p$'s are also a basis of $r_\varpi/p$ over $r_{\zeta-1}/p$ and
the $g_j^p$'s are the dual basis.
So, if $x\in r_\varpi/p$, we have $x=\sum_{j=0}^{f-1}{\rm Tr}(g_j^px)\,f_j^p$.
But
$g_j^px\in X_0^{-p\delta_K}r_\varpi^+/p\subset X_0^{-\delta_K}T^{-p\delta_R}
r_\varpi^+/p$ (we could improve this to $X_0^{-\delta_K}T^{-\frac{p-1}{p}\delta_R}$).
It follows that if $x\in r_\varpi^+/p$, then
${\rm Tr}(g_j^px)\in T^{-p\delta_R}r_{\zeta-1}^+/p$.
Hence one can write ${\rm Tr}(g_j^px)$ as $T^{-p\delta_R}\sum_{i=0}^{p-1}
(1+T)^ia_{j,i}^p$, with $a_{j,i}\in r_{\zeta-1}^+/p$.
It follows that $x=\sum_{i=0}^{p-1}(1+T)^ic_i(x)^p$,
with 
$c_i(x)=T^{-\delta_R}\sum_{j=0}^{f-1}a_{j,i}f_j\in T^{-\delta_R}r_\varpi^+/p$.

The result follows from Lemma~\ref{boring} with $S=r_\varpi$, $N=f\delta_R$, and
taking the $(1+T)^i$'s as the $u_j$'s.
\end{proof}

From this lemma, we get the relation:
$$u_{{\rm Kum},\alpha}=\sum_{i=0}^{p-1}\varphi(a_{\alpha_0,i})\,u_{{\rm cycl},\beta_i},\quad
{\text{with $\beta_i=(i,\alpha_1,\dots,\alpha_d)$.}}$$ 
This implies that $c_{{\rm cycl},\alpha}(x)=\sum_{i=0}^{p-1}\varphi(a_{i,\alpha_0})x_{(i,\alpha_1,\dots,\alpha_d)}$,
and makes it possible to use Lemmas~\ref{patch4} and~\ref{ku3} to finish the proof of Proposition ~\ref{patch1}.

\subsection{Period rings}
let
$\overline R$ be the ``maximal extension of $R$ unramified outside $X_{a+b+1}\cdots X_{d}=0$
in characteristic~$0$ (i.e.~after inverting $p$)". Let $G_R={\rm Gal}(\overline R/R)$.  Define $v_p$
 on ${\overline R}[\frac{1}{p}]$
to be 
the spectral valuation.

\subsubsection{$p$-adic Hodge theory rings}
   If $S=K,R[\frac{1}{p}]$, denote by
$\C(\overline S)$ the completion of $\overline S$
for $v_p$, and let $\C^+(\overline S)$ be the sub-ring of
$x$'s with $v_p(x)\geq 0$.
  $\C(\overline S)$
is a perfectoid algebra.  
Denote by $\E_{\overline S}$ its tilt
and set $\A_{\overline S}=W(\E_{\overline S})$.
We can describe $\E_{\overline S}$ as the set of sequences
$(x_n)_{n\in\N}$, with $x_n\in \C(\overline S)$ and
$x_{n+1}^p=x_n$ for all $n\in\N$.
If $x\in\C(\overline S)$, denote by $x^\flat$ any element
$(x_n)_{n\in\N}$ of $\E_{\overline S}$ with $x_0=x$.

The inclusion $K\subset R[\frac{1}{p}]$ induce
inclusions
$$\overline K\subset\overline R[\tfrac{1}{p}],\quad
\E_{\overline K}\subset \E_{\overline R},\quad
\A_{\overline K}\subset \A_{\overline R}.$$
Define $\ve$ on $\E_{\overline S}$ by
$\ve(x)=v_p(x^\sharp)$, $x^{\sharp}:=x_0$.
This is a valuation on $\E_{\overline{S}}$
for which it is complete.

Let $\E^+_{\overline S}$ be the subring of
$\E_{\overline S}$ of $x$'s such that $\ve(x)\geq 0$, and let
$\A^+_{\overline S}=W(\E^+_{\overline S})$.
So $\E^+_{\overline S}$ is the tilt of
$\C^+(\overline S)$ and, as a ring, it is the projective limit (over $\N$)
of $\C^+(\overline S)/{\goth a}$ for the transition maps $x\mapsto x^p$,
where ${\goth a}$ is the ideal $\{x,\,v_p(x)\geq\frac{1}{p}\}$ (we
could have replaced $\frac{1}{p}$ by any number in $]0,1]$).
The inclusion $\O_K\subset R$ induce
inclusions
$$
\E_{\overline K}^+\subset \E_{\overline R}^+,\quad
\A_{\overline K}^+\subset \A_{\overline R}^+.$$
Let 
$$\epsilon=(1,\zeta_p,\zeta_{p^2},\dots)\in\E^+_{\overline K},
\quad\pi=[\epsilon]-1\in\A^+_{\overline K}
\quad{\rm  and}\quad
\xi=\tfrac{\pi}{\varphi^{-1}(\pi)}
.$$

Any element $x$ of $\A_{\overline S}$ (resp.~$\A^+_{\overline S}$)
can be written uniquely as
$$x=\sum_{k\in\N}p^k[x_k],$$
with $x_k\in \E_{\overline S}$ (resp.~$\E^+_{\overline S}$).
Then $\theta:\A^+_{\overline S}\to \C^+(\overline S)$
defined by $\theta(\sum_{k\in\N}p^k[x_k])=\sum_{k\in\N}p^kx_k^\sharp$
is a surjective ring homomorphism whose kernel is principal, generated by
$p-[p^\flat]$ (or by $P_\varpi([\varpi^\flat])$ or $\xi$).
We extend $\theta$, by $\Q_p$-linearity, to
$\theta:\A^+_{\overline S}[\frac{1}{p}]\to \C(\overline S)$,
and we define $\bdrp(\overline S)$ to be the completion
of $\A^+_{\overline S}[\frac{1}{p}]$ with respect to the ideal
$(p-[p^\flat])$.  We filter $\bdrp(\overline S)$
by the powers of the ideal $(p-[p^\flat])$, i.e.~we
set $F^i\bdr(\overline S)=(p-[p^\flat])^i\bdrp(\overline S)$.

  We define $\A_{\crr}(\overline S)$ as the $p$-adic completion
of $\A^+_{\overline S}[\frac{(p-[p^\flat])^k}{k!},\ k\in\N]=
\A^+_{\overline S}[\frac{[p^\flat]^k}{k!},\ k\in\N]$. 
It is naturally a subring of $\bdrp(\overline S)$.

\subsubsection{$(\varphi,\Gamma)$-modules theory rings}
If $v>0$ and $S=K,R$, let
$\A_{\overline S}^{(0,v]}$ be the subring of
$\A_{\overline S}$ defined by
$$\A_{\overline S}^{(0,v]}=\big\{\sum_{k\in\N}p^k[x_k],\ 
v\ve(x_k)+k\to +\infty {\text{ when $k\to +\infty$}}\big\},$$
and let $\A_{\overline S}^{(0,v]+}$ be the subring of
$\A_{\overline S}^{(0,v]}$ of
the $x=\sum_{k\in\N}p^k[x_k]$ such that
$v\ve(x_k)+k\geq 0$ for all $k\in\N$.
We have
$$\A_{\overline S}^{(0,v]}=\A_{\overline S}^{(0,v]+}[\tfrac{1}{[p^\flat]}].$$
If $\alpha\in \E^+_{\overline K}$ satisfies $\ve(\alpha)=\frac{1}{v}$,
then $\A_{\overline S}^{(0,v]+}$ is also the completion
of $\A^+_{\overline S}[\frac{p}{[\alpha]}]$ for the $p$-adic topology.
If $v\geq 1$, the natural
map $\A^+_{\overline S}[\frac{p}{[\alpha]}]\to \bdrp(\overline S)$
extends, by continuity, to injections
$$\A_{\overline S}^{(0,v]+}\to \bdrp(\overline S),\quad
\A_{\overline S}^{(0,v]}\to \bdrp(\overline S).$$

  If $0 < u$, and if $\beta\in \E^+_{\overline K}$ satisfies
$\ve(\beta)=\frac{1}{u}$,
we define $\A_{\overline S}^{[u]}$ as the completion
of $\A^+_{\overline S}[\frac{[\beta]}{p}]$ for the $p$-adic topology.
If $u\leq 1$, the natural
map $\A^+_{\overline S}[\frac{[\beta]}{p}]\to \bdrp(\overline S)$
extends, by continuity, to an injection
$$\A_{\overline S}^{[u]}\to \bdrp(\overline S).$$

If $0 < u\leq v$, and if $\alpha,\beta\in \E^+_{\overline K}$ satisfy 
$\ve(\alpha)=\frac{1}{v}$ and $\ve(\beta)=\frac{1}{u}$,
we define $\A_{\overline S}^{[u,v]}$ as the completion
of $\A^+_{\overline S}[\frac{p}{[\alpha]},\frac{[\beta]}{p}]$ for the $p$-adic topology.
If $u\leq 1\leq v$, the natural
map $\A^+_{\overline S}[\frac{p}{[\alpha]},\frac{[\beta]}{p}]\to \bdrp(\overline S)$
extends, by continuity, to an injection
$$\A_{\overline S}^{[u,v]}\to \bdrp(\overline S).$$
We use these embeddings into $\bdrp(\overline S)$ to induce filtrations on all the rings $\A_{\overline S}^{\rm deco}$. 
We have $$\varphi(\A_{\overline S}^{(0,v]+})=\A_{\overline S}^{(0,v/p]+},\quad
\varphi(\A_{\overline S}^{(0,v]})=\A_{\overline S}^{(0,v/p]},\quad \varphi(\A_{\overline S}^{[u]})=\A_{\overline S}^{[u/p]},\quad
\varphi(\A_{\overline S}^{[u,v]})=\A_{\overline S}^{[u/p,v/p]}.$$

The relative crystalline ring of periods $\A_{\crr}(\overline{R})=\A^+_{\overline{R}}[\frac{[{p}^{\flat}]^k}{k!},k\in\N]^{\wedge}$ is related to the above rings. We have  
\begin{align*}
\A_{\crr}(\overline{R}) & \subset \A^+_{\overline{R}}[\tfrac{[\beta]}{p}]^{\wedge}, \mbox{ for } \quad 0 < v_{\E}(\beta)\leq p-1;\\
\A_{\crr}(\overline{R}) & \supset \A^+_{\overline{R}}[\tfrac{[\beta]}{p}]^{\wedge}, \mbox{ for } \quad v_{\E}(\beta)\geq p.
\end{align*}
(If $p-1<\ve(\beta)<p$, there exists $C(v)$, $v=\ve(\beta)$, such that $p^{C(v)}\A^+_{\overline{R}}[\frac{[\beta]}{p}]^{\wedge}\subset\A_{\crr}(\overline{R})$.)

To see this, write  $[\beta]=[{p}^{\flat}]^{v}u$, for a unit $u\in \A_{\overline K}^+$, $v=v_{\E}(\beta)$. Since $v_p(k!)=(k-s_k)/(p-1)$, where $s_k\geq 1$
is the sum of digits in the $p$-adic presentation of $k$, we have 
$[{p}^{\flat}]^{[k]}=[{p}^{\flat}]^kp^{-(k-s_k)/(p-1)}u_0$, for a unit $u_0\in \so_F$. For the first inclusion above, it suffices to show that, for all $k$, 
$[{p}^{\flat}]^kp^{-(k-s_k)/(p-1)}\in \A^+_{\overline{R}}[\tfrac{[\beta]}{p}]^{\wedge}$. But 
$([\beta]/p)^{(k-s_k)/(p-1)}=[{p}^{\flat}]^{v(p-1)^{-1}(k-s_k)}p^{-(k-s_k)/(p-1)}u_1$, for a unit $u_1\in\A_{\overline K}^+ $. It follows that, if $0 < v\leq (p-1)$, then $[{p}^{\flat}]^{[k]}\in (([\beta]/p)^{(k-s_k)/(p-1)})$, as wanted. 

 For the second inclusion,  if  $ v\geq p$, then $([\beta]/p)^{k}=[{p}^{\flat}]^{pk}p^{-k}a_2,$ $a_2\in \A_{\crr}(\overline{R}) $. It suffices to show that $(pk)!p^{-k}\in\N$.
But $$
v_p(\tfrac{(pk)!}{p^k})=v_p((pk)!)-v_p(p^k)=\tfrac{pk-s_{pk}}{p-1}-k
$$
and this is nonnegative since $k\geq s_{pk}=s_k$.

  Hence $\A_{\crr}(\overline{R})\subset \A^{[u]}_{\overline{R}}$ for $u\geq 1/(p-1)$, and $\A_{\crr}(\overline{R})\supset \A^{[u]}_{\overline{R}}$ for $u\leq 1/p$.  That means that, in the case of $u_p=(p-1)/p$, for $p>2$, $\A_{\crr}(\overline{R})\subset \A^{[u_p]}_{\overline{R}}$ and, for $p=2$ ($u_p=1/2$),  $\A_{\crr}(\overline{R})\supset \A^{[1/2]}_{\overline{R}}$ (we also have $\A_{\crr}(\overline{R})\subset \phi^{-1}(\A^{[1/2]}_{\overline{R}})$). 

\subsubsection{Fundamental exact sequences}
Recall that, 
  if $r\in\N$, we have  {\it the fundamental exact sequence}
$$0\to\Z_pt^{\{r\}}\to \xymatrix{F^r\acris\ar[r]^-{p^r-\varphi}&\acris}$$
where $t^{\{r\}}:=t^{b(r)}(t^{p-1}/p)^{a(r)}$, for $r=(p-1)a(r)+b(r),$ $0\leq b(r)\leq p-1$. Moreover the map $p^r-\phi$ is $p^r$-surjective. Set $\Z_p(r)^{\prime}:=\tfrac{1}{p^{a(r)}}\Z_p(r)$.
We will need the following generalizations of the above fundamental exact sequence. 
\begin{lemma}
\label{AS}
{\rm (i)} Let $0 <  v$.  We have the following Artin-Schreier  exact sequences 
\begin{equation}
\label{AS11}
  0\to \Z_p\to \xymatrix{\A_{\overline{R}}\ar[r]^-{1-\phi}&\A_{\overline{R}}} \to 0, \quad 
  0\to \Z_p\to \xymatrix{\A^{(0,v]+}_{\overline{R}}\ar[r]^-{1-\phi}&\A^{(0,v/p]+}_{\overline{R}}} \to 0
\end{equation}
{\rm (ii)} The following sequence is $p^r$-exact.
\begin{equation*}
 0\to \Z_p(r)^{\prime}\to \xymatrix{F^{r}\A_{\crr}(\overline{R})\ar[r]^-{p^r-\phi}&\A_{\crr}(\overline{R})}\to 0
 \end{equation*}
{\rm (iii)} Let $0 < u\leq 1\leq v$.  The following sequence is $p^{4r}$-exact
\begin{equation*}
   0\to \Z_p(r)\to \xymatrix{F^r\A^{[u,v]}_{\overline{R}}\ar[r]^-{p^r-\phi}&\A^{[u,v/p]}_{\overline{R}}}\to 0 \end{equation*}
Moreover, all the surjections above have continuous sections.
\end{lemma}
\begin{proof} The exactness of the first two sequences is  proved as in \cite[8.1]{AI}. The third  sequence was treated in \cite[A3.26]{Ts}. The exactness
 of the last one will follow from the exactness of  the following two sequences
\begin{align}
\label{exact1}
 & 0\to \Z_p(r)\to (\A^{[u,v]}_{\overline{R}})^{\phi=p^r}{\to}\A^{[u,v]}_{\overline{R}}/F^r \to 0\\
 & 0\to (\A^{[u,v]}_{\overline{R}})^{\phi=p^r}\to 
\xymatrix{\A^{[u,v]}_{\overline{R}}\ar[r]^-{p^r-\phi}&\A^{[u,v/p]}_{\overline{R}}}\to 0\nonumber
\end{align}

 To treat the first sequence, assume that $p>2$ and map the  sequence ($p^r$-exact) \begin{equation}
\label{exactcr} 
0\to  \Z_p(r)\to \A_{\crr}(\overline{R})^{\phi=p^r}{\to}\A_{\crr}(\overline{R})/F^r\to 0
\end{equation}
into it. We claim that the cokernel of this map is killed by $p^{2r}$: Indeed, note that we have the $p^r$-isomorphisms
$$\A_{\crr}(\overline{R})^{\phi=p^r}\stackrel{\sim}{\to}(\A^{[u]}_{\overline{R}})^{\phi=p^r}
\stackrel{\sim}{\to}(\A^{[u,v]}_{\overline{R}})^{\phi=p^r}.$$ For the first map the  injection is clear and the $p^r$-surjection follows from the fact that
$\phi(\A^{[u]}_{\overline{R}})\subset \A_{\crr}(\overline{R})
$. For the second map, again the injection is clear and to show surjection  it suffices  to prove that the map 
\begin{equation}
\label{exact2}
(p^r-\varphi):\quad \A^{[u,v]}_{\overline{R}}/\A^{[u]}_{\overline{R}}\stackrel{\sim}{\to}
\A^{[u,v/p]}_{\overline{R}}/\A^{[u]}_{\overline{R}}
\end{equation}
 is an isomorphism.
Since we have
 $$
 \A^{(0,v]+}_{\overline{R}}/\A^+_{\overline{R}}\simeq\A^{[u,v]}_{\overline{R}}/\A^{[u]}_{\overline{R}},\quad
 \A^{(0,v/p]+}_{\overline{R}}/\A^+_{\overline{R}}\simeq\A^{[u,v/p]}_{\overline{R}}/\A^{[u]}_{\overline{R}},
 $$
 we are reduced to showing that the map
 $$(p^r-\varphi):\quad  \A^{(0,v]+}_{\overline{R}}/\A^+_{\overline{R}}\to \A^{{(0,v/p]+}}_{\overline{R}}/\A^+_{\overline{R}}
 $$
 is an isomorphism. For $r=0$ this follows from the third sequence in our lemma  and  the Artin-Schreier theory as in  \cite[8.1.1]{AI}:
\begin{equation}
\label{AS1}
0\to \Z_p\to \A^+_{\overline{R}}\stackrel{1-\phi}{\longrightarrow} \A^+_{\overline{R}}\to 0
\end{equation} 
 For $r>0$, we can use the fact that 
the formal inverse of $p^r-\phi$ 
is $-(\phi^{-1}+ p^r\varphi^{-2}+p^{2r}\varphi^{-3}+\cdots)$ and it clearly converges.

 Note also that  the cokernel of the map $\A_{\crr}(\overline{R})/F^r\hookrightarrow \A^{[u,v]}_{\overline{R}}/F^r$ is killed by $p^r$. This follows from the fact that already
 the cokernel of the map $ \A^+_{\overline{R}}/F^r \hookrightarrow \A^{[u,v]}_{\overline{R}}/F^r$ is killed by $p^r$. 
 To see this  write $\A^{[u,v]}_{\overline{R}}=\A^+_{\overline{R}}[p/[\alpha],[\beta]/p]^{\wedge}$, $v_{\E}(\alpha)=1/v, v_{\E}(\beta)=1/u$,  and note that 
\begin{align*}
\tfrac{[\beta]}{p} & =\tfrac{[{p}^{\flat}]}{p}[\beta^{\prime}]=\tfrac{([{p}^{\flat}]-p)}{p}[\beta^{\prime}]+[\beta^{\prime}],\quad  v_{\E}(\beta^{\prime})=1/u-1;\\
\tfrac{p}{[\alpha]} & =(\tfrac{[{p}^{\flat}]}{p})^{-1}[\alpha^{\prime}]=(1+\tfrac{([{p}^{\flat}]-p)}{p})^{-1}[\alpha^{\prime}], \quad v_{\E}(\alpha^{\prime})=
1-1/v,
\end{align*}
and use the fact that the kernel of $\theta$ is generated by $[{p}^{\flat}]-p$. We have shown that the first sequence in (\ref{exact1}) is $p^{3r}$-exact for $p>2$. 

   For $p=2$ the argument is similar. We map the sequence 
$$
0\to\Z_p(r)\to (\A^{[u]}_{\overline{R}})^{\phi=p^r}\to \A^{[u]}_{\overline{R}}/F^r\to 0
$$
to both the sequence (\ref{exact1}) and  the sequence (\ref{exactcr}). By the above, the first map has the cokernel annihilated by $p^r$. The cokernel of the second map  is  annihilated by $p^{2r}$: use the fact that the cokernel of the map $\A^{[u]}_{\overline{R}}/F^r\hookrightarrow \A_{\crr}(\overline{R})/F^r$ is killed by $p^{r}$ and that we have $p^{r}$-isomorphism 
$$(\A^{[u]}_{\overline{R}})^{\phi=p^r}\stackrel{\sim}{\to}\A_{\crr}(\overline{R})^{\phi=p^r}\stackrel{\sim}{\to}
\phi^{-1}(\A^{[u]}_{\overline{R}})^{\phi=p^r}$$

   To show that the second sequence in (\ref{exact1}) is $p^{2r}$-exact it suffices to show that it is $p^{2r}$-exact on the right and for that,
   using the isomorphism (\ref{exact2}),   we can pass to the following sequence
\begin{equation}
\label{exact}
0\to(\A^{[u]}_{\overline{R}})^{\phi=p^r} \to \A^{[u]}_{\overline{R}}\stackrel{p^r-\phi}{\longrightarrow}\A^{[u]}_{\overline{R}}\to 0 \end{equation}
 That is, it remains to show that   the above sequence is $p^{2r}$-exact on the right. Write $\A^{[u]}_{\overline{R}}=
\A^{+}_{\overline{R}}[[\beta]/p]^{\wedge}$, $v_{\E}(\beta)=1/u$. Since the map $
\A^{+}_{\overline{R}}+[\beta]^r\A^{[u]}_{\overline{R}}\hookrightarrow  \A^{[u]}_{\overline{R}}$ is $p^r$-surjective it suffices to show that
$p^r-\phi$ is $p^r$-surjective on $\A^+_{\overline{R}}$ and on $[\beta]^r\A^{[u]}_{\overline{R}}$ separately. Surjectivity on $\A^+_{\overline{R}}$ is clear 
since for $r=0$ this is just the Artin-Schreier theory from (\ref{AS1}) and, for $r>0$, 
the series $-(1+p^r\varphi^{-1}+p^{2r}\varphi^{-2}+\cdots)$ converges
on
$\A^+_{\overline{R}}$ to an inverse of $(p^r\varphi^{-1}-1)$. 
 To check $p^r$-surjectivity on $[\beta]^r\A^{[u]}_{\overline{R}}$, for $x=[\beta]^ry$,
$y\in \A^{[u]}_{\overline{R}}$, it is enough to check that $(1+\tfrac{\phi}{p^r}+\tfrac{\phi^2}{p^{2r}}+\cdots)$
converges (pointwise), as this gives an inverse to $1-\frac{\varphi}{p^r}$.  We have
\begin{align*}
(1+\tfrac{\phi}{p^r}+\tfrac{\phi^2}{p^{2r}}+\cdots)(x)=\sum_{k\geq 0}\tfrac{\phi^k([\beta]^ry)}{p^{kr}}=
\sum_{k\geq 0}\tfrac{[\beta]^{p^kr}}{p^{kr}}\phi^k(y)=
\sum_{k\geq 0}p^{p^kr-kr}(\frac{[\beta]}{p})^{p^{k}r}\phi^k(y).
\end{align*}
We conclude, noticing that $p^kr- kr\geq 0$ and goes to $\infty$ when $k\to \infty$.

  Concerning the last claim of the lemma -- can be  proved for the first 
  two sequences as in \cite[8.1,8.1.2,8.1.3]{AI}. 
  The last two sequences are sequences of spaces with $p$-adic topology that are complete for that topology -- hence the existence of a continuous section is clear.  
\end{proof}
\begin{remark}
The same arguments can be used to prove the following result, where $\B^+_{\crr}=\A_{\crr}[\frac{1}{p}]$.

{\it Let $M$ be a $\phi$-module\footnote{A finite rank vector space over $F$ with a semilinear Frobenius isomorphism $\phi$.}
 over $F$,
$P\in F[X]$ such that $P(0)\neq 0$. 
Then $P(\varphi):\B^+_{\crr}\otimes_F M\to \B^+_{\crr}\otimes_FM$
is surjective.}

Since $\B^+_{\crr}$ contains $W(\overline k)$, Dieudonn\'e-Manin's theorem allows
us to assume that $M$ is the standard module of slope $\lambda=\frac{a}{h}$,
i.e., $M=Fe_1\oplus\cdots\oplus Fe_h$, and $\varphi(e_i)=e_{i+1}$, if $i\leq h$, $\varphi(e_h)=p^a e_1$.
Also, the result is true for $P$ if and only if it is true for all of its irreducible
divisors, which allows to assume that $P(0)=1$, that all roots of $P$ have the same valuation $\alpha$
and that $P(X)=Q(X^h)$ for some $Q\in F[X]$ (by replacing $P=\prod (1-a_i X)$ by its multiple $\prod (1-a_i^hX^h)$).
So, we just have to check that $R(\varphi^h)=Q(p^a\varphi^h)$ is surjective on $\B^+_{\crr}$
(note that all roots of $R$ have valuation $\beta=(\alpha-\lambda)h)$.

Write $1/R=1+b_1X+b_2X+\cdots$. We have $v_p(b_i)\geq -i\beta$, for all $i$,
which implies (by the arguments above: $p^{p^{hk}}-k\beta\to +\infty$ when $k\to +\infty$) 
that $1+b_1\varphi^h+b_2\varphi^{2h}+\cdots$ converges (pointwise)
on $[p^\flat]\B^+_{\crr}$, and that $P(\varphi)$ has an inverse (and hence is surjective)
on $[p^\flat]\B^+_{\crr}$.  

So, we are left to check that
$R(\varphi^h)$ is surjective on $\A^+_{{{\ovk}}}[1/p]$.  
Let us write $R$ as $1+a_1X+\cdots +a_dX^d$, with $a_d\neq 0$.
There are two cases:

$\bullet$ $\beta\leq 0$, which implies that $R\in\O_F[X]$ and $R(\varphi)$ sends $\A^+_{\ovk}$
into itself. Now,  modulo $p$, $R(\varphi^h)$
becomes $x\mapsto x+a_1x^{p^h}+\cdots+a_dx^{p^{dh}}$, and is surjective since
$\E_{\ovk}$ is algebraically closed and $\A^+_{{{\ovk}}}/p$ is its ring of integers
(if $\beta<0$, things are even simpler:
all $a_i$ are $0$ modulo~$p$, and $R(\varphi^h)$ is just $x\mapsto x$ modulo~$p$).
Since $\A^+_{\ovk}$ is $p$-adically complete, this implies that $R(\varphi^h):\A^+_{\ovk}\to\A^+_{\ovk}$
is, indeed, surjective.

$\bullet$ $\beta>0$, in which case $a_d^{-1}R\in\O_F[X]$ and is $X^d$ modulo~$p$.
It follows that $a_d^{-1}R(\varphi^h)$
becomes $x\mapsto x^{p^{dh}}$ modulo~$p$, which makes it clear that it is surjective.
\end{remark}

\subsection{Embeddings into period rings}
\subsubsection{Kummer embeddings}
Choose, inside $\overline{R}$, elements $X_i^{p^{-n}}$, for
$i=1,\dots,d$ and $n\in\N$, satisfying the obvious relations
(i.e. $X_i^{p^{-0}}=X_i$ and $(X_i^{p^{-(n+1)}})^p=X_i^{p^{-n}}$ if $n\geq 0$).
If $i=1,\dots,d$, let $X_i^\flat=(X_i,X_i^{1/p},\dots)\in \E^+_{\overline{R}}$.

Sending $X_0$ to $[\varpi^\flat]$ and $X_i$ to $[X_i^\flat]$, if $i=1,\dots,d$,
induces an embedding $\iota_{\rm Kum}$ of $R_{\varpi,\Box}^+$ into
$\A_{\overline{R}}$ which commutes with Frobenius  $\varphi$
and is compatible with filtrations.  
As $R_\varpi^+$ is \'etale over $R_{\varpi,\Box}^+$,
one can extend $\iota_{\rm Kum}$ to an embedding 
$R_\varpi^+\to \A_{\overline R}$ and, by continuity, to
embeddings 
$$R_\varpi^{\rm PD}\to \acris(\overline R),\quad
R_\varpi^{[u]}\to \A^{[u]}_{\overline R},\quad
R_\varpi^{[u,v]}\to \A^{[u,v]}_{\overline R},$$
which commute with Frobenius (with $\varphi_{\rm Kum}$ on
$R_\varpi^{\rm deco}$) and filtration (if $1\notin [u,v]$,
there is no filtration on the corresponding rings).

\subsubsection{Cyclotomic embedding of $R_\varpi$}
Let $R_\Box^{\rm cyl}=\O_{F_i}\{X_1,\dots,X_d\}$.
If $n\in\N$,
let $R_{\Box,n}^{\rm cycl}=\O_{F_{i+n}}\{X_1^{p^{-n}},\dots,X_d^{p^{-n}}\}$,
and let $R_n$ be the integral closure of $R$ in
the subalgebra of 
$\overline{R}[\frac{1}{p}]$ generated by $R$
and $R_{\Box,n}^{\rm cycl}$.
Set $R_{\Box,\infty}^{\rm cyl}=\cup_{n\in\N}R_{\Box,n}^{\rm cyl}$ and
$R_{\infty}=\cup_{n\in\N}R_{n}$.
Then $R_{\Box,\infty}^{\rm cyl}[\frac{1}{p}]$ and $R_\infty[\frac{1}{p}]$ are Galois extensions of
$R_\Box^{\rm cyl}[\frac{1}{p}]$ and $R[\frac{1}{p}]$ respectively, with Galois group
$\Gamma_R$ which is the semi-direct product
$$1\to\Gamma'_R\to\Gamma_R\to\Gamma_K\to 1,$$
where
\begin{align*}
\Gamma'_R={\rm Gal}(R_\infty[\tfrac{1}{p}]/K_\infty R[\tfrac{1}{p}])\simeq\Z_p^d,\quad
\Gamma_K={\rm Gal}(K_\infty/K)\simeq 1+p^{i(K)}\Z_p,
\end{align*}
and $a\in 1+p^{i(K)}\Z_p$
acts on $\Z_p^d$ by multiplication by~$a$.

\medskip
We define an embedding $\iota_{\rm cycl}:R_{\zeta-1,\Box}^+\to \A_{\overline{R_\Box}}^+$,
by sending $T$ to $\pi_i=\varphi^{-i}(\pi)$ and $X_i$ to $[X_i^\flat]$, where
$X_i^\flat=(X_i,X_i^{1/p},\dots)$.
This embedding commutes with Frobenius
(i.e.~$\varphi\circ\iota_{\rm cyl}=\iota_{\rm cycl}\circ\varphi_{\rm cycl}$)
and filtration (since $P_0$ is sent to $P_0(\pi_i)=\xi$).
\begin{proposition}\label{extr14}
$\iota_{\rm cycl}$ has a unique extension
to an embedding $R_\varpi^+\to\A^+_{\overline R}[[\frac{p}{\pi_i^{2\delta_R}}]]$
such that
$$\iota_{\rm cycl}(X_0)-[\varpi^\flat]\in \tfrac{\xi}{\pi_i^{\delta_R}} \A^+_{\overline R}[[\frac{p}{\pi_i^{2\delta_R}}]],
\quad{\text{ and $\theta\circ\iota_{\rm cycl}$ is the projection
$R_\varpi^+\to R$.}}$$
\end{proposition}
\begin{proof}
This is a consequence of Proposition ~\ref{LIFT}, with 
$\Lambda=\A^+_{\overline R}[[\frac{p}{\pi_i^{2\delta_R}}]]$,
$\lambda=\iota_{\rm cycl}$, $\beta:{\mathcal M}\to\Lambda$ defined by
$\beta(X_0)=[\varpi^\flat]$, $\beta(X_i)=[X_i^\flat]$ if $1\leq i\leq d$,
$\beta(X_{d+1})=\big[\frac{(\varpi^\flat)^h}{X_{a+1}^\flat\cdots X_{a+b}^\flat}\big]$,
$\mu= P_0$ (hence $\lambda(\mu)=\xi$), $J={\rm Ker}\,\theta=\big(\frac{xi}{\pi_i^{2\delta_R}}\big)$
and $\lambda_\varpi:R_\varpi^+\to \Lambda/J=\C^+(\overline R)$ being the natural map $R_\varpi^+\to R$.
 (We have $\theta(Q([\pi^\flat],\pi_i))=Q(\varpi,\zeta-1)=0$, hence
$Q(\beta(X_0),\lambda(T))$ is divisible by $\xi=\lambda(\mu)$ in $\A^+_{\ovk}\subset\Lambda$,
which shows that we can indeed apply Proposition ~\ref{LIFT}.)
\end{proof}

Let $$\pi_K=\iota_{\rm cycl}(X_0).$$
By construction $\theta(\pi_K)=\varpi$, and one can show that, if $n\in\N$,
then $\varpi_n=\theta(\varphi^{-1}(\pi_K))$ is a uniformizer of $K_n$.
Now, $\frac{\pi_K}{[\varpi^\flat]}$ is a unit in $\A^+_{\overline R}[[\frac{p}{\pi_i^{2\delta_R}}]]$.
This allows, using an obvious variant of Remark~\ref{LIFT0}, to extend
$\iota_{\rm cyl}$ to embeddings $R_\varpi\hookrightarrow \A_{\overline R}$,
$R_\varpi^{(0,v]+}\hookrightarrow \A_{\overline R}^{(0,v]+}$
and 
$R_\varpi^{[u,v]}\hookrightarrow \A_{\overline R}^{[u,v]}$, for $u\leq v<p v_R$,
which commute with Frobenius (by unicity)
and filtration (by construction).
\subsubsection{The action of $\Gamma_R$}
We denote by $\A_R^{\rm deco}$ (resp.~$\A_{R,\Box}^{\rm deco}$) the image
of $R_\varpi^{\rm deco}$ (resp.~$R_{\zeta-1,\Box}^{\rm deco}$)
by $\iota_{\rm cycl}$.
It is quite clear that $\A_{R,\Box}^+$ is stable under the action of $G_{R}$.
More precisely, $G_{R}$ acts through $\Gamma_R$, and
we can choose topological
generators $\gamma_j$, $0\leq j\leq d$, of $\Gamma_R$ with the following properties.
Let $$c=\exp(p^i);\quad{\text{ in particular $a:=p^{-i}(c-1)\in\Z_p^*$}}.$$
Then
$$\begin{cases}
\gamma_0(\pi_i)=(1+\pi_i)^c-1=(1+\pi)^{a}(1+\pi_i)-1
\hskip.1cm{\rm and}\hskip.1cm
\gamma_j(\pi_i)=\pi_i {\text{ if $1\leq j\leq d$.}}\\
\gamma_k([X_k^\flat])=[\epsilon]\,[X_k^\flat]=(1+\pi)\,[X_k^\flat]
\hskip.1cm{\rm and}\hskip.1cm
\gamma_j([X_k^\flat])=[X_k^\flat] {\text{ if $j\neq k$ and $1\leq k\leq d$.}}
\end{cases}
$$
It follows that $\gamma_1,\dots,\gamma_d$ are topological generators
of $\Gamma'_R$.

The induced action of $\Gamma_R$ on $R_{\zeta-1,\Box}^+$ is given by the formulas:
\begin{align*}
\gamma_0(T)=(1+T)^{p^ia}(1+T)-1,\quad \gamma_j(X_j)=(1+T)^{p^i}X_j,\ {\text{if $1\leq j\leq d$,}}\\
\gamma_k(T)=T,\ {\text{if $k\neq 0$}},\quad\gamma_k(X_j)=X_j,\ {\text{if $k\neq j$.}}
\end{align*}

\begin{proposition}\label{extr12.1}
Let $\gamma:R_{\zeta-1,\Box}^+\to R_{\zeta-1,\Box}^+$ be a continuous
ring morphism such that there exists $a_j\in\Z_p$, for $0\leq j\leq d$, such that:
$$\gamma(X_j)=(1+T)^{p^ia_j}X_j,\ {\text{if $1\leq j\leq d$}},\quad
\gamma(T)=(1+T)^{p^ia_0}(1+T)-1.$$
Then $\gamma$ extends
uniquely to a continuous ring morphism $R_\varpi^+\to R_\varpi^+[[\tfrac{p}{T^{2\delta_R}}]]$,
such that
$$\gamma(x)-x\in \tfrac{(1+T)^{p^i}-1}{T^{1+\delta_R}}R_\varpi^+[[\tfrac{p}{T^{2\delta_R}}]],
\ {\text{ if
$x\in R_\varpi^+$}}
\quad{\rm and}\quad
\gamma(X_0)-X_0\in \tfrac{(1+T)^{p^i}-1}{T^{\delta_R}}R_\varpi^+[[\tfrac{p}{T^{2\delta_R}}]].$$
\end{proposition}
\begin{proof}
This is a consequence of Proposition ~\ref{LIFT}, with 
$\Lambda=R_\varpi^+[[\tfrac{p}{T^{2\delta_R}}]]$, $\lambda=\gamma$, $\beta(x)=x$,
$\mu=(1+T)^{p^i}-1$, $\lambda_\varpi(x)=x$, and $J=\big(\tfrac{(1+T)^{p^i}-1}{T^{1+\delta_R}}\big)$.
To check the requirements of that proposition, note that $\frac{\lambda(\mu)}{\mu}$
is a unit and that $\gamma(T)-T$ is divisible by $\mu$ in $r_{\zeta-1}^+$, hence
$Q(X_0,\gamma(T))\equiv Q(X_0,T)$ mod~$\mu r_\varpi^+$, and
$Q(\beta(X_0),\lambda(T))\in\mu r_\varpi^+\subset \mu \Lambda$.
The rest of the requirements are obvious.
\end{proof}

\begin{corollary}\label{extr12.3}
{\rm (i)} The action of $\Gamma_R$ extends uniquely to an action
on $R_\varpi$, which stabilizes $R_\varpi^{(0,v]+}$ if $v<p v_R$
and extends to a continuous action on $R_\varpi^{[u,v]}$, if $u\leq v<p v_R$.
Moreover this action is trivial modulo $T^{-1-\delta_R}((1+T)^{p^i}-1)$
on all these rings.

{\rm (ii)} The rings
$\A_R, \A_R^{(0,v]+},\A_R^{(0,v]},\A_R^{[u,v]}$ are stable
by $G_R$ which acts through $\Gamma_R$, and $\iota_{\rm cycl}$ commutes with
the action of $\Gamma_R$.
\end{corollary}
\begin{proof}
For (i), just use Remark~\ref{LIFT0}.
For (ii), use the unicity in Proposiiton~\ref{extr12.1} and~\ref{extr14}: it implies that
$\sigma\circ\iota_{\rm cycl}=\iota_{\rm cycl}\circ\overline\sigma$, if $\sigma\in G_R$
and $\overline\sigma$ is its image in $\Gamma_R$. 
\end{proof}

The action of $\Gamma_R$ on $\A_R^{[u,v]}[\frac{1}{p}]$ is analytic,
 and the
action of its Lie algebra ${\rm Lie}\,\Gamma_R$ is given by:
$$\log\gamma_j=t\partial_j,\quad{\text{if $0\leq j\leq d$}},$$
as one sees easily using the above formulas for the action of $\Gamma_R$
(here $\partial_j$ is the derivation of $\A_R^{[u,v]}[\frac{1}{p}]$
deduced from $\partial_{{\rm cycl},j}$ on $R_\varpi^{[u,v]}[\frac{1}{p}]$ by transport of structure).

\begin{lemma}\label{extr12.2}
If $u\leq v<p$ and $S=R_\varpi^{(0,v]+}$, $R_\varpi^{[u,v]}$, then
$$(\gamma-1)\cdot T^n(p,T^{e_0})^kS\subset T^{n+p^{i-1}-{\delta_R}}
(p,T^{e_0})^{k+1}S
\quad{\rm and}\quad
(\gamma-1)^k\cdot S\subset T^{k(p^{i-1}-{\delta_R}-1)}
(p,T^{e_0})^{k}S.$$
\end{lemma}
\begin{proof}
(ii) follows from the triviality of $\gamma$ modulo
$T^{-1-{\delta_R}}((1+T)^{p^i}-1)$, the fact that
$\gamma-1$ acts as a twisted derivation [\,$(\gamma-1)\cdot xy=((\gamma-1)\cdot x)y+
\gamma(x)((\gamma-1)\cdot y)$\,],
which implies that $$(\gamma-1)\cdot T^nS\subset T^{n-1-{\delta_R}}((1+T)^{p^i}-1)S,$$
and the fact that $T^{p^{i-1}}$ divides
$(1+T)^{p^{i-1}}-1$ (because $v<p$) and
$\tfrac{(1+T)^{p^i}-1}{(1+T)^{p^{i-1}}-1}\in (p,T^{e_0})$ (see Lemma~\ref{new2} below).
\end{proof}

\begin{lemma}\label{new2}
If $v<p$, then:

$\bullet$ $\pi_i^{-p^{i-1}}\pi_1$ is a unit of $\A_R^{(0,v]+}$.

$\bullet$ $p$ is divisible by $\pi_i^{[\frac{(p-1)p^{i-1}}{v}]}$, hence also
by $\pi_i^{(p-1)p^{i-2}}$.

$\bullet$ $\frac{p^2}{\pi_1}\in\A_R^{(0,v]+}$ and is divisible by
$\pi_i^{(2(p-1)-v)p^{i-2}}$.

$\bullet$ $\frac{\pi}{\pi_1}\in (p,\pi_i^{(p-1)p^{i-1}})\A_R^{(0,v]+}$ 
and is divisible by
$\pi_i^{(p-1)p^{i-2}}$, hence also by $\pi_i^{\delta_R+1}$.

\end{lemma}
\begin{proof}
We can work in $r_{\zeta-1}^{(0,v]+}$,
in which case $\pi_i$ becomes $T$ and $\pi_1$ becomes $(1+T)^{p^{i-1}}-1$,
and we are looking at the annulus $0<v_p(T)\leq \frac{v}{(p-1)p^{i-1}}$
on which $(1+T)^{p^{i-1}}-1$ has no zero
and $v_p((1+T)^{p^{i-1}}-1)=p^{i-1}v_p(T)$ since $v<p$.
This proves the first point.

The second is basically the definition of $\A_R^{(0,v]+}$.

The third point follows from the first two, which give
$\frac{p^2}{\pi_1}$ divisible by $\pi_i^{2[\frac{(p-1)p^{i-1}}{v}]-p^{i-1}}$,
and the inequality $2[\frac{(p-1)p^{i-1}}{v}]-p^{i-1}\geq (2(p-1)-v)p^{i-2}$.

The fourth follows from the first two and the identity $\pi=\pi_1^{p-1}+p\pi_1^{p-2}+\cdots+p$.
\end{proof}

\subsection{Fat period rings}
We are going to make period rings fatter (Scholze~\cite{Sch} even call them $\O{\rm B}$'s) by adding to them
the variables $X_0,\dots,X_d$ (i.e.~by tensoring with $R_\varpi^+$ and its siblings). 
This kind of rings~\cite{Ts,Br} originate in linearizations of differential operators.

\subsubsection{Structure theorem}
Let:

$\bullet$ $S=R_{\varpi}^{\rm PD}$, $R_{\varpi}^{[u,v]}$ (and $S_\Box=R_{\varpi,\Box}^{\rm PD}$, $R_{\varpi,\Box}^{[u,v]}$),

$\bullet$ $\Lambda$ be a $p$-adically complete filtered $\O_F$-algebra,

$\bullet$ $\iota:S\to \Lambda$ be a continuous injective morphism of filtered $\O_F$-algebras,

$\bullet$ $f:S\otimes\Lambda\to\Lambda$ be the morphism sending $x\otimes y$ to $\iota(x)y$,

$\bullet$ $S\Lambda_s$ be 
the $p$-adically completed log-PD-envelope of $S\otimes\Lambda\to\Lambda$ with respect to ${\rm Ker}\,f$.

(We take partial divided powers of level $s$: i.e.~$x^{[k]}=\frac{x^k}{[k/p^s]!}$.) 
By definition, $S\Lambda_s$ is the $p$-adic completion of
$S\otimes\Lambda$ adjoined $(x\otimes 1-1\otimes \iota(x))^{[k]}$, for $x\in S$ and $k\in\N$,
and $(V_j-1)^{[k]}$, for $0\leq j\leq d$ and $k\in\N$, where $V_j=\frac{X_j\otimes 1}{1\otimes\iota(X_j)}$.
(Adding the $(V_j-1)^{[k]}$'s makes all the difference between the log-PD-envelope and the PD-envelope.)
The morphism $f:S\otimes\Lambda\to\Lambda$
extends uniquely to a continuous morphism $f:S\Lambda_s\to\Lambda$.

We consider $S$ and $\Lambda$ as sub-algebras of $S\Lambda_s$
(by $x\mapsto x\otimes 1$ and $y\mapsto 1\otimes y$).
Via these identifications, $f$ induces the identity map on $\Lambda$,
$V_j=\frac{X_j}{\iota(X_j)}$ if $0\leq j\leq d$, and 
$(x\otimes 1-1\otimes \iota(x))^{[k]}$ becomes $(x-\iota(x))^{[k]}$ if $x\in S$.

We filter $S\Lambda_s$ by defining $F^rS\Lambda_s$ to be the topological closure
of the ideal generated by products of the form $x_1x_2\prod(V_j-1)^{[k_j]}$, with
$x_1\in F^{r_1}R_1$, $x_2\in F^{r_2}R_2$, and $r_1+r_2+\sum k_j\geq r$.

\begin{lemma}
\label{18saint}
{\rm (i)}
Any element $x\in S\Lambda_s$ can be written, uniquely,
as 
$$x=\sum_{{\bf k}\in\N^{d+1}}x_{{\bf k}}
\prod_{j=0}^d\big(1-V_j\big)^{[k_j]},$$
 with $x_{{\bf k}}\in \Lambda$
for all ${\bf k}=(k_0,\dots,k_d)\in\N^{d+1}$, and $x_{\bf k}\to 0$ when ${\bf k}\to\infty$.

{\rm (ii)}
$x\in F^rS\Lambda_s$ if and only if $x_{\bf k}\in F^{r-|{\bf k}|}\Lambda$
for all ${\bf k}\in\N^{d+1}$.
\end{lemma}
\begin{proof}
Let $\Lambda'=\Lambda [\big(1-V_j\big)^{[k]},\ 0\leq j\leq d,\ k\in\N\big]^\wedge$
and let $\iota':\Lambda'\to
S\Lambda_s$ be the natural map.  We want to prove it is bijective.

For injectivity, use the morphism $f:S\Lambda_s\to \Lambda$:
it kills the
$\prod_{j=0}^d\big(1-V_j\big)^{[k_j]}$ with $|{\bf k}|>0$.
Now $\partial_j=X_{j}\frac{\partial}{\partial X_{j}}$ extends by continuity to 
$S\Lambda_s$ and $\Lambda'$,
and commutes with $\iota'$, so we can prove by
induction on $|{\bf k}|$ that if $\iota'(x)=0$, then $x_{\bf k}=0$ for all ${\bf k}$
(apply $f$ to $\big(\prod\partial_j^{k_j}\big)\circ\iota'$, using the fact that
$\partial_jV_\ell=0$ if $\ell\neq j$ and $\partial_jV_j=V_j=(V_j-1)+1$).

For surjectivity, first note that the kernel of $f\circ\iota':\Lambda'\to \Lambda$
is the PD-ideal generated by the $(V_j-1)$'s since $f\circ\iota'$ is the identity map on $\Lambda $.
Let $S_{\Box}\subset S$.  There exists a unique
morphism $\iota_0:S_{\Box}\to \Lambda'$ such that $\iota'\circ\iota_0$ is the natural injection
$S_{\Box}\to S\Lambda_s$: unicity because $\iota'$ is injective, existence
thanks to the formula
$$\iota_0\big(p^a\prod_{j=0}^dX_{j}^{k_j}\big)=
p^a\prod_{j=0}^d\big(\iota(X_j)^{k_j}\big(1-V_j\big)^{k_j}\big)
= p^a\prod_{j=0}^d\big(\iota(X_j)^{k_j}\big(\sum_{n=0}^{+\infty}
(-1)^n[\tfrac{n}{p^s}]!\tbinom{k_j}{n}\big(V_j-1\big)^{[n]}\big)\big).$$
The above formula shows that $\iota_0(x)-\iota(x)\in {\rm Ker}\,f\circ\iota'$, for all $x\in S_{\Box}$.
By \'etaleness (Remark~\ref{extr6.1}), this implies that $\iota_0$ admits a unique extension to
a morphism $\iota_0:S\to \Lambda'$ such that $\iota_0(x)-\iota(x)\in
{\rm Ker}\,f\circ\iota'$, and by unicity, $\iota'\circ\iota_0$ is the natural injection
$S\to S\Lambda_s$.  This shows that $\iota'(\Lambda')$ contains $S$.
Finally, ${\rm Ker}\,f$ is generated by the $(V_j-1)^{[k]}$'s and by elements of the form
$x-\iota(x)$, for $x\in S$, but these are
in the image of the PD-ideal generated by the $(V_j-1)$'s.
Hence the image of the PD-ideal generated by the $(V_j-1)$'s
by $\iota'$ contains the log-PD-ideal generated by ${\rm Ker}\,f$
which shows that $\iota'$ is indeed surjective.

To prove the statement about filtrations, we argue by induction on $r$.
Applying $\partial_j$, for $0\leq j\leq d$, and arguing by induction on
$|{\bf k}|$, shows that $x_{{\bf k}}\in F^{r-|{\bf k}|}\Lambda $, if ${\bf k}\neq 0$.
Now, this implies that $x_{0}\in \Lambda \cap F^rS\Lambda_s$.
But $x_{0}=f(x_0)$, because $x_{0}\in \Lambda $,
and $f(F^rS\Lambda_s)\subset F^r\Lambda $ because $f$ is the identity
on $F^{r_2}\Lambda $, kills $(V_j-1)^{[k_j]}$ if $k_j\geq 1$, 
and sends $F^{r_1}S$ to $F^{r_1}\Lambda $ since it coincides with $\iota$ on $S$.  
So $x_{0}\in F^r\Lambda $, which concludes the proof.
\end{proof}

\subsubsection{The filtered Poincar\'e Lemma}
Let, as usual, $\Omega^1=\oplus_{j=0}^d\Z\,\tfrac{dX_{j}}{X_{j}}$ and
let $\Omega^n=\wedge^n\Omega^1$. So 
$\Omega^n_{S\Lambda_s/\Lambda}=S\Lambda_s\otimes \Omega^n$.
We filter the de Rham complex of $S\Lambda_s$ as usual by subcomplexes
$$F^r\Omega_{S\Lambda_s/\Lambda}\kr:=F^rS\Lambda_s\to F^{r-1}S\Lambda_s\otimes\Omega^1\to F^{r-2}S\Lambda_s\otimes\Omega^2\to\cdots$$
\begin{lemma}\label{above}
{\rm (Filtered Poincar\'e Lemma)}
The natural map
$$F^r\Lambda\rightarrow F^r\Omega_{S\Lambda_s/\Lambda}\kr$$
is a $p^s$-quasi-isomorphim.
\end{lemma}
\begin{proof}
Let $\epsilon:F^r\Lambda\to F^r S\Lambda_s$ be the natural injection.
A contracting ($\Lambda$-linear) homotopy can be defined as 
follows.
 We define the map $$h^0:F^rS\Lambda_s\to F^r\Lambda ,\quad 
\sum_{{\bf k}\in\N^{d+1}}a_{{\bf k}}\prod_{j=0}^d\big(V_j-1\big)^{[k_j]}\mapsto a_0.$$ 
Clearly $h^0\epsilon=\id$.
For $q >0$, we define  the map $$h^q: F^{j-q}S\Lambda_s\otimes\Omega^q\to 
F^{j-q+1}S\Lambda_s\otimes\Omega^{q-1}$$ by the following formula
\begin{align*}
 & a\prod_{j=0}^{d}\big(V_j-1\big)^{[k_j]}V_{j_1}\tfrac{dX_{j_1}}{X_{j_1}}
\wedge \cdots\wedge V_{j_q}\tfrac{dX_{j_q}}{X_{j_q}} \quad (0\leq j_1 < \cdots < j_q\leq d)\\ 
& \quad \mapsto
\begin{cases}
\frac{k_j!}{(k_j+\delta_{jj_1})!}\,
\frac{[(k_j+\delta_{jj_1})/p^s]!}{[k_j/p^s]!}
a\prod_{j=0}^{d}\big(V_j-1\big)^{[k_j+\delta_{jj_1}]}
V_{j_2}\tfrac{dX_{j_2}}{X_{j_2}}\wedge \cdots\wedge 
V_{j_q}\tfrac{dX_{j_q}}{X_{j_q}}
  &  {\text{if $k_j=0$ for $0\leq j <j_1$,}}\\
0   &  {\text{otherwise.}}
\end{cases}
\end{align*}
We have $\epsilon h^0+h^1d=\id$ and 
$dh^q+h^{q+1}d=\id$, as wanted. 
(The correcting factor $\frac{k_j!}{(k_j+\delta_{jj_1})!}\,
\frac{[(k_j+\delta_{jj_1})/p^s]!}{[k_j/p^s]!}$ is due to the fact that we have partial divided powers;
its valuation is between $0$ and $-s$ depending on the valuation of $k_j+\delta_{jj_1}$, which explains why
we only get a $p^s$-quasi-isomorphism.)
\end{proof}

\subsubsection{Extending the actions of $\varphi$ and $G_R$}
We take
 $s=0$ if $p\geq 3$ and $s=1$ if $p=2$, and define:

\quad $\bullet$  $E_R^{[u,v]}$, $E_{R_\infty}^{[u,v]}$ as $S\Lambda_s$
for $S=R_\varpi^{[u,v]}$, $\Lambda=\A_R^{[u,v]}$ or
$\A_{R_\infty}^{[u,v]}$, and $\iota=\iota_{\rm cycl}$.

\quad $\bullet$ $E_{\overline R}^{\rm PD}$
as $S\Lambda_s$ for $S=R_\varpi^{\rm PD}$,
$\Lambda=\A_{\crr}(\overline R)$, and $\iota=\iota_{\rm Kum}$,

\quad $\bullet$ $E_{\overline R}^{[u,v]}$ 
as $S\Lambda_s$ for $S=R_\varpi^{[u,v]}$,
$\Lambda=\A_{\overline R}^{[u,v]}$
and $\iota=\iota_{\rm Kum}$.

\begin{lemma}\label{fat1}
{\rm (i)}  $E_{\overline R}^{\rm PD}\subset E_{\overline R}^{[u,v]}$.

{\rm (ii)} $
E_R^{[u,v]}\subset E_{R_\infty}^{[u,v]}\subset E_{\overline R}^{[u,v]}.$
\end{lemma}
\begin{proof}
(i) is obvious from the structure result of Lemma~\ref{18saint}.
To prove (ii), granted Lemma~\ref{18saint}, we just have to check that
$(1-\frac{X_0}{\pi_K})^{[k]}\in E_{\overline R}^{[u,v]}$.
But we have $1-\frac{X_0}{\pi_K}=\big(1-\frac{X_0}{[\varpi^\flat]}\big)+
\frac{X_0}{[\varpi^\flat]}\frac{\pi_K-[\varpi^\flat]}{\pi_K}$.
Now 
$1-\frac{X_0}{[\varpi^\flat]}$
has divided powers of level $s$ in $E_{\overline R}^{[u,v]}$ by construction, so
we just have to show
that $\frac{\pi_K-[\varpi^\flat]}{[\varpi^\flat]}$
admits divided powers of level $s$ in $\A_{\overline R}^{[u,v]}$.
For that, we are going to use the assumptions $2p(\delta_R+1)<e_0$,
 $u\geq\frac{p-1}{p}$, $v<p$ (and $u=\frac{3}{4}$, $v=\frac{3}{2}$, if $p=2$).
We have
$$\tfrac{\pi_K-[\varpi^\flat]}{[\varpi^\flat]}
\in\tfrac{\xi}{\pi_i^{\delta_R+1}}\A_R^+[[\tfrac{p}{\pi_i^{2\delta_R}}]]
\quad{\rm and}\quad
\xi=\pi^{e_0}+p\alpha,\ {\text{with $\alpha\in\A_R^+$.}}$$
But, $v^{[u,v]}(\tfrac{p}{\pi_i^{2\delta_R}})\geq 0$, since $v<p$, and
$$v^{[u,v]}\big(\tfrac{\xi}{\pi_i^{\delta_R+1}}\big)\geq \min\big((e_0-\delta_R-1)\tfrac{u}{e_0},
1-(\delta_R+1)\tfrac{v}{e_0}\big)
\geq\min\big((1-\tfrac{1}{2p})u,1-\tfrac{1}{2}\big).$$
It follows that $v^{[u,v]}(\tfrac{\pi_K-[\varpi^\flat]}{[\varpi^\flat]})\geq\frac{1}{(p-1)p^s}$,
for all $p$,
which allows to conclude.
\end{proof}
\begin{remark}\label{fat2}
It follows from the proof that we could have also defined
$E_{\overline R}^{[u,v]}$ using $\iota_{\rm cycl}$ instead of~$\iota_{\rm Kum}$.
\end{remark}

If $(S,\Lambda)$ is any of the pairs above, there are natural, commuting, actions
of $G_R$ and $\varphi$ on $S\otimes\Lambda$ (tensor actions, with $G_R$ acting
trivially on $S$, and $\varphi=\varphi_{\rm Kum}$ or $\varphi_{\rm cycl}$ on $S$, if
$\iota=\iota_{\rm Kum}$ or $\iota_{\rm cycl}$).
\begin{lemma}\label{fat3}
{\rm (i)}  $\varphi$ extends uniquely to continuous morphisms
$$E_{\overline R}^{\rm PD}\to E_{\overline R}^{\rm PD},
\quad E_R^{[u,v]}\to E_R^{[u,v/p]}, 
\quad E_{R_\infty}^{[u,v]}\to E_{R_\infty}^{[u,v/p]},
\quad E_{\overline R}^{[u,v]}\to E_{\overline R}^{[u,v/p]}.$$

{\rm (ii)}
The action of $G_R$ extends uniquely to continuous
actions on $E_{\overline R}^{\rm PD}$, $E_R^{[u,v]}$, $E_{R_\infty}^{[u,v]}$ and $E_{\overline R}^{[u,v]}$,
commuting with $\varphi$.
Moreover $E_{R_\infty}^{[u,v]}$ is the fixed point of
$E_{\overline R}^{[u,v]}$ by ${\rm Gal}(\overline R[\frac{1}{p}]/R_\infty[\frac{1}{p}])\subset G_R$.
\end{lemma}
\begin{proof}
Lemma~\ref{18saint} reduces the statement to the verification that
$(1-\varphi(V_j))^{[k]}$ and $(1-\sigma(V_j))^{[k]}$ belong to the relevant ring,
for $0\leq j\leq d$ and $\sigma\in G_R$.

Let us start with Frobenius.  If we use $\iota_{\rm Kum}$, the result is obvious
as, in that case, $\varphi(V_j)=V_j^p$.  If we use $\iota_{\rm cycl}$,
the same argument is still valid for $1\leq j\leq d$, and we are left
to check that $(1-\frac{X_0^p}{\varphi(\pi_K)})^{[k]}\in E_R^{[u,v/p]}$.
But 
$$1-\tfrac{X_0^p}{\varphi(\pi_K)}=\big(1-\tfrac{X_0^p}{\pi_K^p}\big)\tfrac{\pi_K^p}{\varphi(\pi_K)}+
\big(1-\tfrac{\pi_K^p}{\varphi(\pi_K)}\big),$$
and the result follows from the fact that $\varphi_{\rm cycl}$ is admissible, hence
$1-\tfrac{\pi_K^p}{\varphi(\pi_K)}$ has divided powers of level $s$ in
$\A_R^{[u,v/p]}$.

For the action of $G_R$, we have to check that
$\big(1-\frac{X_j}{\sigma(\iota(X_j))}\big)^{[k]}$ makes sense.
If $\iota=\iota_{\rm Kum}$ or if $j\geq 1$, we have
$\sigma(\iota(X_j))=[\epsilon^c_j(\sigma)]\iota(X_j)$ for
some $c(\sigma)\in\Z_p$.  So we can write
$$1-\tfrac{X_j}{\sigma(\iota(X_j))}=
\big(1-\tfrac{X_j}{\iota(X_j)}\big)+(1-[\epsilon]^{-c_j(\sigma)})\tfrac{X_j}{\iota(X_j)},$$
and the results follows from the fact that $1-[\epsilon]^{-c_j(\sigma)}$ is divisible by $\pi$
which has divided powers of any level.
The case $j=0$ and $\iota=\iota_{\rm cycl}$ is more subtle: we have to consider
$$1-\tfrac{X_0}{\sigma(\pi_K)}=\big(1-\tfrac{X_0}{\pi_K}\big)-\tfrac{X_0}{\pi_K}\tfrac{\pi_K-\sigma(\pi_K)}{\sigma(\pi_K)}.$$
But
$\tfrac{\pi_K-\sigma(\pi_K)}{\sigma(\pi_K)}\in \pi_i^{-\delta_R-1}\pi\A_R^{[u,v]}$,
and $\pi_i^{-\delta_R-1}\pi=\pi_i^{-\delta_R-1}\pi_1\xi$ is divisible by $\xi$
in $\A_R^{[u,v]}$ (by
Lemma~\ref{new2}), and 
and $\xi$ has divided powers in $\A_R^{[u,v]}$
since it has already divided powers in $\A_F^{\rm PD}$.

Finally, the ``Moreover'' is obvious from Lemma~\ref{18saint}, Remark ~\ref{fat2}, and the fact that
$V_j$ is fixed by ${\rm Gal}(\overline R[\frac{1}{p}]/R_\infty[\frac{1}{p}])$ if
we use $\iota_{\rm cycl}$.
\end{proof}

\begin{example}
If $R=\so_K$, we have $E^{\rm PD}_{\overline{R}}=\wh{\A}_{\st}$, where
$\wh{\A}_{\st}$ is the $p$-adic completion of the ring
$\A_{\crr}[\frac{Y^k}{k!},\ k\in\N]$, where $Y=\frac{X_0}{[\varpi^\flat]}-1$.
Since $\varphi(X_0)=X_0^p$, we have $\varphi(Y)=(1+Y)^p-1$, and
since $\sigma([\varpi^\flat])=[\epsilon]^{c(\sigma)}[\varpi^\flat]$, with
$c(\sigma)\in\Z_p$ if $\sigma\in G_K$, we have
$\sigma(Y)=[\epsilon]^{-c(\sigma)}(1+Y)-1$.

A more conceptual description would be to define
$\wh{\A}_{\st}$ as the $H^0$ of $\O_{\overline K}$ for the
log-crystalline cohomology:
$$
\wh{\A}_{\st,n}=H^0_{\crr}(\so_{\overline{K},n}^{\times}/r^{\rm PD}_{\varpi,n}), \quad \wh{\A}_{\st}=\projlim_nH^0_{\crr}(\so_{\overline{K},n}^{\times}/r^{\rm PD}_{\varpi,n}).
$$
\end{example}

\section{Local syntomic computations}\label{Losy}
The goal of this section is to construct, {\it if $K$ contains enough roots of unity},
 a natural ``quasi-isomorphism'' between
the complex ${\rm Syn}(R,r)$ computing syntomic cohomology of $R$, and a complex
${\rm Cycl}(R_\varpi^{[u,v]},r)$ that is closer to Galois cohomology.
For example, in dimension~$0$, choosing a basis of $\Omega^1$, the two complexes
become (with $\partial_{\rm Kum}=X_0\frac{d}{dX_0}$ and $\partial_{\rm cycl}=(1+X_0)\frac{d}{dX_0}$):
\begin{align*}
{\rm Syn}(\so_K,r):\quad  & 
\xymatrix@C=2.8cm{F^rr_\varpi^{\rm PD}\ar[r]^-{(\partial_{\rm Kum},p^r-\phi_{\rm Kum})}
&F^{r-1}r_\varpi^{\rm PD}\oplus r_\varpi^{\rm PD}\ar[r]^-{-(p^r-p\phi_{\rm Kum})+\partial_{\rm Kum}}&r_\varpi^{\rm PD}}\,,\\
{\rm Cycl}(r_\varpi^{[u,v]},r):\quad  & 
\xymatrix@C=2.8cm{F^rr_\varpi^{[u,v]}\ar[r]^-{(\partial_{\rm cycl},p^r-\phi_{\rm cycl})}
&F^{r-1}r_\varpi^{[u,v]}\oplus r_\varpi^{[u,v/p]}\ar[r]^-{-(p^r-p\phi_{\rm cycl})+\partial_{\rm cycl}}&r_\varpi^{[u,v/p]}}\,.
\end{align*}
The main step (Proposition ~\ref{RtoA}) is the shift from Frobenius $\varphi_{\rm Kum}$ to $\varphi_{\rm cycl}$
(this uses standard crystalline techniques).  But, since $\varphi_{\rm cycl}$ is not defined on
$R_\varpi^{\rm PD}$, one must first replace $R_\varpi^{\rm PD}$ by $R_\varpi^{[u,v]}$.
This is done through a string of ``quasi-isomorphisms''
$${\rm Syn}(R,r):={\rm Kum}(R_\varpi^{\rm PD},r)\cong {\rm Kum}(R_\varpi^{[u]},r)
\cong {\rm Kum}^\psi(R_\varpi^{[u]},r)\cong {\rm Kum}^\psi(R_\varpi^{[u,v]},r) \cong
{\rm Kum}(R_\varpi^{[u,v]},r),$$
using techniques coming from $(\varphi,\Gamma)$-module theory.
(Actually, it is better to truncate in degree~$r$ as the constants involved in the
``quasi-isomorphism'' ${\rm Kum}^\psi(R_\varpi^{[u]},r)\cong {\rm Kum}^\psi(R_\varpi^{[u,v]},r)$
are too big in degrees~${\geq}\,r$.)
\subsection{Local complexes computing syntomic cohomology}
If $$S=R_\varpi^{\rm PD}, R_\varpi^{[u]}, R_\varpi^{[u,v]},$$
set 
$$S'=R_\varpi^{\rm PD}, R_\varpi^{[u]}, R_\varpi^{[u,v/p]}.$$
We endow all these rings with Frobenius $\varphi_{\rm Kum}$,
and we have $\varphi(S)\subset S'$ in all cases.

If $i\in\N$, let
$${\text{$J_i=\{0\leq j_1<\cdots<j_i\leq d\}$\quad and\quad $\omega_{\bf j}=\big(\tfrac{dX_{j_1}}{X_{j_1}}\wedge\cdots\wedge\tfrac{dX_{j_i}}{X_{j_i}}\big)$
if ${\bf j}=(j_1,\dots,j_i)\in J_i$.}}$$

We define $\varphi_{\rm Kum}$ and $\psi_{\rm Kum}$ on $\Omega_S^i$ by
$$\varphi_{\rm Kum}\big(\sum_{{\bf j}\in J_i}f_{\bf j}\omega_{\bf j}\big)=
\sum_{{\bf j}\in J_i}\varphi_{\rm Kum}(f_{\bf j})\omega_{\bf j}
\quad
\psi_{\rm Kum}\big(\sum_{{\bf j}\in J_i}f_{\bf j}\omega_{\bf j}\big)=
\sum_{{\bf j}\in J_i}\psi_{\rm Kum}(f_{\bf j})\omega_{\bf j}.$$
(This is not the natural definition, as we have $d(\varphi_{\rm Kum}(f))=p\varphi_{\rm Kum}(df)$
with this definition.  The reason for doing this is that we need the left
inverse $\psi$ of $\varphi$ which makes it necessary, for integrality reasons,
to divide the natural Frobenius by suitable powers of $p$.  The price to pay
is that these powers of $p$ are showing up in the complexes that we are
about to define.)

We define
$F^r\Omega^i_S$ as the sub-$S$-module of $\Omega_S^i$ generated by $F^rS\cdot \Omega^i_S$; we have
thus $$F^r\Omega^i_S=\oplus_{{\bf j}\in J_i}F^rS\cdot\omega_{\bf j}.$$
We filter the de Rham complex $\Omega_S\kr$
by subcomplexes
    $$F^r\Omega_{S}\kr :=    F^rS \to F^{r-1}\Omega^1_{S}\to F^{r-2}\Omega^{2}_{S} \to\ldots   $$
We define the complex ${\rm Kum}(S,r)$ as
    $$
    {\rm Kum}(S,r):=[\xymatrix@C=1.6cm{F^r\Omega_{S}\kr\ar[r]^-{p^r-p\kr\phi_{\rm Kum}}&\Omega_{S'}\kr}].
    $$
We define the {\em syntomic complex} ${\rm Syn}(R,r)$ of $R$ and the 
{\it syntomic cohomology} $H^*_{\synt}(R,r)$ of $R$ as:
  $${\rm Syn}(R,r):={\rm Kum}(R_\varpi^{\rm PD},r)\quad{\rm and}\quad H^*_{\synt}(R,r):=H^*({\rm Syn}(R,r)).$$ 
    For $n\in\N$, we define the  syntomic complexes modulo $p^n$ and syntomic cohomology modulo $p^n$ of $R$ as  
$${\rm Syn}(R,r)_n:={\rm Syn}(R,r)\otimes_{\Z}\Z/p^n, \quad {\rm and}\quad H^*_{\synt}(R_n,r):=H^*({\rm Syn}(R,r)_n).$$

If $S=R_\varpi^{[u]},R_\varpi^{[u,v]}$, set $S''=R_\varpi^{[pu]},R_\varpi^{[pu,v]}$.  We define the complex
${\rm Kum}^\psi(S,r)$, using $\psi_{\rm Kum}$ instead of $\varphi_{\rm Kum}$, as:
    $$
    {\rm Kum}^{\psi}(S,r):=[\xymatrix@C=1.6cm{F^r\Omega_{S}\kr\ar[r]^-{p^r\psi_{\rm Kum}-p\kr}&\Omega_{S''}\kr}].
    $$

Finally, if $S=R_\varpi^{[u,v]}$, we define ${\rm Cycl}(S,r)$ using $\varphi_{\rm cycl}$
in place of $\varphi_{\rm Kum}$ (we cannot define
${\rm Cycl}(S,r)$ for $S=R_\varpi^{\rm PD},R_\varpi^{[u]}$ as $\varphi_{\rm cycl}$
does not send $S$ into $S'$ in these cases (except if $K=F$, i.e~if $K$ is absolutely
unramified)):
$${\rm Cycl}(S,r):=[\xymatrix@C=1.6cm{F^r\Omega_{S}\kr\ar[r]^-{p^r-p\kr\phi_{\rm cycl}}&\Omega_{S'}\kr}].$$

\medskip
{\it For the remainder of this section, we write simply $\varphi$ and $\psi$ for
$\varphi_{\rm Kum}$ and $\psi_{\rm Kum}$.}
     
\subsection{Change of disk of convergence}
   We are going to show that ${\rm Kum}(R_\varpi^{[u]},r)$ computes the syntomic cohomology up to a constant depending on $r$.
For that we will need  
Lemma~\ref{6saint} below.
For $ S=r_\varpi^{\rm PD}$ or $r_\varpi^{[u]}$, 
we denote by  
$v_{X_0}: S\to \N\cup\{+\infty\}$ the valuation relative to 
${X_0}$: if  $f=\sum a_i{{X_0}^i}$ then 
$v_{X_0}(f)=\inf \{i\in\N,\ a_i\neq 0\}$ and, if $N\in\N$, we define
$S_N$ as $\{f\in S,\ v_{X_0}(f)\geq N\}$.
We define $R_{\varpi,N}^{\rm PD}$ and $R_{\varpi,N}^{[u]}$
as the topological closures of
$r_{\varpi,N}^{\rm PD}\otimes_{r_\varpi^+} R_\varpi^+$ and $R_{\varpi,N}^{[u]}\otimes_{r_\varpi^+} R_\varpi^+$
 in $R_{\varpi,N}^{\rm PD}$ and $R_{\varpi,N}^{[u]}$.
\begin{lemma}\label{6saint}
Let $S=R_{\varpi}^{\rm PD},R_{\varpi}^{[u]}$.
If  $s\in\Z$ and $N\geq 1$, then
$(1-p^{-s}\varphi): S_N[\frac{1}{p}]\to  S_N[\frac{1}{p}]$ is bijective.
Moreover, if $ S=R_\varpi^{\rm PD}$ {\rm(}resp.~$ S=R_\varpi^{[u]}${\rm)},

$\bullet$ there exists $N(r,e)$ {\rm(}resp.~$N(r,e,u)${\rm)} such that $(1-p^{-s}\varphi)$ is
$p^{N(r,e)}$-bijective {\rm(}resp.~$p^{N(r,e,u)}$-bijective{\rm)} on $S_1$ for all $s\leq r$,

$\bullet$  if  $N\geq{se}$, {\rm(}resp.~$N\geq{se/u(p-1)}${\rm)}, 
then  $1-p^{-s}\varphi$ is bijective on $S_N$.
\end{lemma}
\begin{proof}
The estimates being easier for $R_\varpi^{[u]}$ (since $[\frac{i}{e}]!$ is replaced
by $[\frac{iu}{e}]$ which behaves better as a function of $i$),
 we will just treat the case  $ S=R_\varpi^{\rm PD}$.
An element $f$ of $S_N$ can be written as
$\sum_{i\geq N}f_i\frac{X_0^i}{[\frac{i}{e}]!}$, where $f_i\in R_\varpi^+$ goes to $0$
when $i\to\infty$. We have then
$$p^{-ks}\varphi^k(f)=\sum_{i\geq N}p^{-ks}\frac{[\frac{p^{k}i}{e}]!}{[\frac{i}{e}]!}
\varphi^k(f_i)
\frac{X_0^{p^ki}}{[\frac{p^ki}{e}]!}.$$
Now,
$$\inf_{i\geq 1}\big(v_p\big([\frac{p^{k}i}{e}]!\big)-v_p\big([\frac{i}{e}]!\big)-ks\big)=
v_p\big([\frac{p^{k}}{e}]!\big)-ks\geq \frac{p^{k-1}}{e}-1-ks$$
goes  to $+\infty$ when $k\to+\infty$.  We deduce that, for $f\in  S_N[\frac{1}{p}]$, the series
 $\sum_{k\in\N}p^{-ks}{\varphi^k}(f)$ converges in $ S_N[\frac{1}{p}]$.
Since the sum  $g$ satisfies $(1-p^{-s}\varphi)g=f$, we deduce that 
$(1-p^{-s}\varphi)$ is invertible, with inverse $\sum_{k\in\N}p^{-ks}{\varphi^k}$,
and we can take $N(r,e)=\inf_{k\in\N}(-1-kr+p^{k-1}/e)$.

Now, if $i\geq se$, then
$$v_p\big([\frac{p^{k}i}{e}]!\big)-v_p\big([\frac{i}{e}]!\big)=
\sum_{j=0}^{k-1}\big[\frac{p^{j}i}{e}\big]
\geq \sum_{j=0}^{k-1}sp^{j}\geq sk.$$
This implies the last statement of the lemma.
\end{proof}

\begin{lemma}
Let
$r,s\in {\mathbf N}$.

{\rm (i)} If $1/(p-1)\leq u\leq 1$, the map $p^s-\varphi$ induces a $p^{s+r}$-isomorphism 
$F^r\Omega_{R_\varpi^{[u]}}^i/F^r\Omega_{R_\varpi^{\rm PD}}^i\simeq \Omega_{R_\varpi^{[u]}}^i/\Omega_{R_\varpi^{\rm PD}}^i$.

{\rm (ii)} If $u'\leq u\leq pu'$, the map $p^s-\varphi$ induces a $p^{s+r}$-isomorphism
$F^r\Omega_{R_\varpi^{[u]}}^i/F^r\Omega_{R_\varpi^{[u']}}^i\simeq \Omega_{R_\varpi^{[u]}}^i/\Omega_{R_\varpi^{[u']}}^i$.
\end{lemma}
\begin{proof}
If ${\bf j}\in J_i$, we have $\varphi(\omega_{\bf j})=p^j\omega_{\bf j}$.  This makes it possible, by decomposing everything
in the basis of the $\omega_{\bf j}$'s, to only treat the case $i=0$ (the map becoming $p^s-p^i\varphi$).
Let $A=R_\varpi^{\rm PD}$ or $R_\varpi^{[u']}$ and $B=R_\varpi^{[u]}$; we have $A\subset B$
and $\varphi(B)\leq A$.

To show $p^s$-injectivity, since $F^rA=A\cap F^rB$, 
it suffices to show that $(p^s-p^i\varphi)x\subset A$ implies that $x\in A$. 
But this follows from the fact that $\varphi(B)\subset A$ and the identity $p^sx=(p^s-p^i\varphi)x+p^i\varphi(x)$.

Now let $f\in B$.  Write $f$ as $f_1+f_2$ with $f_2\in R_{\varpi,se/u(p-1)}^{[u]}$ and 
$f_1\in p^{-[u (se/u)/e(p-1)]}R_\varpi^+\subset p^{-s}A$.
By Lemma~\ref{6saint}, we can write $f_2$ as $(1-p^{i-s}\varphi)g$, with $g\in B$.
By Remark~\ref{0wiesia}, we can write $g=g_1+g_2$ with $g_1\in F^rB$ and
$g_2\in p^{-[ru]}R_\varpi^+$, hence $(1-p^{i-s}\varphi)p^{-s}g_2\in p^{-s-r}A$
and $f-(1-p^{i-s}\varphi)g_1\in p^{-s-r}A$.
Finally, we obtain $f\in p^{-s-r}A+p^{-s}(p^s-p^i\varphi)F^rB$, which allows to conclude. 
 \end{proof}
 
\begin{proposition}\label{8saint}
{\rm (i)} 
For $\frac{1}{p-1}\leq u\leq 1$, 
the morphism of complexes ${\rm Kum}(R_\varpi^{\rm PD},r)\to {\rm Kum}(R_\varpi^{[u]},r)$
induced by the 
injection $R_\varpi^{\rm PD}\subset R_\varpi^{[u]}$
is a $p^{6r}$-quasi-isomorphism.

{\rm (ii)} If $u'\leq u\leq pu'$,
the morphism of complexes ${\rm Kum}(R_\varpi^{[u']},r)\to {\rm Kum}(R_\varpi^{[u]},r)$
induced by the
injection $R_\varpi^{[u']}\subset R_\varpi^{[u]}$
is a $p^{6r}$-quasi-isomorphism.
\end{proposition}
\begin{proof} The above lemma, applied  with $s=r-1$ and  $s=r$, allows to show that
the cohomology of the quotient complexes is annihilated by $p^{6r}$.
\end{proof}
\subsection{$(\varphi,\partial)$-modules and $(\psi,\partial)$-modules}
We are going to show that ${\rm Kum}^\psi(R_\varpi^{[u]},r)$ computes syntomic cohomology.

\begin{lemma}\label{uphipsi}
The following commutative diagram 
$$\xymatrix{
F^r \Omega\kr_{R_\varpi^{[u]}}\ar[r]^-{p^r-p\kr\varphi}\ar[d]^{{\id}}&\Omega\kr_{R_\varpi^{[u]}}\ar[d]^{\psi}\\
F^r \Omega\kr_{R_\varpi^{[u]}}\ar[r]^-{p^r\psi-p\kr}&\Omega\kr_{R_\varpi^{[pu]}}
}$$
defines a quasi-isomorphism from ${\rm Kum}(R_\varpi^{[u]},r)$ to  ${\rm Kum}^\psi(R_\varpi^{[u]},r)$. 
\end{lemma}
\begin{proof}
Let $S=R_\varpi^{[u]}$.
Since $\psi:\Omega\kr_{S}\to \Omega\kr_{R_\varpi^{[pu]}}$ is surjective,
the  quotient complex is  trivial.  It suffices thus to show that the kernel complex 
$$0\to S^{\psi=0}\to (\Omega^1_{S})^{\psi=0}\to (\Omega^2_{S})^{\psi=0}\to \ldots $$
is acyclic. 
We have
$$
\Omega_{S}^{j,\psi=0} =S^{\psi=0}\otimes\Omega^j,
$$
where
$$\Omega^1=\oplus_{1\leq j\leq d+1}\Z\,\tfrac{dX_j}{X_j}\quad{\rm and}\quad \Omega^i=\wedge^i\Omega^1.$$

Now, $S^{\psi=0}=\oplus_{\alpha\in\{0,\cdots,p-1\}^{d+1}, \alpha\neq 0} S_\alpha$, and this decomposition
induces a direct sum decomposition of the above de Rham complex, so we can
argue for every $\alpha$ separately that the following de Rham complex
$$0\to S_{\alpha}\to S_{\alpha}\otimes\Omega^1\to S_{\alpha}\otimes\Omega^2\to \ldots $$
is exact, and it is enough to prove this modulo~$p$.  But Lemma~\ref{prepsi} tells us that this complex
has a very simple shape modulo~$p$: if $d=0$, it is just $\xymatrix{S_\alpha\ar[r]^{\alpha_0}&S_\alpha}$,
if $d=1$, it is the total complex attached to the double complex
$$\xymatrix{S_\alpha\ar[r]^{\alpha_0}\ar[d]^{\alpha_1}&S_\alpha\ar[d]^{\alpha_1}\\
S_\alpha\ar[r]^{\alpha_0}&S_\alpha}$$ 
and, for general $d$, it is the total complex attached to a $(d+1)$-dimensional cube,
with all vertices equal to $S_\alpha$ and arrows in the $i$-th direction equal to $\alpha_i$.
As one of the $\alpha_i$ is invertible by assumption, this implies that the cohomology of the
total complex is $0$.
\end{proof}
\subsection{Change of annulus of convergence}
\begin{lemma}
\label{15saint} If $u\leq 1\leq v$,
the natural morphism  ${\rm F}^rR_\varpi^{[u,v]}/{\rm F}^rR_\varpi^{[u]}\to
R_\varpi^{[u,v]}/R_\varpi^{[u]}$ is a $p^{r}$-isomorphism.
\end{lemma}
\begin{proof}
It suffices to treat the case when $R=\so_K$. The above map is clearly injective. To prove $p^r$-surjectivity,
we need to verify that 
$p^{r+[kv]}X_0^{-ke}$ is in the image. For that, note that 
$p$ being divisible by $\varpi ^e$ in  $\so_K$, we can write
 $p=X_0^eA+BP$, with $A,B\in r_\varpi^+$.
This implies that we can write $p^{r+[kv]}$ as 
$B_kP^r+A_kX_0^{e[kv]}$, and  $p^{r+[kv]}X_0^{-ke}$
as  $A_kX_0^{([kv]-k)e}+X_0^{-ke}B_kP^r$.
But $A_kX_0^{([kv]-k)e}\in r_\varpi^+\subset r_\varpi^{[u]}\subset r_\varpi^{[u,v]}$, 
hence $X_0^{-ke}B_kP^r\in r_\varpi^{[u,v]}$,
and its image in  $r_\varpi^{[u,v]}[\frac{1}{p}]$ being divisible by  $P^r$, we have written
 $p^{r+[kv]}X_0^{-ke}$ as a sum of  an element from  $F^rr_\varpi^{[u,v]}$
and an element from  $r_\varpi^{[u]}$.  This allows us to conclude.
\end{proof}

\begin{lemma}\label{16saint}
Let  $u\leq 1\leq v$. The natural morphism 
from  ${\rm Kum}^\psi(R_\varpi^{[u]},r)$ to  ${\rm Kum}^\psi(R_\varpi^{[u,v]},r)$ induces a $p^{2r}$-quasi-isomorphism
$$\tau_{\leq r}{\rm Kum}^\psi(R_\varpi^{[u]},r)\stackrel{\sim}{\to}\tau_{\leq r}{\rm Kum}^\psi(R_\varpi^{[u,v]},r).$$
The same is true for the complexes modulo $p^n$. 
\end{lemma}
\begin{proof}
Our map is induced by the following diagram
$$
\xymatrix{
F^r\Omega\kr_{R_\varpi^{[u]}}\ar[d]\ar[r]^{p^r\psi-p\kr} & \Omega\kr_{R_\varpi^{[pu]}}\ar[d]\\
F^r\Omega\kr_{R_\varpi^{[u,v]}}\ar[r]^{p^r\psi-p\kr} & \Omega\kr_{R_\varpi^{[pu,v]}}
}
$$
It suffices to show that the mapping fiber
\begin{equation}
\label{modp}
[\xymatrix@C=1.6cm{F^r\Omega\kr_{R_\varpi^{[u,v]}}/F^r\Omega\kr_{R_\varpi^{[u]}}\ar[r]^-{p^r\psi-p\kr} & \Omega\kr_{R_\varpi^{[pu,v]}}/ \Omega\kr_{R_\varpi^{[pu]}}}]
\end{equation}
 is $p^{2r}$-acyclic. By Lemma~\ref{15saint}, we can ignore the filtration and, working in
the basis of the $\omega_{\bf j}$ of $\Omega^i$,  
it is enough to show that $p^r\psi-p^i:R_\varpi^{[u,v]}/R_\varpi^{[u]}\to
R_\varpi^{[pu,v]}/R_\varpi^{[pu]}$ is a $p^r$-isomorphism if $i\leq r$.
But $R_\varpi^{[u,v]}/R_\varpi^{[u]}\simeq
R_\varpi^{[pu,v]}/R_\varpi^{[pu]}$ and 
$1-p^s\psi$, for $s=r-i\geq 0$, is invertible on  $R_\varpi^{[u,v]}/R_\varpi^{[u]}$  with inverse 
$1+p^s\psi+p^{2s}\psi^2+\cdots$ (this converges even if  $s=0$ because $\psi$ is topologically 
nilpotent). This allows us to conclude.

  For complexes modulo $p^n$, since the quotients
  ${\rm F}^rR_\varpi^{[u,v]}/{\rm F}^rR_\varpi^{[u]}\hookrightarrow
R_\varpi^{[u,v]}/R_\varpi^{[u]}  $
  are $p$-torsion free, it again suffices to show that the mod $p^n$ analog of the mapping fiber (\ref{modp}) is $p^{2r}$-acyclic. Note that the mod-$p^n$ version of Lemma \ref{15saint} holds 
(though we now only have a $p^{r}$-injection). Having that the rest of the argument is the same. 
\end{proof}

\begin{remark}
Truncating is not absolutely necessary at this point, as $\psi$ is very
rapidly nilpotent, but the constants that come out involve $\log e$.
\end{remark}
\begin{corollary}\label{17saint}
For  $pu\leq v$, the natural map from  ${\rm Kum}(R_\varpi^{[u]},r)$ to 
$${\rm Kum}(R_\varpi^{[u,v]},r):=
[\xymatrix@C=1.6cm{F^r\Omega\kr_{R_\varpi^{[u,v]}}\ar[r]^-{p^r-p\kr\phi}& \Omega\kr_{R_\varpi^{[u,v/p]}}}]
$$
induces  a  $p^{2c_e+2r}$-quasi-isomorphism
$$\tau_{\leq r}{\rm Kum}(R_\varpi^{[u]},r)\stackrel{\sim}{\to}\tau_{\leq r}{\rm Kum}(R_\varpi^{[u,v]},r)$$
\end{corollary}
\begin{proof}
By Lemma \ref{uphipsi} we can pass from the complex ${\rm Kum}(R_\varpi^{[u]},r)$ to ${\rm Kum}^{\psi}(R_\varpi^{[u]},r)$ and, by the above lemma, we can pass from 
$\tau_{\leq r}{\rm Kum}^{\psi}(R_\varpi^{[u]},r)$ to $\tau_{\leq r}{\rm Kum}^{\psi}(R_\varpi^{[u,v]},r)$.  It remains thus to show that we can pass from ${\rm Kum}^{\psi}(R_\varpi^{[u,v]},r)$ to ${\rm Kum}(R_\varpi^{[u,v]},r)$. Or that the map induced by the following
commutative diagram 
$$\xymatrix{F^r\Omega\kr_{R_\varpi^{[u,v]}}\ar[d]^{\id} \ar[r]^{p^r-p\kr\phi} & \Omega\kr_{R_\varpi^{[u,v/p]}}\ar[d]^{\psi}\\
F^r\Omega\kr_{R_\varpi^{[u,v]}}\ar[r]^{p^r\psi-p\kr} & \Omega\kr_{R_\varpi^{[pu,v]}}
}$$
is a quasi-isomorphism. 
We can now use the  same arguments as in the proof of Lemma~\ref{uphipsi} to conclude.
\end{proof}
\subsection{Change of Frobenius} 
To pass from syntomic cohomology to \'etale cohomology, we will need to change the crystalline Frobenius into the $(\phi,\Gamma)$-module Frobenius. 
\begin{proposition}\label{RtoA}
The complexes ${\rm Kum}(R_\varpi^{[u,v]},r)$ and
${\rm Cycl}(R_\varpi^{[u,v]},r)$ are $2^{2(d+1)}$-quasi-isomorphic.
\end{proposition}
\begin{proof}
If $A\subset A'$ are filtered rings with a Frobenius $\varphi:A\to A'$,
let $${\rm S}(A,r,\varphi)=
[\xymatrix@C=1.3cm{F^r\Omega_{A}\kr\ar[r]^-{p^r-p\kr\phi}&\Omega_{A'}\kr}],$$
where $\varphi$ on $\Omega^n_A$ is the divided Frobenius (by $p^n$).
So, for example,
$${\rm Kum}(R_\varpi^{[u,v]},r)={\rm S}(R_\varpi^{[u,v]},r,\varphi_{\rm Kum})\quad{\rm and}\quad
{\rm Cycl}(R_\varpi^{[u,v]},r)={\rm S}(R_\varpi^{[u,v]},r,\varphi_{\rm cycl}).$$
Since $(\A_R^{[u,v]},\varphi)\cong (R_\varpi^{[u,v]},\varphi_{\rm cycl})$, Lemma~\ref{fat3} provides us with
morphisms of complexes
$${\rm S}(R_\varpi^{[u,v]},r,\varphi_{\rm Kum})\rightarrow {\rm S}(E^{[u,v]}_R,r,\varphi)
\leftarrow {\rm S}(R_\varpi^{[u,v]},r,\varphi_{\rm cycl}).$$
Lemma~\ref{above2} below shows
that these are $2^{d+1}$-quasi-isomorphisms, which allows to conclude.
\end{proof}

Let $R_1,R_2$ be two copies of $R_\varpi^{[u,v]}$: we have isomorphisms
$\iota_i:R_\varpi^{[u,v]}\to R_i$.  We set $X_{i,j}=\iota_i(X_j)$, if $i=1,2$, $0\leq j\leq d$.

Let $R_3=(R_1\otimes R_2)^{\rm PD}$, with respect to $\iota=\iota_2\circ\iota_1^{-1}$.
Hence $R_3\cong E_R^{[u,v]}$ as a ring, without the actions of $G_R$ or $\varphi$.

If $i=1,2$, let $\Omega^1_i=\oplus_{j=0}^d\Z\,\tfrac{dX_{i,j}}{X_{i,j}}$.
Set $\Omega^1_3=\Omega^1_1\oplus\Omega^1_2$ and, if $i=1,2,3$, 
let $\Omega^n_i=\wedge^n\Omega^1_i$.
Then $\Omega^n_{{R_3}}={R_3}\otimes\Omega_3^n$.
We filter the de Rham complex of ${R_3}$ as usual by subcomplexes
$$F^r\Omega_{R_3}\kr:=F^r{R_3}\to F^{r-1}{R_3}\otimes\Omega^1_3\to F^{r-2}{R_3}\otimes\Omega^2_3\to\cdots$$
\begin{lemma}\label{above2}
{\rm (Filtered Poincar\'e Lemma)}
The natural maps
$$F^r\Omega_{R_1}\kr\rightarrow F^r\Omega_{R_3}\kr\leftarrow F^r\Omega_{R_2}\kr$$
are $2^{d+1}$-quasi-isomorphims.
\end{lemma}
\begin{proof}
By symmetry, it is enough to consider the first map.
First, we claim that we have the $\partial_{R_1}$  filtered Poincar\'e Lemma, i.e., that the following sequence 
  $$\xymatrix@C=.6cm{0\ar[r]& F^rR_2\ar[r]&F^r{R_3}\ar[r]^-{\partial_{R_1}}& F^{r-1}{R_3}\otimes\Omega^1_1\ar[r]^-{\partial_{R_1}}&\cdots}
  $$
   is $2$-exact.   Indeed, this is a special case of Lemma~\ref{above}, with $S=R_1$ and $\Lambda=R_2$.

  Now, we can extend the above to a sequence of maps of de Rham complexes
  $$\xymatrix@C=.6cm{0\ar[r]& F^{r}\Omega_{R_2}\kr\ar[r]& 
 F^{r}{R_3}\otimes \Omega_2\kr \ar[r]^-{\partial_{R_1}}&
  F^{r-1}{R_3}\otimes(\Omega^1_1\wedge
  \Omega_2\kr)\ar[r]& \cdots}
  $$
  The contracting homotopy used in the proof of Lemma~\ref{above},
 being $R_2$-linear, extends as well and shows that the rows of the above double complex are $2$-exact. Since the total complex 
  of the double complex 
  $$
 \xymatrix@C=.6cm{F^{r}{R_3}\otimes \Omega_2\kr \ar[r]^-{\partial_{R_1}}& 
  F^{r-1}{R_3}\otimes(\Omega^1_1\wedge
  \Omega_2\kr)\ar[r]& \cdots}
  $$
is equal to the de Rham complex $F^r\Omega_{{R_3}}\kr$, we are done.
\end{proof} 
\subsection{Syntomic cohomology and de Rham cohomology} We will show that, up to some universal constants, syntomic cohomology has a simple relation to de Rham cohomology. 
Let $S=R^{\rm PD}_{\varpi}$, $\phi=\phi_{{\rm Kum}}$, $r\geq 0$. Set 
\begin{align*}
{\rm HK}(S,r):=[\Omega\kr_S\lomapr{p^r-p\kr\phi}\Omega\kr_S],\quad
{\rm DR}(S,r):=\Omega\kr_S/F^r.
\end{align*}
We note that 
$\tau_{\leq r-1}{\rm DR}(S,r)\stackrel{\sim}{\to}{\rm DR}(S,r)$ and that the natural map $\tau_{\leq r+1}{\rm HK}(S,r)\to {\rm HK}(S,r)$ is a $p^{2r}$-quasi-isomorphism (since $1-p^s\phi$, $s\geq 1$, is invertible on $\Omega^{r+s}_S$).
\begin{proposition}
\label{added}
{\rm (i)} The natural map
$$\tau_{\leq r+1}{\rm Syn}(R,r)\to {\rm Syn}(R,r)$$
is a $p^{2r}$-quasi-isomorphism and $H^{r+1}({\rm Syn}(R,r))\stackrel{\sim}{\to}H^{r+1}({\rm HK}(S,r))$.

{\rm (ii)}
The complex $\tau_{\leq r-1}{\rm HK}(S,r)$ is $p^{N}$-acyclic, for a constant $N=N(e,d,p,r)$. Hence the natural map ${\rm DR}(S,r)\to \tau_{\leq r-1}{\rm Syn}(R,r)[1]$ is a $p^{N}$-quasi-isomorphism.

{\rm (iii)}
 The above statements are valid also modulo $p^n$. Moreover, $H^{r+1}({\rm HK}(S,r)_n)$ is, \'etale locally on $R_n$, $p^{N(r)}$-trivial.
\end{proposition}
\begin{proof} For (i) and (ii), we are going to argue $p$-adically: the mod-$p^n$ argument for (iii) is analogous. Since we have 
\begin{align*}
{\rm Syn}(R,r)={\rm Kum}(R^{\rm PD}_{\varpi},r)=[F^r\Omega\kr_S\lomapr{p^r-p\kr\phi}\Omega\kr_S]=
[{\rm HK}(S,r)\to {\rm DR}(S,r)],
\end{align*}
the first claim of our proposition follows from what we noted just before the proposition.

For the second claim, write $X_0S$ for $X_0S[\frac{1}{p}]\cap S$ and define $X_0\Omega\kr_S$ accordingly.
Writing differentials in the basis of the $\omega_{\bf j}$'s and using Lemma~\ref{6saint}, we obtain
that the subcomplex $[X_0\Omega\kr_S\lomapr{p^r-p\kr\phi}X_0\Omega\kr_S]$ is $p^{N_1(e,r,d)}$-acyclic.
Hence ${\rm HK}(S,r)$ is $p^{N_2(e,r,d)}$-quasi-isomorphic to 
$[\Omega\kr_{\overline S}\lomapr{p^r-p\kr\phi}\Omega\kr_{\overline S}]$, where
$\overline S:=S/X_0S\cong R_\varpi^+/X_0$.

Now, the following lemma shows that $\overline S$ is equipped with a left inverse $\psi$ of $\varphi$ verifying $\psi\circ \partial_i=p\,\partial_i\circ\psi$,
if $1\leq i\leq d$.
\begin{lemma}
$\psi=\psi_{\rm Kum}$ stabilizes $X_0R_\varpi^+$. 
\end{lemma}
\begin{proof}
We have $\psi=p^{-d-1}\varphi^{-1}\circ {\rm Tr}$,
with ${\rm Tr}={\rm Tr}_{R_\varpi^+/\varphi(R_\varpi^+)}$,
and the trace can be computed by taking the sum over the conjugates: if $\eta=(\eta_0,\dots,\eta_d)$ is a tuple of
$p$-th roots of unity, we define an automorphism $\sigma_\eta$ of $R_\varpi^+[\zeta_p]$, sending $X_i$ to $\eta_i X_i$
(this gives us $\sigma_\eta$ on $R_{\varpi,\Box}^+[\zeta_p]$, and we extend it to $R_\varpi^+[\zeta_p]$
by \'etaleness).  The $\sigma_\eta(x)$'s are the conjugates of $x$, hence 
${\rm Tr}(x)=\sum_\eta\sigma_\eta(x)$, which makes it plain that,
if $x$ is divisible by $X_0$, then so is ${\rm Tr}(x)$.
Now, if $\varphi(x)\in X_0R_\varpi^+$, then $x\in X_0R_\varpi^+$: this can be seen by successive approximations, 
using the last line of the proof
of Lemma~\ref{ku1} and the fact that $\varphi(X_0R_\varpi^+)\subset X_0R_\varpi^+$.
Hence, if $x\in X_0R_\varpi^+$, then $p^{d+1}\psi(x)\in X_0R_\varpi^+$ and, since $X_0R_\varpi^+[\frac{1}{p}]\cap R_\varpi^+=
X_0R_\varpi^+$, this implies $\psi(x)\in X_0 R_\varpi^+$.
\end{proof}
We can then use the arguments in the proof of Lemma~\ref{uphipsi} 
to deduce that $[\Omega\kr_{\overline S}\lomapr{p^r-p\kr\phi}\Omega\kr_{\overline S}]$
is quasi-isomorphic to
$[\Omega\kr_{\overline S}\lomapr{p^r\psi-p\kr}\Omega\kr_{\overline S}]$.
But, in degree $i\leq r-1$, $p^r\psi-p^i$ is a $p^r$-isomorphism (and even a $p^i$-isomorphism),
with inverse $-1-p^{r-i}\psi-p^{2(r-i)}\psi^2-\cdots$.  Hence $\tau_{\leq r-1} 
[\Omega\kr_{\overline S}\lomapr{p^r\psi-p\kr}\Omega\kr_{\overline S}]$ is $p^r$-acyclic,
which proves (ii).

  For the very last claim of the proposition, since the map $\Omega^{r+1}_S\lomapr{1-p\phi}\Omega^{r+1}_{S}$ 
is injective, it suffices to show that the map $\Omega^{r}_{S_n}\lomapr{1-\phi}\Omega^{r}_{S_n}$ is surjective,
which amounts to showing that $1-\varphi$ is, \'etale locally,  surjective on $S_n$.  By d\'evissage
we can assume $n=1$, and then we are reduced to the statement that $x-x^p$ is, \'etale locally, a surjection,
which is clear as Artin-Schreier extensions are \'etale.
\end{proof}

  The following lemma relates the complex ${\rm DR}(S,r)$ which appears in (ii) of the above proposition
to the usual de Rham complex.
\begin{lemma}\label{noneedbeilinson}
${\rm DR}(S,r)$ is quasi-isomorphic to $\sigma_{\leq r-1}\Omega\kr_R$ after inverting $p$. Here $\sigma_{\leq r-1}$ denotes the silly truncation.
\end{lemma}
\begin{proof}
We have $r_\varpi^+[\frac{1}{p}]/F^r\cong K[P]/P^r$.  Hence
$S[1/p]/F^r\cong R\otimes_{\O_K}(K[P]/P^r)$ (we see that $R_{\varpi,\Box}^{\rm PD}/F^r\cong R_\Box\otimes(K[P]/P^r)$
by inspection and deduce the result for $S$ by \'etaleness).

Writing differentials in the basis of the $\omega_{\bf j}$'s  makes it possible to write
${\rm DR}(S,r)$ as the total complex of the double complex
$$\xymatrix@C=.6cm@R=.6cm{
R\otimes(K[P]/P^r)\ar[d]^{\partial_0}\ar[rr]^-{(\partial_j)_{1\leq j\leq d}} &&\big(R\otimes(K[P]/P^{r-1})\big)^d
\ar[d]^{\partial_0}\ar[r] &\cdots\ar[r] &R^{\binom{d}{r-1}}\ar[d]^{\partial_0}\ar[r]&0\ar[r]\ar[d]&\cdots\\
R\otimes(K[P]/P^{r-1})\ar[rr]^-{(\partial_j)_{1\leq j\leq d}} &&\big(R\otimes(K[P]/P^{r-2})\big)^d
\ar[r] &\cdots\ar[r] &0\ar[r]&0\ar[r]&\cdots}$$
Now, $\partial_0$ induces an isomorphism $PK[P]/P^r\cong K[P]/P^{r-1}$.
Hence the subcomplex obtained by replacing $K[P]/P^s$ by $PK[P]/P^s$ in the first row is acyclic,
and the quotient complex is clearly isomorphic to $\sigma_{\leq r-1}\Omega\kr_R$.
\end{proof}
\begin{remark}\label{avoidpanick}
It follows from Lemma~\ref{noneedbeilinson} that, if $i\leq r-2$,
then $H^i({\rm DR}(S,r))=H^i_{\rm dR}({\rm Spf}\,R)$
after inverting $p$.  On the other hand, $H^{r-1}({\rm DR}(S,r))$ is much
bigger than $H^{r-1}_{\rm dR}({\rm Spf}\,R)$, since {\it we take all differential forms} modulo exact ones
instead of closed ones. For example, if $r=1$, $r-1=0$, and  $H^0({\rm DR}(S,r))$ is actually $R$, and not just the constants.
\end{remark}

 Let $\sx$ be a  semistable formal scheme over $\so_K$ (i.e., locally of the form described in section 2.1.2 with $h=1$). 
 Set  $\sx_{K,\tr}=\sx_K^{\an}\setminus \sd$, $\sd$ being the divisor at infinity, and $\sx^{\an}_K$ -- the rigid analytic space associated to $\sx$. Define
 $$\R\Gamma_{\dr}(\sx_K,\log\sd):=\R\Gamma(\sx_K^{\an},\Omega\kr_{\sx_K^{\an}}(\log \sd)),$$
 where 
  $\Omega\kr_{\sx^{\an}_K}(\log \sd)$ denotes  the (logarithmic) de Rham complex of $\sx^{\an}_K$. If $\sx$ is quasi-compact, we have
  $$
  \R\Gamma_{\rm dR}(\sx_{K},\log\sd)\simeq\R\Gamma(\sx,\Omega^{\kr}_{\sx/\so_K^{\times}})_{\Q}:=(\holim_n \R\Gamma(\sx_n,\Omega^{\kr}_{\sx_n/\so_{K,n}^{\times}}))_{\Q}.
    $$
  \begin{corollary}
\label{derham}Let $\sx$ be a quasi-compact formal semistable scheme over $\so_K$. 
There is a natural map
$$
\alpha_{r,i}:H^{i-1}_{\rm dR}(\sx_{K,\tr}){\to} H^i_{\synt}(\sx,r)_{\Q}.
$$
It is an isomorphism for $i\leq r-1$ and an injection for $i=r$.
\end{corollary}
\begin{proof}
The very definition of syntomic cohomolgy gives us a natural map
$$
(\holim_n \R\Gamma_{\crr}(\sx_n,\so_{\sx_n/W_n(k)}/\sj^{[r]}_{\sx_n/W_n(k)}))_{\Q}\to H^i_{\synt}(\sx,r)_{\Q}
$$
from crystalline cohomology to syntomic cohomology. 
We claim that  the domain of this map can be identified with analytic de Rham cohomology. Indeed, we have a natural map from crystalline cohomology to analytic de Rham cohomology
\begin{align*}
(\holim_n \R\Gamma_{\crr}(\sx_n,\so_{\sx_n/W_n(k)}/\sj^{[r]}_{\sx_n/W_n(k)}))_{\Q} & \to (\holim_n \R\Gamma(\sx_n,\Omega^{\kr}_{\sx_n/\so_{K,n}^{\times}}/F^r))_{\Q}\\
 & \simeq \R\Gamma_{\rm dR}(\sx_{K},\log\sd)/F^r \to \R\Gamma_{\rm dR}(\sx_{K,\tr})/F^r
 \end{align*}
   This is a quasi-isomorphism: the last map is an quasi-isomorphism by  Lemma \ref{log-comp} below;
the first map is a quasi-isomorphism since 
 it suffices to argue locally and there we can use Lemma~\ref{noneedbeilinson}.

  Hence we have a natural map
  $$H^{i-1}_{\rm dR}(\sx_{K,\tr}){\to} H^i_{\synt}(\sx,r)_{\Q}, \quad i\leq r.
  $$
   By 
  point (ii) of Proposition~\ref{added}  it
is an isomorphism  for $i< r$ and   an injection for $i=r$.
\end{proof}
\begin{lemma}
\label{log-comp}
For a  semistable formal scheme $\sx$ over $\so_K$, the natural morphism
$$
\R\Gamma_{\rm dR}(\sx_{K},\log\sd)\to \R\Gamma_{\rm dR}(\sx_{K,\tr})$$
is a filtered quasi-isomorphism.
\end{lemma}
\begin{proof}
Let $\Omega\kr_{\sx_K}(*\sd)$ be the complex of meromorphic (along $\sd$) differentials \cite[p.17]{Kih}.  There is  a natural morphism
  $$\Omega\kr_{\sx^{\an}_K}(*\sd) \to j_*\Omega\kr_{\sx_{K,{\tr}}}. $$
  By \cite[Theorem 2.3]{Kih}, this is a quasi-isomorphism: after some preliminary reductions, this amounts to showing that the usual integration works on essentially singular differentials. Since 
  $j:\sx_{K,{\tr}}\hookrightarrow \sx^{\an}_{K}$ is quasi-Stein, we have $\R j_*\Omega\kr_{\sx_{K,{\tr}}}=j_*\Omega\kr_{\sx_{K,{\tr}}}$.

 We also have a natural morphism
$$r: \Omega\kr_{\sx^{\an}_K}(\log \sd)\to \Omega\kr_{\sx^{\an}_K}(*\sd).
$$
Integration causes no convergence issues here. In fact, the computations in \cite[p. 18]{Kih} go through in this case: there is a residue map $\Res: \Omega\kr_{\sx^{\an}_K}(*\sd)\to
\Omega\kr_{\sx^{\an}_K}(\log \sd)$ such that $\Res r=\id$. Integrating gives a homotopy between $r \Res$ and $\id$. 
Our lemma follows. 
\end{proof}

\begin{remark}One can compute syntomic cohomology of $R=\so_K^{\times}$ rather explicitly, and the
result shows that $\alpha_{r,r}$ is not necessarily surjective: in the case $r=1$ below,
there is an extra $\Z_p$ appearing in $H^1_{\synt}(\O^{\times}_K,r)$.
In higher dimensions the difference is much more serious (see Remark~\ref{avoidpanick}).
\end{remark}
Denote by $\so_K^{(r)}$ the intersection of $\so_K$ with $R^{\rm PD}_{\varpi}+P^rF[X_0]_P$. 
We have $\so_K^{(1)}=\so_K$.  We will state the following proposition without a proof.
\begin{proposition}
{\rm (i)} For $r\geq 2$, we have
$$\begin{cases}
 H^0_{\synt}(\so^{\times}_K,r) \mbox{ is } p^{13r} \mbox{-isomorphic to } 0,\\
  H^1_{\synt}(\so^{\times}_K,r) \mbox{ is } p^{N(r,e)}\mbox{-isomorphic to }\so_K^{(r)},\\
    H^2_{\synt}(\so^{\times}_K,r) \mbox { is } p^{N(r,e)}\mbox{-isomorphic to } 0.
\end{cases}
$$
{\rm (ii)} 
For $r=1$, we have $p^{N(r,e)}$-isomorphisms 
$$\begin{cases} H^0_{\synt}(\so^{\times}_K,r)\cong 0,\\
H^1_{\synt}(\so^{\times}_K,r)\cong \so_K^{(r)}\oplus\Z_p\\
H^2_{\synt}(\so^{\times}_K,r)\cong  (\O_F/(\varphi-1)).
\end{cases} 
$$
\end{proposition}
\section{Syntomic cohomology  and $(\phi,\Gamma)$-modules}
Assume that $K$ has enough roots of unity. Let $u=(p-1)/p,$ $v=p-1$ if $p\geq 3$, and $u=\frac{3}{4}$, $v=\frac{3}{2}$ if $p=2$. 
In particular, $\big(1+\frac{p^{i-1}-\delta_R-1}{e_0}\big)u>\big(1+\frac{1}{p-1}-\frac{1}{2p}\big)u
>\frac{1}{p-1}$, for all $p$.

We use the isomorphism $R_\varpi^{[u,v]}\cong \A_R^{[u,v]}$ of filtered rings
with a Frobenius to add an action of $\Gamma_R$ on the complexes that we consider.
This allows us to relate the complex
${\rm Cycl}(R_\varpi^{[u,v]},r)$, that we showed to be $p^{Nr}$-quasi-isomorphic, for a universal constant $N$, 
to our original syntomic complex, to complexes coming from the theory
of $(\varphi,\Gamma)$-modules and known to compute Galois cohomology.
In dimension~$0$, this amounts to relating the complex
$${\rm Kos}(\varphi,\partial,F^r\A_K^{[u,v]}):\quad
\xymatrix@C=1.8cm{F^r\A_K^{[u,v]}\ar[r]^-{(\partial,p^r-\phi)}
&F^{r-1}\A_K^{[u,v]}\oplus \A_K^{[u,v/p]}\ar[r]^-{-(p^r-p\phi)+\partial}&\A_K^{[u,v/p]}}\,,
$$
that we obtain from ${\rm Cycl}(r_\varpi^{[u,v]},r)$,
to the complex of Herr's thesis~\cite{Hr}:
$${\rm Kos}(\varphi,\Gamma_K,\A_K(r)):\quad
\xymatrix@C=1.8cm{\A_K(r)\ar[r]^-{(\tau_0,1-\phi)}
&\A_K(r)\oplus \A_K(r)\ar[r]^-{-(1-\phi)+\tau_0}&\A_K(r)}\,,
$$
where $\tau_0=\gamma_0-1$, and $\gamma_0$ is our topological generator of $\Gamma_K$.

There are two main steps:

--- One gets rid of the filtration by dividing by suitable powers of $t$; this turns
the differential $\partial$ into the action of the Lie algebra of $\Gamma_R$ (Lemma~\ref{ref1}). This looks similar to the construction in \cite{BMS}, where  $[\varepsilon]-1$ appears in the place of $t$ (note that $t/([\varepsilon]-1)$ is a unit in $\A_{\crr}$).

--- One uses the analyticity of the action to pass from the Lie algebra of $\Gamma_R$
to $\Gamma_R$ itself (Lemma~\ref{ref2}).

The rest is standard $(\varphi,\Gamma)$-modules techniques to change the domains
of convergence and move from $\A_R^{[u,v]}$ to $\A_R$.

\subsection{From complexes of differential forms to Koszul complexes}\label{moved}
To start, we are going to transform ${\rm Cycl}(R_{\varpi}^{[u,v]},r)$
into a ``Koszul complex'' by expressing the differentials
$d:F^{r-i}\Omega^i_S\to F^{r-i-1}\Omega^{i+1}_S$ in suitable basis.
Namely, we define $\omega_0=\frac{dT}{1+T}$ and $\omega_j=\frac{dX_j}{X_j}$ if $1\leq j\leq d$.
If 
$${\bf j}=(j_1,\dots,j_i)\in J_i=\{0\leq j_1<\cdots<j_i\leq d\},\quad{\text{ we set
$\omega_{\bf j}=\omega_{j_1}\wedge\dots\wedge\omega_{j_i}$,}}$$
 so that an element of
$F^a\Omega^i_S$ can be written uniquely as
$\sum_{{\bf j}\in J_i}x_{\bf j}\omega_{\bf j}$, with $x_{\bf j}\in F^aS$.
In other words the $\omega_{\bf j}$, for ${\bf j}\in J_i$, form a ``basis''
of $F^a\Omega^i_S$ over $F^aS$; these are the basis that we use to describe $d$.
In this way, $d$ becomes a map, involving the differential operators $\partial_j$, for $0\leq j\leq d$,
from $(F^{r-i}S)^{J_i}$ to $(F^{r-i-1}S)^{J_{i+1}}$.
We denote by ${\rm Kos}(\partial,F^rS)$ the complex
$F^r\Omega\kr_S$ expressed in these basis and ${\rm Kos}(\varphi,\partial,F^rR_\varpi^{[u,v]})$
the complex ${\rm Cycl}(R_\varpi^{[u,v]},r)$ written in the same way.
By definition, 
\begin{align*}
{\rm Kos}(\partial,F^rS)=&\ F^rS\stackrel{(\partial_j)}{\to}(F^{r-1}S)^{J_1}\to (F^{r-2}S)^{J_2}\to \cdots\\
{\rm Kos}(\varphi,\partial,F^rR_\varpi^{[u,v]})=&\ 
[\xymatrix@C=1.6cm{{\rm Kos}(\partial,F^rR_\varpi^{[u,v]})\ar[r]^-{p^r-p\kr\phi_{\rm cycl}}&
{\rm Kos}(\partial,R_\varpi^{[u,v/p]})}].
\end{align*}
and
$${\rm Kos}(\varphi,\partial,F^rR_\varpi^{[u,v]})\simeq {\rm Cycl}(R_\varpi^{[u,v]},r).$$

The arithmetic and the geometric variables behave quite differently in what follows,
so it is convenient to separate them.
Let $J'_i\subset J_i$ be the set of ${\bf j}$'s with $j_1\neq 0$,
hence 
$$J'_i=\{(j_1,\dots,j_d),\ 1\leq j_1<\cdots<j_d\leq d\},$$
and let $\partial'=(\partial_1,\dots,\partial_d)$.
We denote by ${\rm Kos}(\partial',F^rS)$ the subcomplex of ${\rm Kos}(\partial,F^rS)$
made of the $(F^{r-i}S)^{J'_i}$.
Then
${\rm Kos}(\partial,F^rS)$ is the homotopy limit\footnote{Strictly speaking, we should multiply $\partial_0$ by $p$ to
have it defined, but we will ignore this as it does not change anything in the arguments
that follow.}
$${\rm Kos}(\partial,F^rS)=
\big[\xymatrix{{\rm Kos}(\partial',F^rS)\ar[r]^-{\partial_0}&
{\rm Kos}(\partial',F^rS)}\big],$$
and ${\rm Kos}(\varphi,\partial,F^rR_\varpi^{[u,v]})$
is the homotopy limit
$$
{\rm Kos}(\varphi,\partial,F^rR_\varpi^{[u,v]})=\left[\begin{aligned}\xymatrix@C=2cm{
{\rm Kos}(\partial',F^rR_\varpi^{[u,v]})
\ar[d]^{\partial_0}\ar[r]^{p^r-p\kr\phi} & {\rm Kos}(\partial',R_\varpi^{[u,v/p]})\ar[d]^{\partial_0}\\
{\rm Kos}(\partial',F^{r-1}R_\varpi^{[u,v]})
\ar[r]^{p^r-p^{\jcdot+1}\phi} &{\rm Kos}(\partial',R_\varpi^{[u,v/p]}) 
}\end{aligned}\right]
$$

We can now use the isomorphisms $\iota_{\rm cycl}:R_\varpi^{[u,v]}\simeq \A_R^{[u,v]}$
and $\iota_{\rm cycl}:R_\varpi^{[u,v/p]}\simeq \A_R^{[u,v/p]}$, which commute with Frobenius and filtration,
to obtain a (tautological) quasi-isomorphism
$$
{\rm Kos}(\varphi,\partial,F^rR_\varpi^{[u,v]})
\simeq{\rm Kos}(\varphi,\partial,F^r\A_R^{[u,v]})=\left[\begin{aligned}\xymatrix@C=2cm{
{\rm Kos}(\partial',F^r\A_R^{[u,v]})
\ar[d]^{\partial_0}\ar[r]^{p^r-p\kr\phi} & {\rm Kos}(\partial',\A_R^{[u,v/p]})\ar[d]^{\partial_0}\\
{\rm Kos}(\partial',F^{r-1}\A_R^{[u,v]})
\ar[r]^{p^r-p^{\jcdot+1}\phi} &{\rm Kos}(\partial',\A_R^{[u,v/p]})
}\end{aligned}\right]
$$
What we have achieved by this procedure is twofold:

\quad$\bullet$ we moved to the world of period rings,

\quad$\bullet$ we gained an action of $\Gamma_R$ whose infinitesimal action is
related to the differentials $\partial$ by the very useful formula 
$$\nabla_j:=\log\gamma_j=t\partial_j,\quad{\text{for $0\leq j\leq d$.}}$$
  
\subsection{Continuous group cohomology and Koszul complexes}\label{CGC}
  Before we proceed, we will make a little digression on Koszul complexes (we will need explicit  formulas). Consider the Iwasawa algebra $S=\Z_p[[\tau_1,\ldots,\tau_d]]$. The Koszul complex associated to $(\tau_1,\cdots, \tau_d)$ is the following complex \begin{align*}
K(\tau_1,\cdots,\tau_d) &: =K(\tau_1)\wh{\otimes}_{\Z_p}K(\tau_2)\wh{\otimes}_{\Z_p}\cdots \wh{\otimes}_{\Z_p}K(\tau_d),\\
K(\tau_i) & :=(0\to \Z_p[[\tau_i]]\stackrel{\tau_i}{\to}\Z_p[[\tau_i]]\to 0)
\end{align*}
Here the right hand term is placed in degree $0$. 
Degree $q$ of this complex equals the exterior power $\bigwedge_{S} ^q S^d$. In the standard basis $\{e_{i_1\cdots i_q}\}, 1\leq i_1 <\cdots <i_q\leq d,$
of $\bigwedge_S ^q S^d$ the differential $d^1_q:\bigwedge_{S} ^q S^d\to \bigwedge_{S} ^{q-1} S^d $ is given by the formula
\begin{equation}
\label{formula}
d^1_q(a_{i_1\cdots i_q})=\sum_{k=1}^{q}(-1)^{k+1}a_{i_1\cdots \widehat{i_k}\cdots i_q}\tau_{i_k}
\end{equation}
The augmentation map $S\to \Z_p$ makes $K(\tau_1,\ldots,\tau_d)$ into a resolution of $\Z_p$ in the category of topological $S$-modules.

 Let $\Z_p[[\Gamma_R^{\prime}]]$ denote the Iwasawa algebra of $\Gamma_R^{\prime}$, i.e., the completed group ring 
$$
\Z_p[[\Gamma_R^{\prime}]]:=\invlim\Z_p[\Gamma_R^{\prime}/H],
$$
where the limit is taken over all the open normal subgroups $H$  of $\Gamma_R^{\prime}$ and every group ring $\Z_p[\Gamma_R^{\prime}/H]$ is equipped with the $p$-adic topology. We have $\Z_p[[\Gamma_R^{\prime}]]\simeq \Z_p[[\tau_1,\cdots,\tau_d]]$, $\tau_j:=\gamma_{j}-1$, $j\in \{1,\cdots, d\}$. The Koszul complex $K(\tau_1,\ldots,\tau_d)$ is the complex
$$
\xymatrix{0\ar[r]& \Z_p[[\Gamma_R^{\prime}]]^{J'_d}\ar[r]^-{d^1_{d-1}}&\cdots\cdots
\ar[r]^-{d^1_1} &\Z_p[[\Gamma_R^{\prime}]]^{J'_1}\ar[r]^-{d^1_0}
&\Z_p[[\Gamma_R^{\prime}]]^{J'_0}\ar[r]& 0},
$$
with  differentials given by formula \ref{formula}. It is easy to see (for example, by induction on $d$) that this is a resolution of $\Z_p$ in the category of topological  $\Z_p[[\Gamma_R^{\prime}]]$-modules. Similarly, we define the Koszul complex $K(\tau^c_1,\ldots,\tau^c_d)$ (with differentials $d^c_q$), $\tau^c_j:=\gamma^c_{j}-1$, $c=\exp(p^i)$. Since $(\gamma_{j}^c)$, $1\leq i\leq d$, is a basis of $\Gamma'_R$, this is also a resolution of $\Z_p$.

  Set $\Lambda:=\Z_p[[\Gamma_{R}]]$ (recall that we have an exact sequence
$1\to \Gamma_R'\to\Gamma_R\to\Gamma_K\to 1$). Consider the complex $K(\Lambda)$:
$$
\xymatrix{0\ar[r]& \Lambda^{J'_d}\ar[r]^-{d^1_{d-1}}&\cdots\cdots
\ar[r]^-{d^1_1} &\Lambda^{J'_1}\ar[r]^-{d^1_0}
&\Lambda^{J'_0}\ar[r]& 0},
$$
By \cite[Lemma 4.3]{KM}, we have the isomorphism 
  $\invlim_m(\Z_p[\Gamma_{K}/(\Gamma_{K})^{p^m}]\otimes_{\Z_p}K(\tau_1,\cdots, \tau_d))\simeq K(\Lambda)$ of 
  left $\Lambda$- and right $\Z_p[[\tau_1,\cdots,\tau_d]]$-modules. It follows that the 
  complex $K(\Lambda)$  is a resolution of $\Z_p[[\Gamma_{K}]]$ in the category of topological   left $\Lambda$-modules. 
  Similarly, we have the complex $K^c(\Lambda)$ (obtained from $K(\tau_1^c,\cdots, \tau^c_d$)) that is also a resolution of $\Z_p[[\Gamma_{K}]]$. 
  
  We define the map
$$
\tau_{0}:\quad
K^c(\Lambda)\to K(\Lambda)
$$
by the following  commutative diagram of topological left $\Lambda$-modules
 $$
 \xymatrix{
 0\ar[r]  & \Lambda^{J'_d}\ar[r]^{d^c_{d-1} }\ar[d]^{\tau^d_{0}}& \cdots\cdots \ar[r]^{d^c_1}& \Lambda^{J'_1}\ar[r]^{d^c_0} \ar[d]^{\tau^1_{0}}& \Lambda\ar[r]
 \ar[d]^{\tau^0_{0}} & \Z_p[[\Gamma_{K}]]\ar[d]^{\gamma_{0}-1}\ar[r] & 0\\
 0\ar[r]  & \Lambda^{J'_d}\ar[r]^{d^1_{d-1} }& \cdots\cdots \ar[r]^{d^1_1} & \Lambda^{J'_1}\ar[r]^{d^1_0} & \Lambda\ar[r] & \Z_p[[\Gamma_{K}]]\ar[r] & 0
 }
 $$
The vertical maps are defined as follows
\begin{align*}
\tau^0_{0} =\gamma_{0}-1,\quad 
 & \tau^q_{0}: (a_{i_1\cdots i_q})\mapsto (a_{i_1\cdots i_q}(\gamma_{0}-\delta_{i_1\cdots i_q})),\quad 1\leq q\leq d,\quad 1\leq i_1 <\cdots < i_q\leq d\\
 & \delta_{i_1\cdots i_q}=\delta_{i_q}\cdots\delta_{i_1},\quad \delta_{i_j}=(\gamma^c_{i_j}-1)(\gamma_{i_j}-1)^{-1}
\end{align*}
Note that 
$$
\frac{\gamma^c_{j}-1}{\gamma_{j}-1}=\sum_{n\geq 1}\binom{c}{n}(\gamma_{j}-1)^{n-1}
$$
makes sense since $\gamma_{j}-1$ is  topologically nilpotent. It is a unit in $\Z_p[[\Gamma_R^{\prime}]]$.
To see that $\tau_{0}$ is a map of complexes note that $$
(\gamma_{0}-\delta_{i_1\cdots i_q})=\tau_{i_q}^{c}
\tau_{i_{q-1}}^{c}\cdots \tau_{i_{1}}^{c}\tau_{0}^0\tau_{i_{1}}^{-1}\cdots \tau_{i_{q-1}}^{-1} \tau_{i_q}^{-1}$$
as the $\tau_j$'s commute for $j\in \{1,\cdots,d\}$. 

  Set $K(\Lambda,\tau):=[\xymatrix{K^{c}(\Lambda)\ar[r]^-{\tau_{0}}& K(\Lambda)}]$. Since we have the exact sequence 
  $$0\to \Z_p[[\Gamma_{K}]]\stackrel{\gamma_{0}-1}{\longrightarrow} \Z_p[[\Gamma_{K}]]  \to 0$$
  the complex $K(\Lambda,\tau)$ resolves $\Z_p$ in the category of topological left $\Lambda$-modules. 
    
   For a topological $\Gamma_{R}$-module $M$, denote by 
${\rm Kos}(\Gamma_R^{\prime},M)$ the complex 
$${\rm Kos}(\Gamma_R^{\prime},M):=\Hom_{\Lambda,\cont}(K(\Lambda),M)=
   \Hom_{\Lambda}(K(\Lambda),M)$$ (which we will also call  the Koszul complex). Similarly,  we define the Koszul complex  
${\rm Kos}^c(\Gamma_R^{\prime},M)$ using the resolution $K^c(\Lambda)$ of $\Z_p$.   In degree $q$ of these two complexes there are $\binom{d}{q}$ copies of $M$.   
If this does not cause confusion we will write the Koszul complexes ${\rm Kos}(\Gamma_R^{\prime},M)$ and ${\rm Kos}^c(\Gamma_R^{\prime},M)$ as
 $$
{\rm Kos}(\Gamma_R^{\prime},M)= M\stackrel{(\tau_j)}{\to} M^{J'_1}\to\cdots\to M^{J'_d},\quad 
{\rm Kos}^c(\Gamma_R^{\prime},M)=M\stackrel{(\tau^c_j)}{\to} M^{J'_1}\to\cdots\to M^{J'_d}. $$
 The map $\tau_{0}:
K^c(\Lambda)\to K(\Lambda)$ induces a map of complexes
$$
\tau_{0}:\quad 
{\rm Kos}(\Gamma_R^{\prime},M)\to {\rm Kos}^c(\Gamma_R^{\prime},M)
$$
which we represent by the following diagram:
$$
\xymatrix{    M^{J'_0}\ar[r]^{(\tau_{j})}\ar[d]^{\tau^0_{0}} & M^{J'_1}\ar[r]\ar[d]^{\tau^1_{0}} & M^{J'_2}\ar[d]^{\tau^2_{0}}  \ar[r] & \\
   M^{J'_0}\ar[r]^{(\tau^c_{j})}    &  M^{J'_1}\ar[r] & M^{J'_2} \ar[r] &
}
$$
Set $$
{\rm Kos}(\Gamma_{R},M):=\Hom_{\Lambda,\cont}(K(\Lambda,\tau),M)=
[\xymatrix{{\rm Kos}(\Gamma_R^{\prime},M)\ar[r]^-{\tau_{0}}&{\rm Kos}^c(\Gamma_R^{\prime},M)})]
$$

  Let $X_{\jcdot}$ denote the standard complex \cite[V.1.2.1]{Laz} computing the continuous cohomology $\R\Gamma_{\cont}(\Lambda,M):=\Hom_{\Lambda,\cont}(X_{\jcdot},M)$. We have 
  $X_n=\Lambda\widehat{\otimes}T^n\Lambda$, where $T^n$ denotes the completed tensor product. 
Let $Y_{\jcdot}$ denote the standard complex computing the group cohomology of 
  $\Gamma_{R}$:
  we have $Y_n=T^{n+1}(\Z_p[\Gamma_{R}])$ and $\R\Gamma(\Gamma_{R},M)=\Hom_{\Z_p[\Gamma_{R}]}(Y_{\jcdot},M)$. Continuous group cohomology
  $\R\Gamma_{\cont}(\Gamma_{R},M)$ is computed by the complex $\Hom_{\Z_p[\Gamma_{R}],\cont}(Y_{\jcdot},M)$ of continuous cochains. The continuous map $Y_{\jcdot}\to X_{\jcdot}$ 
  induces a morphism
  $\Hom_{\Lambda,\cont}(X_{\jcdot},M)\to \Hom_{\Z_p[\Gamma_{R}]}(Y_{\jcdot},M)$ and Lazard shows \cite[V.1.2.6]{Laz} that Mahler's theorem implies
  that this morphism factors through $\Hom_{\Z_p[\Gamma_{R}],\cont}(Y_{\jcdot},M)$ and induces an isomorphism
  $\Hom_{\Lambda,\cont}(X_{\jcdot},M)\to \Hom_{\Z_p[\Gamma_{R}],\cont}(Y_{\jcdot},M)$. Hence 
  $\R\Gamma_{\cont}(\Lambda,M)\stackrel{\sim}{\to}\R\Gamma_{\cont}(\Gamma_{R},M)$. By choosing  a map between the two projective resolutions 
  $K(\Lambda,\tau)$ and $X_{\jcdot}$ of $\Z_p$ (by \cite[V.1.1.5.1]{Laz} such a map exists between any two projective resolutions and is unique up to a homotopy) 
  we obtain a functorial  quasi-isomorphism (unique in the derived category) 
  $$ {\rm Kos}(\Gamma_{R},M)\stackrel{\sim}{\to} \R\Gamma_{\cont}(\Lambda,M)    $$
  Adding up, we have obtained a quasi-isomorphism
  $$
 \lambda:\quad  {\rm Kos}(\Gamma_{R},M)\stackrel{\sim}{\to} \R\Gamma_{\cont}(\Gamma_{R},M).$$

\subsection{$(\varphi,\partial)$-modules and $(\varphi,\Gamma)$-modules}
If 
$$S=\A_R^{[u,v]},\A_R^{(0,v]+},\A_R,$$
set
$$S'=\A_R^{[u,v/p]},\A_R^{(0,v/p]+},\A_R.$$
Write $S(r),S'(r)$
for the $\Gamma_{R}$-module $S,S'$
with the action of $\Gamma_{R}$ twisted by  $\chi^r$.
Define the complex
\begin{align*}
{\rm Kos}(\varphi,\Gamma_R,S(r)): 
 = \left[\begin{aligned}\xymatrix{{\rm Kos}(\Gamma^{\prime}_R,S(r) )\ar[d]^{\tau_{0}}\ar[r]^{1-\phi} & {\rm Kos}(\Gamma^{\prime}_R,S'(r) )\ar[d]^{\tau_{0}}\\
{\rm Kos}^c(\Gamma^{\prime}_R,S(r) )\ar[r]^{1-\phi} & {\rm Kos}^c(\Gamma^{\prime}_R,S'(r) )
}\end{aligned}\right]
\end{align*}

\begin{proposition}\label{26saint}
 There exists
 a universal constant $N$ and a natural\footnote{We use ``natural'' to mean ``defined by universal formulas''.}
 $p^{Nr}$-quasi-isomorphism
$$\tau_{\leq r}{\rm Kos}(\varphi,\Gamma_R,\A_R^{[u,v]}(r))\simeq 
\tau_{\leq r}{\rm Kos}(\varphi,\partial,F^r\A_R^{[u,v]}).$$
\end{proposition}
\begin{proof} 
    Denote by ${\rm Kos}(\Lie \Gamma'_R,{\A_R^{[u,v]}(r)})$ the complex 
    $$
   {\A_R^{[u,v]}(r)}     \to  {\A_R^{[u,v]}(r)}^{J'_1}\to\cdots \to  {\A_R^{[u,v]}(r)}^{J'_d}
    $$
    with differentials dual to those in formula \ref{formula} (with $\tau_j$
 replaced by $\nabla_{j}$).  Consider 
     the maps
$$\nabla_{0}:\quad {\rm Kos}(\Lie \Gamma'_R,{\A_R^{[u,v]}(r)})\to {\rm Kos}(\Lie \Gamma'_R,{\A_R^{[u,v]}(r)})$$
 defined in the following way:
 $$ 
 \xymatrix{
 \A_R^{[u,v]}(r)\ar[r]^{(\nabla_{j})} \ar[d]^{\nabla_{0}}&  {\A_R^{[u,v]}(r)}^{J'_1}
\ar[r]\ar[d]^{\nabla_{0}+p^i}&  \cdots \ar[r] & {\A_R^{[u,v]}(r)}^{J'_q}\ar[r] \ar[d]^{\nabla_{0}+qp^i}  & \cdots\\ 
 \A_R^{[u,v]}(r)\ar[r]^{(\nabla_{j})} &  {\A_R^{[u,v]}(r)}^{J'_1}\ar[r]&  \cdots \ar[r] & {\A_R^{[u,v]}(r)}^{J'_q}\ar[r]  & \cdots
 }
 $$
Define the complex
    \begin{align*}
{\rm Kos}(\varphi,\Lie\Gamma_R,\A_R^{[u,v]}(r)): 
 = \left[\begin{aligned}\xymatrix{{\rm Kos}(\Lie\Gamma_R^{\prime},\A_R^{[u,v]}(r) )\ar[d]^{\nabla_{0}}\ar[r]^{1-\phi} & {\rm Kos}(\Lie\Gamma_R^{\prime},\A_R^{[u,v/p]}(r) )\ar[d]^{\nabla_{0}}\\
{\rm Kos}(\Lie\Gamma_R^{\prime},\A_R^{[u,v]}(r) )\ar[r]^{1-\phi} & {\rm Kos}(\Lie\Gamma_R^{\prime},\A_R^{[u,v/p]}(r) )
}\end{aligned}\right]
\end{align*}
Our proposition follows from Lemma \ref{ref1} and Proposition \ref{ref2} below.
\end{proof}
\begin{remark}
\label{Lie}
The Lie algebra $\Lie \Gamma^{\prime}_R$ of the $p$-adic Lie group $\Gamma^{\prime}_R$ is a free $\Z_p$-module of rank $d$: $\Lie\Gamma^{\prime}_R=\Z_p[\nabla_j,\ 1\leq j\leq d]$.
 It is commutative.
 The Lie algebra $\Lie \Gamma_R$ of the $p$-adic Lie group $\Gamma_R$ is a free $\Z_p$-module of rank $d+1$: $\Lie\Gamma_R=\Z_p[\nabla_j,\ 0\leq j\leq d]$.
 We have $[\nabla_j,\nabla_l]=0,$ $1\leq j,l\leq d,$ and $[\nabla_j,\nabla_0]=p^i\nabla_j$, $1\leq j\leq d.$ One easily checks that
the above Koszul complexes compute Lie algebra cohomology of $\Lie \Gamma^{\prime}_R$ and $\Lie \Gamma_R$ with values in ${\A_R^{[u,v]}(r)}$. That is, we have quasi-isomorphisms
\begin{align*}
\R\Gamma(\Lie \Gamma^{\prime}_R, {\A_R^{[u,v]}(r)}) & \simeq {\rm Kos}(\Lie \Gamma'_R,{\A_R^{[u,v]}(r)}),\\
\R\Gamma(\Lie \Gamma_R, {\A_R^{[u,v]}(r)}) & \simeq [{\rm Kos}(\Lie \Gamma'_R,{\A_R^{[u,v]}(r)})\stackrel{\nabla_0}{\longrightarrow}{\rm Kos}(\Lie \Gamma'_R,{\A_R^{[u,v]}(r)})]\,.
\end{align*}
This implies that 
$$
{\rm Kos}(\varphi,\Lie\Gamma_R,\A_R^{[u,v]}(r))\simeq [\R\Gamma(\Lie \Gamma^{\prime}_R, {\A_R^{[u,v]}(r)})\stackrel{1-\phi}{\longrightarrow}\R\Gamma(\Lie \Gamma^{\prime}_R, {\A_R^{[u,v]}(r)})]\,.$$
\end{remark}
By Lemma \ref{19saint}, multiplication by $t^r$ induces $p^{2r}$-isomorphisms 
$$F^r\A_R^{[u,v]}\simeq t^r\A_R^{[u,v]}\quad{\rm and}\quad
 t^r\A_R^{[u,v/p]}\simeq \A_R^{[u,v/p]}.$$
If $S=\A_R^{[u,v]},\A_R^{[u,v/p]}$, and if we twist the Galois action by
$\chi^r$ on the source (i.e.~replace $S$ by $S(r)$), then multiplication
by $t^r$ becomes Galois-equivariant.  
\begin{lemma}
\label{ref1}
{\rm (i)} There is
 a natural $p^{40r}$-quasi-isomorphism
$$\tau_{\leq r}{\rm Kos}(\varphi,\Lie\Gamma_R,\A_R^{[u,v]}(r))
\simeq \tau_{\leq r}{\rm Kos}(\varphi,\partial,F^r\A_R^{[u,v]}).
$$

{\rm (ii)} There is a natural $p^{80r}$-quasi-isomorphism
$$\tau_{\leq r}{\rm Kos}(\varphi,\Lie\Gamma_R,\A_R^{[u,v]}(r))_n
\simeq \tau_{\leq r}{\rm Kos}(\varphi,\partial,F^r\A_R^{[u,v]})_n.
$$
\end{lemma}
\begin{proof}
We will present the proof of the first claim. The proof for the complexes modulo $p^n$ is similar and we will just point out  the key point where it differs. 

We first construct $p^{4r}$-quasi-isomorphisms
 $$\tau_{\leq r}{\rm Kos}(\Lie \Gamma'_R,S(r))\simeq \tau_{\leq r}{\rm Kos}(\partial',F^rS),$$
 via the following diagram (with the convention $F^jS=S$ for all $j$, if $S=\A_R^{[u,v/p]}$):
 $$
 \xymatrix{
 S(r)\ar[r]^{(\nabla_{j})} \ar[d]_{\wr}^{t^r}&  S(r)^{J'_1}\ar[r]\ar[d]_{\wr}^{t^r}&  \cdots \ar[r] & {S(r)}^{J'_r}\ar[r] \ar[d] _{\wr}^{t^r}& {S(r)}^{J'_{r+1}}\ar[r]\ar[d]_{\wr}^{t^r}  & \cdots\\ 
 F^rS\ar[r]^{(\nabla_{j})} &  F^{r}{S}^{J'_1}\ar[r]&  \cdots \ar[r] 
& F^r{S}^{J'_r}\ar[r]  & F^r{S}^{J'_{r+1}}\ar[r]  & \cdots\\
F^rS\ar[u]^{t^0=\id}_{\wr}\ar[r]^{(\partial_j)}& (F^{r-1}S)^{J'_1}\ar[u]^{t^1}_{\wr}\ar[r] & \cdots\ar[r] & S^{J'_r}\ar[u]^{t^r}_{\wr}\ar[r] & S^{J'_{r+1}}\ar[u]^{t^{r+1}}\ar[r] & \cdots }
 $$
 The top vertical map, being multiplication by $t^r$, 
is a $p^{4r}$-quasi-isomorphism. For the same reason, the bottom map
is a $p^{4r}$-isomorphism in degree~$\leq r$ and injective in higher degrees. (Note, however that {\it it
is not a $p^N$-isomorphism in degree~$\geq r+1$ (for any $N$),
 which explains why we have to truncate
at degree~$r$.}) 
In the case of complexes modulo $p^n$, the above argument works (with double exponents in the error terms) since, by Lemma \ref{19saint}, $t$ is divisible in $S$ by at most $p^2$ (hence multiplication by $t^{r+1}$ is $p^{2(r+1)}$-injective). 
Compatibility with Frobenius and  with $\partial_0$ 
gives us a $p^{36r}$-quasi-isomorphism
between
$$
  \tau_{\leq r} \left[\begin{aligned}
\xymatrix@C=1.5cm{
{\rm Kos}(\partial',F^r{\A_R^{[u,v]}})\ar[r]^-{p^{r}-p^{\jcdot}\varphi}\ar[d]^{\partial_{0}} & 
{\rm Kos}(\partial',{\A_R^{[u,v/p]}})\ar[d]^{\partial_{0}} \\
{\rm Kos}(\partial',F^{r-1}{\A_R^{[u,v]}})   \ar[r]^-{p^{r}-p^{\jcdot+1}\varphi} &  {\rm Kos}(\partial',{\A_R^{[u,v/p]}})}\end{aligned}\right]
$$
and 
$$
  \tau_{\leq r} \left[\begin{aligned}
\xymatrix@C=1.3cm{
{\rm Kos}(\Lie \Gamma'_R,{\A_R^{[u,v]}(r)})\ar[r]^-{p^{r}(1-\varphi)}\ar[d]^{\nabla_{0}} & {\rm Kos}(\Lie \Gamma'_R,{\A_R^{[u,v/p]}(r)})\ar[d]^{\nabla_{0}} \\
{\rm Kos}(\Lie \Gamma'_R,{\A_R^{[u,v]}(r)})   \ar[r]^-{p^{r}(1-\varphi)} &  {\rm Kos}(\Lie \Gamma'_R,{\A_R^{[u,v/p]}(r)})}\end{aligned}\right]
$$
This last complex being $p^r$-quasi-isomorphic to $\tau_{\leq r}{\rm Kos}(\varphi,\Lie\Gamma_R,\A^{[u,v]}_R(r))$,
our lemma follows.
\end{proof}
\begin{proposition}
\label{ref2}
There exists
 a natural quasi-isomorphism 
$$
\laz:\quad  {\rm Kos}(\varphi,\Gamma_R,\A_R^{[u,v]}(r))
\stackrel{\sim}{\to}{\rm Kos}(\phi,\Lie\Gamma_R,\A_R^{[u,v]}(r))
$$
\end{proposition}
\begin{proof}
If $M=\A_R^{[u,v]}(r), \A_R^{[u,v/p]}(r)$,
  consider the map $\beta: {\rm Kos}(\Gamma_R^{\prime},M)\to  {\rm Kos}(\Lie \Gamma'_R,{M})$:
$$
\xymatrix@C=1.2cm{    
M \ar[r]^{(\gamma_{j}-1)}\ar[d]^{\id}_{\wr} & M^{J'_1}\ar[r]\ar[d]^{\beta_{1}}& 
M^{J'_2}\ar[r]\ar[d]^{\beta_{2}}\ar[r] & \\
  M\ar[r]^{(\nabla_j)} &{M}^d\ar[r]&   {M^{d_2}\ar[r] } & 
  }
  $$
with
  $$
  \beta_{q}: (a_{i_1\cdots i_q})\mapsto (\nabla_{i_q}\cdots\nabla_{i_1}\tau_{i_1}^{-1}\cdots\tau_{i_q}^{-1}(a_{i_1\cdots i_q}));\quad 1\leq q\leq d.
 $$
 Similarly, we define the map $\beta^c: {\rm Kos}^c(\Gamma_R^{\prime},M)\to  {\rm Kos}(\Lie \Gamma'_R,{M})$ as:
 $$
 \xymatrix@C=1.2cm{
   M\ar[r]^{(\gamma_{j}^c-1)} \ar[d]^{\beta^c_{0}}&    M^{J'_1}\ar[r] \ar[d]^{\beta^c_{1}}& M^{J'_2}\ar[r] \ar[d]^{\beta^c_{2}}&\\
 M\ar[r]^{(\nabla_j)} & M^{J'_1}\ar[r]&   M^{J'_2}\ar[r]  & 
}
$$
Here    
$$\beta^c_{0}=\nabla_0\tau_0^{-1},\quad  \beta^c_{q}: (a_{i_1\cdots i_q})\mapsto (\nabla_{i_q}\cdots\nabla_{i_1}\nabla_{0}\tau_{0}^{-1}\tau_{i_1}^{c,-1}\cdots\tau_{i_q}^{c,-1}(a_{i_1\cdots i_q}));\quad 1\leq q\leq d.
 $$
\begin{lemma}\label{25saint}
The maps $\beta$ and $\beta^c$ are well-defined isomorphisms.
\end{lemma}
\begin{proof}
We will treat the map $\beta$ first. Since, for $j,k\in \{1,\cdots,d\}$, $\nabla_j\nabla_k\tau_k^{-1}\tau_j^{-1}=(\nabla_j/\tau_j)(\nabla_k/\tau_k)$, 
it suffices to show that for $S=\A_R^{[u,v]}$ or $\A_R^{[u,v/p]}$
 and for $j\in \{1,\cdots,d\}$, the map $\nabla_j/\tau_j:\Lambda\to\Lambda$ is a well-defined isomorphism of $S(r)$. 
Write
$$\frac{\log(1+X)}{X}=1+a_1X+a_2X^2+\cdots\quad 
\frac{X}{\log(1+X)}=1+b_1X+b_2X^2+\cdots$$
We have $v_p(a_k)\geq -\frac{k}{p-1}$ for all $k$ (immediate) which implies that
$v_p(b_k)\geq -\frac{k}{p-1}$ for all $k$.

Now, $\tau_j=(\gamma_j-1)$ if $1\leq j\leq d$, and 
$\tau_0=c^r\gamma_0-1=(c^r-1)\gamma_0+(\gamma_0-1)$ (the $c^r$ comes from the fact that
we are in $S(r)$, not $S$).  Since $c^r-1$ is divisible by $p^2$,
one infers from Lemma~\ref{extr12.2}, that $\tau_j^k(S(r))\subset
\pi_i^{k(p^{i-1}-\delta_R-1)}(p,\pi_i^{e_0})^kS(r)$.
It follows that, since $\big(1+\frac{p^{i-1}-\delta_R-1}{e_0}\big)u>\frac{1}{p-1}$,
 the series
$1+a_1\tau_j+a_2\tau_j^2+\cdots$ and $1+b_1\tau_j+b_2\tau_j^2+\cdots$
converge as series of operators on $S(r)$, and the limit operators are inverse
of each other.  As the first one is $\nabla_j/\tau_j$, this allows to conclude
that $\nabla_j/\tau_j$ is an isomorphism of $S(r)$ for $0\leq j\leq d$.

That settles the case of $\beta$.
  Let us show that the map $\beta_{q}^c$, for $1\leq q\leq d$, is a well-defined  isomorphism, i.e., that the maps $\nabla_{i_q}\cdots\nabla_{i_1}\nabla_0\tau^{-1}_0\tau_{i_1}^{c,-1}\cdots\tau_{i_q}^{c,-1}$, $1\leq i_1 <\cdots <i_q\leq d$,  are well-defined  isomorphisms. We can write the last  map   as 
$(\nabla_{i_q}/\tau_{i_q})\cdots(\nabla_{i_1}/\tau_{i_1})\tau_{i_q}\cdots\tau_{i_1}\nabla_0\tau^{-1}_0\tau_{i_1}^{c,-1}\cdots\tau_{i_1}^{c,-1}$.
Hence, by the computations done above,  we are reduced to showing that the map $\tau_{i_q}\cdots\tau_{i_1}\nabla_0\tau_0^{-1}\tau_{i_1}^{c,-1}\cdots\tau_{i_1}^{c,-1}$ is well-defined and an isomorphism. 

  We have
\begin{align*}
\tau_{i_q}\cdots\tau_{i_1}\nabla_0\tau_0^{-1}\tau_{i_1}^{c,-1}\cdots\tau_{i_q}^{c,-1} & 
=\sum_{k\in \N}a_k\tau_{i_q}\cdots\tau_{i_1}((c^r-1)\gamma_{0}+(\gamma_0-1))^k\tau_{i_1}^{c,-1}\cdots\tau_{i_q}^{c,-1}
\end{align*}
The formula
$$(\gamma_j^a-1)(\gamma_{0}-x)=(\gamma_{0}-x\delta(\gamma^a_j))(\gamma_j^{a/c}-1),\quad \delta(\gamma_j^a):=\frac{\gamma^a_j-1}{\gamma_j^{a/c}-1},
$$
yields
\begin{align*}
(\gamma_j^a-1)(\gamma_{0}-1)^k=(\gamma_{0}-\delta(\gamma_j^a))
(\gamma_{0}-\delta(\gamma^a_j)\delta(\gamma_j^{a/c}))\cdots(\gamma_{0}-\delta(\gamma_j^a)\cdots
\delta(\gamma_j^{a/c^{k-1}}))(\gamma_j^{a/c^k}-1)
\end{align*}
Hence we can write
\begin{align*}
\tau_{i_q}\cdots\tau_{i_1}(\gamma_{0}-1)^k\tau_{i_1}^{c,-1}\cdots\tau_{i_1}^{c,-1} & =(\gamma_{0}-\delta_k)\cdots(\gamma_{0}-\delta_1)
\frac{\gamma_{i_q}^{1/c^k}-1}{\gamma_{i_q}^c-1}\cdots\frac{\gamma_{i_1}^{1/c^k}-1}{\gamma_{i_1}^c-1}\\
 & =(\gamma_{0}-\delta_k)\cdots(\gamma_{0}-\delta_1)\delta_0,
\end{align*}
where $\delta_j\in 1+(p^2,(\gamma_1-1),\cdots,(\gamma_d-1))$ (because $c-1$ is divisible by $p^2$). 

Writing $(\gamma_{0}-\delta_j)=(\gamma_{0}-1)+(1-\delta_j)$, one concludes
that 
$$\tau_{i_q}\cdots\tau_{i_1}(\gamma_{0}-1)^k\tau_{i_1}^{c,-1}\cdots\tau_{i_1}^{c,-1}
\in (p^2,\gamma_0-1,\dots,\gamma_d-1)^k.$$
It follows that the series
of operators $\sum_{k\in \N}a_k\tau_{i_q}\cdots\tau_{i_1}((c^r-1)\gamma_{0}+(\gamma_0-1))^k\tau_{i_1}^{c,-1}\cdots\tau_{i_q}^{c,-1}$ converges, which shows that
$\nabla_{i_q}\cdots\nabla_{i_1}\nabla_0\tau^{-1}_0\tau_{i_1}^{c,-1}\cdots\tau_{i_q}^{c,-1}$
is well-defined.
The same arguments show that the series of operators
$\sum_{k\in \N}b_k\tau_{i_q}^c\cdots\tau_{i_1}^c((c^r-1)\gamma_{0}+(\gamma_0-1))^k
\tau_{i_1}^{-1}\cdots\tau_{i_q}^{-1}$ converges and the sum is the inverse of the
previous operator.

This proves the lemma.
\end{proof}
\begin{remark}\label{251saint}
By Remark \ref{Lie}, the above isomorphisms
can be written as a  quasi-isomorphism
$$
(\beta,\beta^c): \quad \R\Gamma_{\cont}(\Gamma_R,M)\stackrel{\sim}{\to} \R\Gamma(\Lie\Gamma_R,M).
$$
Hence this is a (very simple -- almost commutative) example of integral Lazard isomorphism between Lie algebra cohomology and continuous  group cohomology for $p$-adic Lie groups (see \cite{HKN} for a detailed treatment of integral Lazard isomorphism and \cite{Tam} for an analytic (rational) version of the Lazard isomorphism). 
\end{remark}
  The above maps 
yield the following isomorphism of complexes (where $M=\A_R^{[u,v]}(r)$ and $M'=\A_R^{[u,v/p]}(r)$):
$$
 \left[\begin{aligned}
\xymatrix@C=1.2cm{
{\rm Kos}(\Lie \Gamma'_R,{M})\ar[r]^-{1-\phi}\ar[d]^{\nabla_{0}} &{\rm Kos}(\Lie \Gamma'_R,{M'}) \ar[d]^{\nabla_{0}} \\
{\rm Kos}(\Lie \Gamma'_R,{M})  \ar[r]^-{1-\varphi} &  {\rm Kos}(\Lie \Gamma'_R,{M'})
}\end{aligned}\right]
  \stackrel{(\beta,\beta^c)}{\longleftarrow} 
  \left[\begin{aligned}\xymatrix@C=1.2cm{
 {\rm Kos}(\Gamma^{\prime}_R,M)\ar[d]^{\tau_{0}}\ar[r]^-{1-\phi} & {\rm Kos}(\Gamma^{\prime}_R,M')\ar[d]^{\tau_{0}}\\
 {\rm Kos}^c(\Gamma^{\prime}_R,M)\ar[r]^-{1-\phi}
& {\rm Kos}^c(\Gamma^{\prime}_R,M')}
\end{aligned}\right]
$$
Set $\laz:=(\beta,\beta^c)$. 
This fullfills the requirements of Proposition ~\ref{ref2}.
\end{proof}

\subsection{Change of annulus of convergence}
\begin{lemma}\label{string1}
The natural map 
$${\rm Kos}(\phi,\Gamma_R,\A_R^{(0,v]+}(r))\to {\rm Kos}(\phi,\Gamma_R,\A_R^{[u,v]}(r))$$
is a quasi-isomorphism.
\end{lemma}
\begin{proof}
Use the isomorphism $R_\varpi^{\rm deco}\simeq \A_R^{\rm deco}$
which defines $\A_R^{\rm deco}$ to translate everything into the language
of anlytic functions.
It suffices to prove that the map 
$1-\varphi:R_\varpi^{[u,v]}/R_\varpi^{(0,v]+}\to R_\varpi^{[u,v/p]}/ R_\varpi^{(0,v/p]+}$ is an
isomorphism. 

First note that the natural map $ R_\varpi^{[u,v]}/ R_\varpi^{(0,v]+}\to R_\varpi^{[u,v/p]}/ R_\varpi^{(0,v/p]+}$
(induced by the inclusion of $ R_\varpi^{[u,v]}$ into $ R_\varpi^{[u,v/p]}$) is an isomorphism
(injectivity is true because elements of the kernel are analytic functions which take integral values on
$\frac{u}{e_0}\leq v_p(X_0)\leq \frac{v}{e_0}$ and which extend to analytic
functions taking integral values on $0<v_p(X_0)\leq \frac{v}{pe_0}$, hence belong
to $ R_\varpi^{(0,v]+}$; and surjectivity -  because $r_{\zeta-1}^{[u,v/p]}=r_{\zeta-1}^{[u]}+r_{\zeta-1}^{(0,v/p]+}$
as is clear from the definitions).  Denote by $M$ the module
$ R_\varpi^{[u,v]}/ R_\varpi^{(0,v]+}$; by the above $1-\varphi$ can be considered
as an endomorphism of $M$.

Now, an element $x$ of $ R_\varpi^{[u,v]}$ can be written as
$x=\sum_{k\in\N}x_k\,\frac{X_0^k}{p^{[ku/e_0]}}$, with $x_k\in R_\varpi^{(0,v]+}$
going to $0$ $p$-adically.
So $\varphi(x)=\sum_{k\in\N}p^{[pku/e_0]-[ku/e_0]}\varphi(x_k)\big(\frac{\varphi(X_0)}{X_0^p}\big)^k
\frac{X_0^{pk}}{p^{[pku/e_0]}}$ and, since $[pku/e_0]-[ku/e_0]\geq 1$ if $[ku/e_0]\neq 0$,
we see that $\varphi(x)\in  R_\varpi^{(0,v/p]+}+p R_\varpi^{[u,v/p]}$; hence $\varphi(M)\subset pM$.

To conclude, it remains to check that $M$ does not contain $p$-divisible elements.
Let $(e_i)_{i\in I}$ be a collection of elements of
$R_\varpi^+$ whose images form a basis of $R_\varpi^+/(p,X_0)$ over $k=r_{\zeta-1}^+/(p,X_0)$.
Then $(e_i)_{i\in I}$ is a topological basis of
$R_\varpi^{[u,v]}$ over $r_{\zeta-1}^{[u,v]}$ and of $R_\varpi^{(0,v]+}$ over
$r_{\zeta-1}^{(0,v]+}$.  
Writing everything in the $(e_i)_{i\in I}$ basis,
reduces the question to proving that $r_{\zeta-1}^{[u,v]}/r_{\zeta-1}^{(0,v]+}$ has
no divisible element.  This is rather obvious if you look at Laurent expansions.
\end{proof}  

\subsection{Change of disk of convergence} 
\begin{proposition}\label{string2}
 The natural map  
$${\rm Kos}(\phi,\Gamma_R,\A_R^{(0,v]+}(r))
\to  {\rm Kos}(\phi,\Gamma_R, \A_{R}(r))$$
is a $p^{8}$-quasi-isomorphism 
\end{proposition}
\begin{proof}
We will use the quasi-isomorphisms  with $\psi$-complexes. The proposition is a
direct consequence of Lemmas~\ref{later1} and~\ref{later2} below.
\end{proof}
 Set $\ell=p^{i-1}$. 
From  Proposition~\ref{extr10},
we know that
$\pi_{i}^{-\ell}\A_R^{(0,v]+}$ is stable under the action of   $\psi$ (we have $\ell\geq \ell_R$).  
 For $S=\A_{R},\A_R^{(0,v]+}$,
let ${\rm Kos}(\psi,\Gamma_R,S(r))$ be the complex:
$$
{\rm Kos}(\psi,\Gamma_R,S(r)):=\left[\begin{aligned}\xymatrix{
{\rm Kos}(\Gamma^{\prime}_R,\pi_i^{-\ell}S(r))\ar[r]^-{\psi-1}\ar[d]^{\tau_{0}} &  {\rm Kos}(\Gamma^{\prime}_R,\pi_i^{-\ell}S(r))\ar[d]^{\tau_{0}} \\
{\rm Kos}^c(\Gamma^{\prime}_R,\pi_i^{-\ell}S(r))\ar[r]^-{\psi-1}
& {\rm Kos}^c(\Gamma^{\prime}_R,\pi_i^{-\ell}S(r))}
\end{aligned}
\right]
$$
(For $S=\A_R$, there is no difference between $\pi^{-\ell}_iS$ and $S$.)
 
\begin{lemma}
\label{later1}
If $S=\A_{R},\A_R^{(0,v]+}$,
the natural map
$${\rm Kos}(\phi,\Gamma_R,S(r))\to {\rm Kos}(\psi,\Gamma_R,S(r)),$$
induced by the identity on the first column  and $\psi$
on the second column, is a $p^{p+7}$-quasi-isomorphism.
\end{lemma}
\begin{proof}
The arguments are the same in both cases, but the case $S=\A_R^{(0,v]+}$
is a little bit subtler (because of the difference between $\pi_i^{-\ell}S$
and $S$), so we will only treat that case.

The quotient complex is annihilated by 
$p^{2}$, since so is $\pi^{-\ell}_{i}\A_R^{(0,v]+}/\A_R^{(0,v]+}$ (by Lemma \ref{new2}).
It remains to check that  the kernel complex 
$$\big[\xymatrix{{\rm Kos}\big(\Gamma^{\prime}_R,\big(\A_R^{(0,v/p]+}(r)\big)^{\psi=0}\big)
\ar[r]^-{\tau_{0}}
& {\rm Kos}^c\big(\Gamma^{\prime}_R,\big(\A_R^{(0,v/p]+}(r)\big)^{\psi=0}\big)}\big]
$$
is  $p^{4}$-acyclic.  
Now, according to Proposition ~\ref{extr10}, 
we have a $p^{p+1}$-isomorphism
$$\big(\A_R^{(0,v/p]+}(r)\big)^{\psi=0}\simeq
\bigoplus_{\alpha\in\{0,\dots,p-1\}^{[0,d]}, \,\alpha\neq 0}
\varphi(\A_R^{(0,v]+})[x^\alpha],\quad{\text{where $[x^\alpha]=(1+\pi_i)^{\alpha_0}[x_1^{\alpha_1}]\cdots
[x_d^{\alpha_d}]$.}}$$
Hence, up to $p^{p+1}$, we can replace the kernel complex
by
$$\bigoplus_{\alpha\in\{0,\dots,p-1\}^{[0,d]},\,\alpha\neq 0}
\big[\xymatrix{{\rm Kos}\big(\Gamma^{\prime}_R,\varphi(\A_R^{(0,v]+})(r)[x^\alpha]\big)
\ar[r]^-{\tau_{0}}
& {\rm Kos}^c\big(\Gamma^{\prime}_R,\varphi(\A_R^{(0,v]+})(r)[x^\alpha]\big)}\big],$$
and we can treat the complexes corresponding to each $\alpha$ separately.
There are two cases:

$\bullet$ 
$\alpha_k\neq 0$ for some $k\neq 0$.
We claim that then 
${\rm Kos}\big(\Gamma^{\prime}_R,\varphi(\A_R^{(0,v]+})[x^\alpha]\big)$
and
${\rm Kos}^c\big(\Gamma^{\prime}_R,\varphi(\A_R^{(0,v]+})[x^\alpha]\big)$
are $p^{2}$-acyclic
 (the twist by $(r)$
has disappeared because the cyclotomic character is trivial on $\Gamma'_R$).
As the proof is the same in both cases, we will only treat the first complex.
Write it as the double complex:
 $$
\begin{CD}
 \phi(\A_R^{(0,v]+})[x^\alpha]@>(\gamma_j-1)>>\phi(\A_R^{(0,v]+})^{J''_1}[x^\alpha]@>>>\phi(\A_R^{(0,v]+})^{J''_2}[x^\alpha]@>>> \cdots \\
@VV\gamma_{k}-1V @VV\gamma_k-1V @VV\gamma_{k}-1V\\
\phi(\A_R^{(0,v]+})[x^\alpha]@>(\gamma_j-1)>> \phi(\A_R^{(0,v]+})^{J''_1} [x^\alpha]@>>> \phi(\A_R^{(0,v]+})^{J''_2}[x^\alpha]@>>> \cdots
\end{CD}
$$
 where the first horizontal maps involve $\gamma_j$'s with $j\neq k$, $1\leq j\leq d$.
Now, we have:
$$(\gamma_k-1)\cdot(\varphi(y)[x^\alpha])=\varphi(\pi_1G(y))[x^\alpha],$$
where
$$G(y)=(1+\pi_1)^{\alpha_k}\pi_1^{-1}(\gamma_k-1)y+\pi_1^{-1}((1+\pi_1)^{\alpha_k}-1)y.$$
(We use the fact that $\gamma_k([x^\alpha])=[\epsilon]^{\alpha_k}[x^\alpha]=
\varphi((1+\pi_1)^{\alpha_k})[x^\alpha]$.)

Now, $G$ is $\pi_i$-linear and, since $(\gamma_k-1)$ is trivial modulo~$\pi_i^{-\delta_R-1}\pi$
(cf.~Corollary~\ref{extr12.3}),
 it follows that, modulo~$\pi_1$,
$G$ is just multiplication by $\alpha_k$, since $\pi_i^{\delta_R+1}\pi_1$ divides $\pi$ (Lemma~\ref{new2}). 
 Hence $G$ is invertible.
It follows, that $\gamma_k-1$ is injective
on $\phi(\A_R^{(0,v]+})[x^\alpha]$ and, since $\frac{p^2}{\pi_1}\in\A_R^{(0,v]+}$, that
the cokernel of $(\gamma_k-1)$ is killed by $p^2$.

$\bullet$ 
$\alpha_k=0$ for all $k\neq 0$, and
$\alpha_0\neq 0$. Then to prove that the kernel complex is $p^{2}$-acyclic,
we will prove that $\tau_0:{\rm Kos}\to {\rm Kos}^c$ is injective and the cokernel
complex is killed by $p^{2}$.
This amounts to showing the same statement for
$$\gamma_0-\delta_{i_1}\cdots\delta_{i_q}:\phi(\A_R^{(0,v]+})[x^\alpha](r)\to
\phi(\A_R^{(0,v]+})[x^\alpha](r),\quad\delta_{i_j}=\tfrac{\gamma_{i_j}^c-1}{\gamma_{i_j}-1}.$$
We have
$$(\gamma_0-\delta_{i_1}\cdots\delta_{i_q})\cdot(\varphi(y)[x^\alpha])(r)
=\big(c^r\varphi(\gamma_0(y))(1+\pi)^{p^{-i}(c-1)\alpha_0}[x^\alpha]\big)(r)-
\big(\varphi(\delta_{i_1}\cdots\delta_{i_q}\cdot y)[x^\alpha]\big)(r).$$
So, we are lead to study the map
$F$ defined by
$$F=c^r(1+\pi_1)^{a}\gamma_0-\delta_{i_1}\cdots\delta_{i_q},
\quad {\text{where $a=p^{-i}(c-1)\alpha_0\in\Z_p^*$.}}$$
Now, $c^r-1$ is divisible by $p^2$, $(1+\pi_1)^a=1+a\pi_1$ mod~$\pi_1^2$,
and $\delta_{i_j}-1\in(\gamma_{i_j}-1)\Z_p[[\gamma_{i_j}-1]]$.
Hence, we can write $\pi_1^{-1}F$ in the form
$\pi_1^{-1}F=a+\pi_1^{-1}F',$
with $F'\in(p^2,\pi_1^2,\gamma_0-1,\dots,\gamma_d-1)\Z_p[[\pi_1,\Gamma_R]]$.
It follows from Lemmas~\ref{new2} and~\ref{extr12.2}
 that there exists $N>0$
such that $\pi_1^{-1}F'=0$ on
$\pi_i^k\A_R^{(0,v]+}/\pi_i^{k+N}
\A_R^{(0,v]+}$ for all $k\in\N$.  Hence $\pi_1^{-1}F$
induces multiplication by $a$ on $\pi_i^k\A_R^{(0,v]+}/\pi_i^{k+N}
\A_R^{(0,v]+}$ for all $k\in\N$, which implies
that it is an isomorphism of $\A_R^{(0,v]+}$, and
proves what we want since $\pi_1$ divides $p^2$ by Lemma~\ref{new2}.
\end{proof}

\begin{remark}\label{KL20}
If $\alpha\in\{0,1,\dots,p-1\}^{[0,d]}$, let $M_\alpha=x^\alpha\varphi(\E_R^+)$.
The above proof shows that there exists $N>0$ such that $\frac{\gamma_j-1}{\pi}$
is the multiplication by $\alpha_j$ on $\pi_i^{pk}M_\alpha/\pi_i^{pk+pN}M_\alpha$,
for all $k\in\Z$.  Hence, if $\alpha_j\neq 0$, then $\gamma_j-1$ is invertible
on $x^\alpha\varphi(\E_R)$ and $(\gamma_j-1)^{-1}\cdot \pi_i^{pk}M_\alpha\subset
\pi^{-1}\pi_i^{pk}M_\alpha$, for all $k\in\Z$.
In view of the relationship between Koszul complexes and group cohomology,
this implies that $H^i(\Gamma_R,\pi_i^{pk}M_\alpha)$ is killed by $\pi$,
for all $k\in\Z$, and $\alpha\neq 0$.

Now, we defined maps $c_{{\rm cycl},\alpha}$ on $R_\varpi$ (cf.~Proposition ~\ref{patch1}),
and it follows from Proposition~\ref{patch1} (using $\iota_{\rm cycl}$ to transfer the result
to $\A_R$), that $c_{{\rm cycl},\alpha}(\pi_i^{pk}\E_R^+)\subset \pi_i^{pk-\delta_R} M_\alpha$.
As $x\mapsto (c_{{\rm cycl},\alpha}(x))_{\alpha\neq 0}$ gives an isomorphism
between $\E_R^{\psi=0}$ and $\oplus_{\alpha\neq 0} x^\alpha\varphi(\E_R)$,
commuting with the action of $\Gamma_R$, we infer that
$H^i(\Gamma_R, (\pi^{pk}\E_R^+)^{\psi=0})$ is killed by $\pi\pi_i^{\delta_R}$, for
all $k\in\Z$ and $i\in\N$.
\end{remark}

\begin{lemma}
\label{later2}
The natural map
$${\rm Kos}(\psi,
\Gamma_R,\A_R^{(0,v]+}(r))\to {\rm Kos}(\psi,\Gamma_R,\A_{R}(r))$$ is a 
quasi-isomorphism. 
\end{lemma}
\begin{proof}
The map is injective, so we just have to check that the quotient complex is acyclic.
Since $\A_{R}$ is the completion of $\A_R^{(0,v]+}[\pi^{-1}_{i}]$ with respect to the $p$-adic topology,
 Proposition~\ref{extr10}
 implies that $\psi: \A_{R}/\pi^{-\ell}_i\A_R^{(0,v]+}\to \A_{R}/\pi^{-\ell}_i\A_R^{(0,v]+}$ is (pointwise) topologically
nilpotent, hence $1-\psi$ is bijective.
It follows that $1-\psi$ is bijective on 
${\rm Kos}(\Gamma^{\prime}_{R},\A_{R}/\pi^{-\ell}_i\A_R^{(0,v]+})$ and 
${\rm Kos}^c(\Gamma^{\prime}_{R},\A_{R}/\pi^{-\ell}_i\A_R^{(0,v]+})$.
This allows to conclude.
\end{proof}
\subsection{Passage to Galois cohomology}
\subsubsection{Galois cohomology of $R$}
Let $\widetilde R_\infty$ be the integral closure of $R$ in the sub-$R[\frac{1}{p}]$-algebra of $\overline R[\frac{1}{p}]$
generated by $\zeta_{m}, X^{m^{-1}}_{a+b+1},\cdots,X^{m^{-1}}_{d}$, for all $m\geq 1$.
We have $R_\infty\subset \widetilde R_\infty$ and, by Abhyankar's Lemma (Lemma \ref{Abh} below), $\overline R$ is 
the maximal extension of $\widetilde R_\infty$ which is \'etale in characteristic zero.

Let $\widetilde\Gamma_R=\Gal(\widetilde R_{\infty}[1/p]/R[1/p])$.  Then $\Gamma_R$ is a quotient
of $\widetilde\Gamma_R$ and the kernel $\Gamma_1$ of the projection is isomorphic
to $\prod_{\ell\neq p}(\Z_\ell(1))^c$; in particular, it is of ``order'' prime to $p$.

The algebra $R_\infty[\frac{1}{p}]$ is perfectoid; let
Let $\E_{\widetilde R_\infty}$ be its tilt,
and let $\E_{\widetilde R_\infty}^+$ be its ring of integers (it is
the tilt of $\widetilde R_\infty$).  Finally, let $\A_{\widetilde R_\infty}=W(\E_{\widetilde R_\infty})$;
this is the fixed points of $\A_{\overline R}$ by ${\rm Gal}(\overline R[\frac{1}{p}]/\widetilde R_\infty[\frac{1}{p}])$.

\begin{proposition}
\label{KLiu}
{\rm (i)} The exact sequence (\ref{AS11}) induces a quasi-isomorphism
$$[\xymatrix{\R\Gamma_{\cont}(G_R,\A_{\overline{R}}(r))\ar[r]^-{1-\phi}&
\R\Gamma_{\cont}(G_R,\A_{\overline{R}}(r))}]\stackrel{\sim}{\leftarrow}\R\Gamma_{\cont}(G_R,\Z_p(r)).
$$

{\rm (ii)}
The inclusions $\A_{R_\infty}\subset\A_{{R}^{\prime}_\infty}\subset\A_{\overline R}$ induce  quasi-isomorphisms
\begin{align*}
\mu_H:\quad \R\Gamma_{\cont}(\Gamma_{R},\A_{R_\infty}(r))\stackrel{\sim}{\to} \R\Gamma_{\cont}(\widetilde\Gamma_{R},
\A_{{\widetilde R}_\infty}(r))\stackrel{\sim}{\to}
 \R\Gamma_{\cont}(G_R,\A_{\overline{R}}(r)).
  \end{align*}

{\rm (iii)} 
The inclusion $\A_R\subset\A_{R_\infty}$ induces a quasi-isomorphism
\begin{align*}
\R\Gamma_{\cont}(\Gamma_R,\A_R(r))\stackrel{\mu_{\infty}}{\to}
  \R\Gamma_{\cont}(\Gamma_R,\A_{R_\infty}(r)).
  \end{align*}
\end{proposition}
\begin{proof}
(i) is immediate.  The first quasi-isomorphism in (ii) follows from the fact that $\Gamma_1$ is of order prime to~$p$,
and the second 
follows by almost \'etale descent as in \cite[3.2,3.3]{Co}, \cite[\S2]{AB}, \cite[\S7]{AI}.
Finally, we will prove (iii) by an adaptation of the usual decompletion techniques as in
 \cite[Theorem\,7.16]{AI}, \cite{KL}.

By d\'evissage, it is enough to prove the statement modulo~$p$, in which case the twist disappears
and $\A_R$, $\A_{R_\infty}$ are replaced by $\E_R$, $\E_{R_\infty}$.  Now, $\E_{R_\infty}$
is the completion of the perfectisation $\varphi^{-\infty}(\E_R)$ of $\E_R$.  
But $\E_R=\varphi(\E_R)\oplus \E_R^{\psi=0}=\varphi^2(\E_R)\oplus\varphi(\E_R^{\psi=0})\oplus
\E_R^{\psi=0}=\cdots$, which 
gives us that $$\varphi^{-n}(\E_R)=\E_R\oplus\varphi^{-1}(\E_R^{\psi=0})\oplus\cdots\oplus
\varphi^{-n}(\E_R^{\psi=0}).$$  Hence we can write an element $x$ of $\varphi^{-\infty}(\E_R)$,
uniquely, as $x=s_0(x)+\sum_{n\geq 1}s_n^*(x)$, with $s_0(x)\in\E_R$ and $s_n^*(x)\in
\varphi^{-n}(\E_R^{\psi=0})$ if $n\geq 1$, with $s_n^*(x)=0$ for all but a finite number of $n$'s.  

The map $\E_R$-linear map $s_n=s_0+s_1^*+\cdots+s_n^*$
is the normalised trace map from $\varphi^{-\infty}(\E_R)$ to $\varphi^{-n}(\E_R)$.
On $\varphi^{-n-k}(\E_R)$ it is given by the formula $\varphi^{-n}\circ\psi^n\circ\varphi^{n+k}$,
which shows, using (ii) of Proposition ~\ref{extr10} (or rather its reduction modulo~$p$),
that $$s_n(\varphi^{-n+k}(\E_R^+))\subset \varphi^{-n}(\pi_i)^{-\ell_R}\varphi^{-n}(\E_R^+)\subset
\pi_i^{-\ell_R}\varphi^{-n}(\E_R^+);$$
hence the $s_n$'s are a uniform family of uniformly continuous 
maps $\varphi^{-\infty}(\E_R)$, and they extend to $\E_{R_\infty}$,
and any element $x\in \E_{R_\infty}$  
can be written, uniquely,
as $x=s_0(x)+\sum_{n\geq 1}s_n^*(x)$, with $s_0(x)\in\E_R$ and $s_n^*(x)\in
\varphi^{-n}(\E_R^{\psi=0})$ if $n\geq 1$, with $s_n^*(x)\to 0$ when $n\to\infty$
(this last condition means that $s_n^*(x)\in \pi^{k_n}\varphi^{-n}((\E_R^+)^{\psi=0})$, with $k_n\to +\infty$).
In particular, this gives a $\Gamma_R$-equivariant decomposition $\E_{R_\infty}=\E_R\oplus Z$
(sending $x$ to $(s_0(x),x-s_0(x))$).

Now, if $c$ is a continuous cocycle on $\Gamma_R$, the same is true of the $s_n^*(c)$,
and there exists some $N$ such that
 $s_n^*(c)$ has values in $\pi_i^{-N+k_n} \varphi^{-n}((\E_R^+)^{\psi=0})$, 
with $k_n\geq 0$ and $k_n\to +\infty$.
But Remark~\ref{KL20} tells us that $\varphi^n(s_n^*(c))$ is the coboundary of a cochain
with values in $\pi_i^{p^n(-N+k_n)}\pi^{-1}\pi_i^{-\delta_R}(\E_R^+)^{\psi=0}$.
Hence if we set $M=N+p^i+\delta_R$, it follows that $s_n^*(c)$ is the coboundary of a cochain
$c_n$ with values in $\pi_i^{-M+k_n}\varphi^{-n}((\E_R^+)^{\psi=0})$,
and $c-s_0(c)$ is the coboundary of $\sum_{n\geq 1}c_n$; hence $H^i(\Gamma_R,Z)=0$.

This concludes the proof.
\end{proof}

\subsubsection{The map $\alpha_r^{\laz}$}
  \begin{theorem}
  \label{mainCN}
Assume that $K$ contains enough roots of unity. 
There exist  universal constants $N$ and $c_p$ such that there exist
 $p^{Nr+c_p}$-quasi-isomorphisms,
\begin{align*}
\alpha_r^{\laz}:\quad \tau_{\leq r}{\rm Syn}(R,r) & \simeq\tau_{\leq r}
\R\Gamma_{\cont}(G_{R},\Z_p(r)),\\
\alpha_{r,n}^{\laz}:\quad \tau_{\leq r}{\rm Syn}(R,r)_n & \simeq\tau_{\leq r}
\R\Gamma_{\cont}(G_{R},\Z/p^n(r)).
\end{align*}
  \end{theorem}
\begin{proof}We will argue integrally -- the case modulo $p^n$ being analogous. 
Section~\ref{Losy} provides us with a ``quasi-isomorphism''
(in degrees~$\leq r$ as, in larger degrees the constants become too large)
$${\rm Syn}(R,r)\simeq {\rm Cycl}(R_\varpi^{[u,v]},r).$$
Choosing a basis of $\Omega^1$ and using the isomorphism $R_\varpi^{[u,v]}\simeq \A_R^{[u,v]}$,
we change ${\rm Cycl}(R_\varpi^{[u,v]},r)$ into a Koszul complex
and obtain the isomorphism (see \S\,\ref{moved}):
$${\rm Cycl}(R_\varpi^{[u,v]},r)\simeq{\rm Kos}(\varphi,\partial,F^r\A_R^{[u,v]}).$$
Then, multiplying by suitable powers of $t$, we can get rid of the filtration (Lemma~\ref{ref1})
(in degrees~$\leq r$; this is the only place where the truncation is absolutely necessary),
and change the derivatives into the action of the Lie algebra of $\Gamma_R$, to obtain:
$$\tau_{\leq r}{\rm Kos}(\varphi,\partial,F^r\A_R^{[u,v]})\simeq 
\tau_{\leq r}{\rm Kos}(\varphi,\Lie\Gamma_R,\A_R^{[u,v]}(r)).$$
Standard analytic arguments change this into a Koszul complex for the group (Lemma~\ref{ref2}):
$${\rm Kos}(\varphi,\Lie\Gamma_R,\A_R^{[u,v]}(r))\simeq {\rm Kos}(\varphi,\Gamma_R,\A_R^{[u,v]}(r)).$$
Then, using $(\varphi,\Gamma)$-module techniques, we get a string of ``quasi-isomorphisms''
(Lemmas~\ref{string1}, \ref{later1}, \ref{later2}):
\begin{align*}{\rm Kos}(\varphi,\Gamma_R,\A_R^{[u,v]}(r))
 & \simeq {\rm Kos}(\varphi,\Gamma_R,\A_R^{(0,v]+}(r))
\simeq {\rm Kos}(\psi,\Gamma_R,\A_R^{(0,v]+}(r))\\
 & \simeq {\rm Kos}(\psi,\Gamma_R,\A_R(r))
  \simeq {\rm Kos}(\varphi,\Gamma_R,\A_R(r)).
\end{align*}
General nonsense about Koszul complexes (see~\S\,\ref{CGC}) gives us a quasi-isomorphism:
$${\rm Kos}(\varphi,\Gamma_R,\A_R(r))\simeq[\xymatrix{{\rm R}\Gamma_{\rm cont}(\Gamma_R,\A_R(r))
\ar[r]^-{1-\varphi}&{\rm R}\Gamma_{\rm cont}(\Gamma_R,\A_R(r))}].$$
Finally, general $(\varphi,\Gamma)$-module theory (Proposition~\ref{KLiu})
gives a quasi-isomorphism:
$$[\xymatrix{{\rm R}\Gamma_{\rm cont}(\Gamma_R,\A_R(r))
\ar[r]^-{1-\varphi}&{\rm R}\Gamma_{\rm cont}(\Gamma_R,\A_R(r))}]\simeq
{\rm R}\Gamma_{\rm cont}(G_R,\Z_p(r)).$$
\end{proof}
\begin{remark}
Quasi-isomorphisms between continuous Galois cohomology and Koszul complexes of the type
$$
{\rm Kos}(\varphi,\Gamma_R,\A_R(r))\simeq
{\rm R}\Gamma_{\rm cont}(G_R,\Z_p(r)),
$$
that we have proved above, 
were derived before in the case of local fields with imperfect residue fields by 
Morita \cite{KM} (for torsion representations) and  Zerbes \cite{Zer} 
(for rational representations), and in the case of perfect residue field,
but with a complex similar to our complex for $d=1$, by Tavares Ribeiro \cite{TR}. 
They are all generalizations of the complex from Herr's thesis~\cite{Hr}.
\end{remark}

\subsection{Comparison with local Fontaine-Messing  period map}
The aim of this section is to prove that our period map coincides with that of 
Fontaine-Messing. We will first recall the definition of the latter. 

    Let $E^{\rm PD}_{\overline{R},n}$ denote the log-PD-envelope of $\overline{R}_n$
in $\A_{\crr}^{\times}(\overline{R})_n\otimes R^+_{\varpi}$ 
compatible with the PD-structures on $\A_{\crr,n}$ and $\O_F$.   
We have $E^{\rm PD}_{\overline{R},n}=E^{\rm PD}_{\overline{R}}/p^n$, where
$E^{\rm PD}_{\overline{R}}$ is the ring appearing in Lemma~\ref{fat1}, hence
$E^{\rm PD}_{\overline{R},n}$ has a natural action of $G_{R}$ (trivial on $R_\varpi^+$)
and a Frobenius $\phi$ compatible with the Kummer Frobenius on 
$R^+_{\varpi}$.

Diagram (\ref{diagram}) extends to 
the following diagram of maps of schemes.
$$
\xymatrix@R=.4cm{
 & \Spec E^{\rm PD}_{\overline{R},n}\ar[dd]\ar[dr]\\
\Spec {\overline{R}}_n\ar@{^{(}->}[ru]\ar[dd]\ar@{^{(}->}[rr] & &  \Spec (\A_{\crr}(\overline{R})_n \otimes R^+_{\varpi})\ar[dd]\\
& \Spec R^{\rm PD}_{\varpi,n}\ar[rd]\ar[dd] & \\
\Spec R_n\ar[dd]\ar@{^{(}->}[ru]\ar@{^{(}->}[rr]  & & \Spec R^+_{\varpi,n}\ar[dd]\\
 & \Spec r^{\rm PD}_{\varpi,n} \ar[rd]& \\
\Spec \so_{K,n}\ar@{^{(}->}[rr]\ar@{^{(}->}[ru]  & & \Spec r^+_{\varpi,n}
}
$$
As always, the bottom map is defined by $X_0\mapsto \varpi$.

\smallskip
 Set $\Omega_{E^{\rm PD}_{\overline{R},n}}:=E^{\rm PD}_{\overline{R},n}\otimes_{R^+_{\varpi,n}}\Omega_{R^+_{\varpi,n}}$. For $r\in \N$, we filter the de Rham complex $\Omega_{E^{\rm PD}_{\overline{R},n}}^{\jcdot}$ by subcomplexes
$$
F^r\Omega_{E^{\rm PD}_{\overline{R},n}}\kr:=F^rE^{\rm PD}_{\overline{R},n}\to F^{r-1}E^{\rm PD}_{\overline{R},n}\otimes_{R^+_{\varpi,n}}\Omega^1_{R^+_{\varpi,n}}\to 
F^{r-2}E^{\rm PD}_{\overline{R},n}\otimes_{R^+_{\varpi,n}}\Omega^2_{R^+_{\varpi,n}}\to \cdots
$$
 Define the syntomic complexes
$$
{\rm Syn}(\overline{R},r)_n:=
[F^r\Omega_{E^{\rm PD}_{\overline{R},n}}^{\jcdot}\veryverylomapr{{p^r-p^{\jcdot}\phi}}\Omega_{E^{\rm PD}_{\overline{R},n}}^{\jcdot}].
$$

  For a continuous $G_{R}$-module $M$, let  $C(G_{R},M)$ denote the complex of continuous cochains of $G_{R}$ with values in $M$. One
defines the Fontaine-Messing period map  $$ \tilde{\alpha}^{\rm FM}_{r,n}: \quad {\rm Syn}(R,r)_n \to  C(G_{R},\Z/p^n(r)^{\prime})$$ 
as the composition
\begin{align*}
 {\rm Syn}(R,r)_n & =[F^r\Omega_{R^{\rm PD}_{\varpi,n}}\kr\verylomapr{p^r-p\kr\phi}\Omega_{R^{\rm PD}_{\varpi,n}}^{\jcdot}]
  \to C(G_{R},[F^r\Omega_{E^{\rm PD}_{\overline{R},n}}\kr\verylomapr{p^r-p^{\jcdot}\phi}\Omega_{E^{\rm PD}_{\overline{R},n}}^{\jcdot}])\\
   & \stackrel{\sim}{\leftarrow} 
   C(G_{R},[F^r \A_{\crr}(\overline{R})_n\verylomapr{p^r-\phi}
   \A_{\crr}(\overline{R})_n])\stackrel{\sim}{\leftarrow}C(G_{R},\Z/p^n(r)^{\prime})
\end{align*}
The second quasi-isomorphism above follows from the Poincar\'e Lemma, i.e., from the quasi-isomorphism  
$$F^r\A_{\crr}(\overline{R})_n\stackrel{\sim}{\to}F^r\Omega_{E^{\rm PD}_{\overline{R},n}}^{\jcdot},\quad r\geq 0,
$$
proved in \cite[Lemma 3.1.7]{Ts} (cf. also Lemma \ref{above}).
 
   The following theorem allows us to pass from the local Fontaine-Messing period map to the Lazard type period map   we have defined in sections 3 and 4.  
 \begin{theorem}
 \label{Laz=FM} If  $K$ has enough roots of unity,   the map $\tilde{\alpha}^{\rm FM}_{r,n}$  is 
   $p^{Nr+c_p}$-equal to the  map $\tilde{\alpha}^{\laz}_{r,n}$ from Theorem~\ref{mainCN}.
\end{theorem}
\begin{proof}
Choose $u,v$ as usual.  
The equality of the two maps follows from the commutative diagram below (where we did the $p$-adic version for simplicity). 
The objects and maps of this diagram are described after the diagram. It will be clear from this description
that $K_{\partial,\varphi}(F^rR_\varpi^{\rm PD})={\rm Syn}(R,r)$ and that
$\tilde\alpha^{\rm FM}_r$ is exactly the map obtained form the quasi-isomorphisms in the first row
(except for the fact we should get $C_G(\Z_p(r)')$ instead of $C_G(\Z_p(r))$, but the two are
$p^r$-isomorphic), whereas $\tilde{\alpha}_r^{\laz}$ is obtained by composing the quasi-isomorphisms
forming the lower boundary of the diagram.
        Tildas in the diagram denote maps that we have proved above to be 
   $p^{Nr}$-quasi-isomorphisms or that are known to be such. Little diagram chase shows that this implies
Theorem ~\ref{Laz=FM}.
\end{proof}
$$
{\small
\xymatrix@C=0.5cm{
K_{\partial,\varphi}(F^rR_\varpi^{\rm PD})\ar[d]^{\wr}_{\tau_{\leq r}}\ar[r] & C_G(K_{\partial,\varphi}(F^rE_{\overline R}^{\rm PD}))\ar[d] & 
C_G(K_\varphi(F^r\A_{\crr}(\overline R)))
\ar[l]^{\sim}_{\rm PL}\ar[d] & 
C_G(\Z_p(r))\ar[l]^-{\sim}_-{\rm FES}\ar[dl]^{\sim}\ar[d]^{\wr}\ar[dr]_{\sim}^{\rm AS}
\\
K_{\partial,\varphi}(F^rR_\varpi^{\rm [u,v]})\ar[r]  \ar[rdd]\ar@/_35pt/_{\rm PL}[rddddd]^{\wr}_{\rm PL}& C_G(K_{\partial,\varphi}(F^rE_{\overline R}^{[u,v]})) & C_G(K_{\varphi}(F^r\A_{\overline R}^{[u,v]}))
\ar[l]^{\sim}_{\rm PL} & C_G(K_{\varphi}(\A_{\overline R}^{(0,v]+}(r))) \ar[l]_-{t^r}\ar[r] & C_G(K_{\varphi}(\A_{\overline R}(r)))
\\
& C_\Gamma(K_{\partial,\varphi}(F^rE_{R_\infty}^{[u,v]}))\ar[u]_{\mu_H} & C_\Gamma(K_{\varphi}(F^r\A_{R_\infty}^{[u,v]})) \ar[u]_{\mu_H}
\ar[l]^{\sim}_{\rm PL} & C_\Gamma(K_{\varphi}(\A_{R_\infty}^{(0,v]+}(r)))\ar[u] _{\mu_H}\ar[r]\ar[l]_-{t^r}& C_\Gamma(K_{\varphi}(\A_{R_\infty}(r)))\ar[u]^{\wr}_{\mu_H}
\\
& C_\Gamma(K_{\partial,\varphi}(F^rE_{R}^{[u,v]}))\ar[u]_{\mu_\infty} & C_\Gamma(K_{\varphi}(F^r\A_{R}^{[u,v]})) \ar[l]^{\sim}_{\rm PL}\ar[u]_{\mu_\infty} & 
 C_\Gamma(K_{\varphi}(\A_{R}^{(0,v]+}(r)))\ar[l]_-{t^r}\ar[u]_{\mu_\infty}\ar[r]& C_\Gamma(K_{\varphi}(\A_{R}(r)))\ar[u]^{\wr}_{\mu_\infty}
\\
 & K_{\partial,\varphi,\Gamma}(F^rE_{R}^{[u,v]})\ar[d]^{\laz}_{\wr}\ar[u]^{\wr}_{\lambda} & 
 K_{\varphi,\Gamma}(F^r\A_{R}^{[u,v]})\ar[u]^{\wr}_{\lambda}\ar[l]^{\sim}_{\rm PL}\ar[d]^{\laz}_{\wr}& 
 K_{\varphi,\Gamma}(\A_{R}^{(0,v]+}(r))\ar[l]_-{t^r}\ar[r]^{\sim}\ar[d]^{\laz\can}_{\wr}\ar[u]^{\wr}_{\lambda} & 
 K_{\varphi,\Gamma}(\A_{R}(r)) \ar[u]^{\wr}_{\lambda}
\\
 & K_{\partial,\varphi,\Lie\Gamma}(F^rE_{R}^{[u,v]})
& K_{\varphi,\Lie\Gamma}(F^r\A_{R}^{[u,v]})\ar[l]^{\sim}_{\rm PL}  & K_{\varphi,\Lie\Gamma}(\A_{R}^{[u,v]}(r))\ar[l]_-{t^r}^-{\sim}
\\
  & K_{\partial,\varphi,\partial_{\A}}(F^rE_{R}^{[u,v]})\ar[u]_{t\kr} & K_{\varphi,\partial_\A}(F^r\A_{R}^{[u,v]}) \ar[l]^{\sim}_{\rm PL}\ar[u]^{\wr}_{t\kr, \tau_{\leq r}}
}
}
$$

In the diagram:

\begin{itemize}
\item[$\bullet$] $G$ and $\Gamma$ are $G_R$ and $\Gamma_R$,

\item[$\bullet$] $C_G$ or $C_\Gamma$ denotes the complex of continuous cochains of $G$ or $\Gamma$,

\item[$\bullet$] $K$ denotes a complex of Koszul type:

\begin{itemize}
\item[---] the indices indicate the operators involved in the complex: 
\begin{itemize}

\item[$\diamond$] $\partial$
is a shorthand for $\big(X_0\frac{\partial}{\partial X_0},\dots, X_d\frac{\partial}{\partial X_d}\big)$, 

\item[$\diamond$] $\Gamma$ is a shorthand for $(\gamma_0-1,\dots,\gamma_d-1)$, where the $\gamma_i$'s are our
chosen topological generators of~$\Gamma$,

\item[$\diamond$] $\Lie\Gamma$ is a shorthand for $\big(\nabla_0,\dots,\nabla_d)$, where $\nabla_i=\log\gamma_i$,
so that the $\nabla_i$'s are a basis of $\Lie\Gamma$ over $\Z_p$,

\item[$\diamond$] $\partial_\A$ is a shorthand for 
$\big((1+T)\frac{\partial}{\partial T},X_1\frac{\partial}{\partial X_1},\dots, X_d\frac{\partial}{\partial X_d}\big)$,
viewed as operators on $\A_R^{[u,v]}$ or $E_R^{[u,v]}$, via the isomorphism $\iota_{\rm cycl}:R_{\varpi}^{[u,v]}\cong
\A_R^{[u,v]}$,
\end{itemize}

\item[---] only the first term of the complex is indicated: the rest is implicit and obtained from the first term
so that the maps involved make sense: $\varphi$ does not respect filtration or annulus of convergence,
and $\partial$ or $\partial_\A$ decrease the degrees of filtration by $1$.
\end{itemize}
\end{itemize}

\smallskip
For example, choosing a basis of $\Omega_{R_\varpi/\O_F}$ transforms complexes involving differentials
into complexes of Koszul type: ${\rm Kum}(S,r)=K_{\partial,\varphi}(F^rS)$ if $S=R_\varpi^{\rm PD}$
or $R_\varpi^{[u,v]}$.

\medskip
Let us now turn our attention to the maps between rows.

\begin{itemize}
\item[$\bullet$] The maps AS and FES originating for the upper right corner come from the fundamental exact sequences
of Lemma~\ref{AS} and Artin-Schreier theory of Proposition ~\ref{KLiu}.

\item[$\bullet$] Going from first row to second row just uses the injection $R_\varpi^{\rm PD}\subset R_\varpi^{[u,v]}$.

\item[$\bullet$] Going from third row to second row is the inflation map from $\Gamma_R$ to $G_R$, using the injection
$R_\infty\subset \overline R$.  Note that we could use almost \'etale descent (i.e.~Faltings' purity
theorem or its extension by Scholze or Kedlaya-Liu) to prove that it is a quasi-isomorphism.
\item[$\bullet$] Going form fourth row to third row just uses the injection of $R$ into $R_\infty$.
We could prove that this induces quasi-isomorphisms using the usual decompletion techniques,
but we do not need it for the theorem.

\item[$\bullet$] The maps $\lambda$ connecting the fourth row to the fifth are the maps connecting
continuous cohomology of $\Gamma_R$ to Koszul complexes; they are defined in \S~\ref{CGC}.

\item[$\bullet$] All the maps $\laz$ connecting the sixth row to the fifth
are defined as in Lemma \ref{25saint} and Remark ~\ref{251saint}. 
The same lemma shows that they are well-defined isomorphisms.
(In the fourth column, $\laz$ is composed with the canonical map
$\A_R^{(0,v]+}\to\A_R^{[u,v]}$; this is just to save space and not add an extra column.)

\item[$\bullet$] The maps $t^{\jcdot}$ connecting the last row to the sixth are the maps appearing
in the proof of Lemma~\ref{ref1} (multiplication by suitable powers of $t$).
\end{itemize}

\medskip
Finally, let us describe the maps between columns.

\begin{itemize}
\item[$\bullet$] The maps from the first column to the second are induced
by the natural injections of rings; the PL-map is a quasi-isomorphism by Lemma~\ref{above2}.

\item[$\bullet$] The maps from the third column to the second are also induced
by the natural injections of rings; the PL-map are quasi-isomorphisms by 
Lemma~\ref{above} (for the first 6 rows) and Lemma~\ref{above2} (for the last row).

\item[$\bullet$] From the fourth to the third the map is multiplication by $t^r$
(as explained before Lemma~\ref{ref1}),
composed with inclusion of rings.

\item[$\bullet$] From the fourth to the last the map is just
inclusion of rings.
\end{itemize}

\section{Global applications}
Unless otherwise stated, we work in the category of integral quasi-coherent  log-schemes. 
We will denote by $\so_K$,
$\so_K^{\times}$, and $\so_K^0$ the scheme $\Spec (\so_K)$ with the trivial, canonical
(i.e., associated to the closed point), and $({\mathbf N}\to \so_K, 1\mapsto 0)$
log-structure, respectively.
\subsection{Syntomic cohomology and $p$-adic nearby cycles}
\subsubsection{Definition of syntomic cohomology}
 Let $X$ be a log-scheme, log-smooth over $\so_K^{\times}$. For any $r\geq 0$, consider its absolute (meaning over $\O_F$) log-crystalline cohomology complexes
\begin{align*}
\R\Gamma_{\crr}(X^{},\sj^{[r]})_n: &  =\R\Gamma(X_{\eet},\R u_{X^{}_n/W_n(k)*}\sj^{[r]}_{X^{}_n/W_n(k)}),\quad
\R\Gamma_{\crr}(X^{},\sj^{[r]}):=\holim_n\R\Gamma_{\crr}(X^{},\sj^{[r]})_n,
\end{align*}
where  $u_{X^{}_n/W_n(k)}: (X^{}_n/W_n(k))_{\crr}\to X_{\eet}$ is the
projection from the log-crystalline to the \'etale topos and  $\sj^{[r]}_{X^{}_n/W_n(k)}$ is the r'th divided 
power of the canonical PD-ideal $\sj_{X^{}_n/W_n(k)}$ ( for $r\leq 0$, by convention,
$\sj^{[r]}_{X^{}_n/W_n(k)}:=\so_{X^{}_n/W_n(k)}$ and we will often omit it from the notation).

   We have the mod $p^n$ and completed syntomic complexes
\begin{align*}
\R\Gamma_{\synt}(X^{},r)_n & :=
[\R\Gamma_{\crr}(X^{},\sj^{[r]})_n\verylomapr{p^r-\phi} \R\Gamma_{\crr}(X^{}_n)],\\
  \R\Gamma_{\synt}(X^{},r) & :=\holim _n\R\Gamma_{\synt}(X^{},r)_n.
\end{align*}
Here the Frobenius $\phi$ is defined by the composition
\begin{align*}
\phi: \R\Gamma_{\crr}(X^{},\sj^{[r]})_n   \to \R\Gamma_{\crr}(X^{}_n) \stackrel{\sim}{\to }\R\Gamma_{\crr}(X^{}_1/\O_F)_n
 \stackrel{\phi}{\to}
\R\Gamma_{\crr}(X^{}_1/\O_F)_n
  \stackrel{\sim}{\leftarrow}\R\Gamma_{\crr}(X^{}_n)
\end{align*}
The mapping fibers are taken in the $\infty$-derived category of abelian groups. 

We have
$ \R\Gamma_{\synt}(X^{},r)_n\simeq \R\Gamma_{\synt}(X^{},r)\otimes^L{\mathbf Z}/p^n. $
There is  a canonical quasi-isomorphism 
\begin{align*}
\R\Gamma_{\synt}(X^{},r)_n  \stackrel{\sim}{\to}[\R\Gamma_{\crr}(X^{})_n\verylomapr{(p^r-\phi,\can)} \R\Gamma_{\crr}(X^{})_n\oplus \R\Gamma_{\crr}(X^{},\so/\sj^{[r]})_n].
\end{align*}
Similarly in the completed  case.

  The above definitions sheafify.
 Let $\sj^{[r]}_{\crr,n}$, $\sa_{\crr,n}$, and $\sss_n(r)$ for $r\geq 0$ be the
 \'etale sheafifications  of the presheaves sending \'etale map $U^{}\to X^{}$ to $
\R\Gamma_{\crr}(U^{},J^{[r]})_n$,  $
\R\Gamma_{\crr}(U^{})_n$, and $\R\Gamma_{\synt}(U^{},r)_n$, respectively.  We have
$$
\sss_n(r)\simeq [\sj^{[r]}_{\crr,n}\stackrel{p^r-\phi}{\longrightarrow}\sa_{\crr,n}],\quad \R\Gamma_{\synt}(X^{},r)_n=\R\Gamma(X^{}_{\eet},\sss_n(r)).
$$

\begin{remark}Recall that syntomic cohomology can also be defined  as cohomology of the syntomic site with values in  syntomic Tate twists. Let us explain briefly how. 
 We will be working with fine log-schemes. For a log-scheme $X$ we denote by $X_{\synt}$ the small log-syntomic site of $X$ defined as follows. The objects are
  morphisms $f:Y\to Z$ that are {\em  log-syntomic} in the sense of Kato, i.e., 
the morphism  $f$ is integral, the underlying morphism of schemes is flat and locally of finite presentation, 
and, \'etale locally on $Y$, there is a factorization $Y\stackrel{i}{\hookrightarrow}W\stackrel{h}{\rightarrow}Z$ 
where $h$ is log-smooth and $i$ is an exact closed immersion that is transversally regular over $Z$.
 We also require $f$ to be locally quasi-finite on the underlying schemes and the cokernel of the map
 $(f^*M_Z)^{\gp}\to M^{\gp}_Y$ ($M_?$ being the log-structure sheaf) to be torsion. 

 For a log-scheme $X$ log-syntomic over $\Spec(W(k))$, define 
$$
\so^{\crr}_n(X) =H^0_{\crr}(X,\so)_n,\qquad 
\sj_n^{[r]}(X) =H^0_{\crr}(X,\sj^{[r]})_n.
$$
We know  that the presheaves $\sj_n^{[r]} $ are sheaves on $X_{n,\synt}$, flat over ${\mathbf Z}/p^n$, and that
$\sj^{[r]}_{n+1}\otimes{\mathbf Z}/p^n\simeq  \sj^{[r]}_{n}$.  There is a  canonical 
isomorphism 
$$
\R\Gamma(X_{\synt},\sj_n^{[r]})\simeq \R\Gamma_{\crr}(X,\sj^{[r]})_n
$$
that is compatible with Frobenius. 
Set $$\sss_n(r):=[\sj_n^{[r]} \stackrel{p^r-\phi}{\longrightarrow}\so^{\crr}_n].$$
This is the syntomic Tate twist. 
In the same way we can define log-syntomic sheaves $\sss_n(r)$ on $X_{m,\synt}$ for $m\geq n$. 
Abusing notation, we define $\sss_n(r)=i_*\sss_n(r)$ for the natural map $i: X_{m,\synt}\to X_{\synt}$. Since $i_*$ is exact, $\R\Gamma(X_{m,\synt},\sss_n(r))=\R\Gamma(X_{\synt},\sss_n(r))$.
We have
\begin{align*}
 \sss_n(r) & =\R\epsilon_*\sss_n(r),\quad 
\varepsilon: X_{\synt}\to X_{\eet},\\
\R\Gamma_{\synt}(X,r)_n & =\R\Gamma(X_{\eet},\sss_n(r))=\R\Gamma(X_{\synt},\sss_n(r)).
\end{align*}
\end{remark}

  Proposition~\ref{added}
gives a simple description, up to some universal torsion, of syntomic cohomology sheaves.
Namely, assume that $X$ is semistable over $\so_K$ or a base change of such. We write
\begin{align*}
\sss_n(r)  \simeq [\sj^{[r]}_{\crr,n}\lomapr{p^r-\phi}\sa_{\crr,n}]
  \simeq [\sa_{\crr,n}^{\phi=p^r}\to \sa_{\crr,n}/\sj_{\crr,n}^{[r]}],
\end{align*}
where we set $\sa_{\crr,n}^{\phi=p^r}:=[\sa_{\crr,n}\lomapr{p^r-\phi}\sa_{\crr,n}]$.
\begin{corollary}\label{synt}
{\rm (i)}
 For $i\neq r$, the sheaf $\sh^i(\sa_{\crr,n}^{\phi=p^r})$ is annihilated by $p^{N}$, for a constant $N=N(e,d,p,r)$.

{\rm (ii)}
  For $i\geq r+1$, the sheaf $\sh^i(\sss_n(r))$ is annihilated by $p^{N(r)}$.

{\rm (iii)}
  For $i\leq r-1$, the natural map  $\sh^{i-1}(\sa_{\crr,n}/\sj^{[r]}_{\crr,n}){\to}\sh^i(\sss_n(r))$ is a $p^{N}$-quasi-isomorphism, for a constant $N=N(e,d,p,r)$. Moreover we have the short exact sequence
$$
0\to \sh^{r-1}(\sa_{\crr,n}/\sj^{[r]}_{\crr,n})\to \sh^r(\sss_n(r))\to \sh^r(\sa_{\crr,n}^{\phi=p^r})\to 0
$$
\end{corollary}

\subsubsection{Syntomic complexes and $p$-adic nearby cycles}
  Let $X$ be a  fine and saturated  log-scheme log-smooth over $\so^{\times}_K$. Denote by $X_{\tr}$ the locus where the log-structure is trivial. This is an open dense subset of the generic fiber of $X$.  Fontaine-Messing \cite{FM}, Kato \cite{K1} have constructed  period morphisms  ($i: X_0\hookrightarrow X, j: X_{\tr}\hookrightarrow X$)
$$\alpha^{\rm FM}_{r,n}:  \quad \sss_n(r)_{X}  \rightarrow i^*Rj_*{\mathbf Z}/p^n(r)^{\prime}_{X_{\tr}},\quad r\geq 0.
$$

    Assume now that $X$ has semistable reduction over $\so_K$ or is a base change of a scheme with semistable reduction over the ring of integers of a subfield of $K$. That is, locally, $X$ can be written as $\Spec(A)$  for a ring $A$ \'etale over 
  \begin{equation}
  \label{form}
 \so_K[X_1^{\pm 1},\cdots,X_a^{\pm 1}, X_{a+1},\cdots, X_{a+b},X_{a+b+1},\cdots,X_{d},X_{d+1}]/(X_{d+1}X_{a+1}\cdots X_{a+b}-\varpi^h), \quad 1\leq h\leq e.
  \end{equation}
  If we put $D:=\{X_{a+b+1}\cdots X_d=0\}\subset \Spec(A)$ then the log-structure on  $\Spec A$  is associated to the special fiber and to the divisor $D$. We have $\Spec(A)_{\tr}=\Spec(A_K)\setminus D_K$.

    The purpose of this section is to prove the following theorem.
 \begin{theorem}
 \label{main}
Let $X$ be a scheme with semistable reduction over $\so_K$ or  a base change of a scheme with semistable reduction over the ring of integers of a subfield of $ K$.  For $0\leq i\leq r$, consider the period map
$$
\alpha^{\rm FM}_{r,n}:\quad  \sh^i(\sss_n(r)_{X}) \rightarrow i^*R^ij_*{\mathbf Z}/p^n(r)'_{X_{\tr}}.
$$

{\rm (i)}
If $K$ has enough roots of unity then the kernel  and cokernel of this map is annihilated by $p^{Nr+c_p}$ for universal constants $N$ and $c_p$.

{\rm (ii)}
In general, the kernel  and cokernel of this map is annihilated by $p^N$ for an integer $N=N(K,p,r)$, which depends on $K$, $p$, $r$.
\end{theorem}
\begin{proof}
 It suffices to argue locally. Take an $\so_K$-algebra $A$ as in (\ref{form}) and such that $\Spec (A/p)$ is nonempty and connected. 
We have 
\begin{align*}
\R\Gamma_{\synt}(\Spec (A)_,r)_n ={\rm Syn}(\widehat{A},r)_n,\quad 
\R\Gamma_{\synt}(\Spec (A),r)={\rm Syn}(\widehat{A},r).
\end{align*}
 The Fontaine-Messing period map $$\alpha^{\rm FM}_{r,n}=\alpha^{\rm FM}_{r,n,A}:\quad {\rm Syn}(A,r)_n \to \R\Gamma(Y_{\tr,\eet},\Z/p^n(r)^{\prime}),\quad Y:=\Spec A^h,$$  can be described as the composition of the henselian version of the map $\tilde{\alpha}^{\rm FM}_{r,n}$  with the natural map $C(G_{R},\Z/p^n(r)^{\prime})\to \R\Gamma(Y_{\tr,\eet},\Z/p^n(r)^{\prime})$ for $R:=A^h$ and $G_R$ - the Galois group of the maximal extension of $R$ unramified in characteristic zero outside the divisor $D_K$. The henselian version of the period map $\tilde{\alpha}^{\rm FM}_{r,n}$ is obtained by replacing $\overline{\widehat{R}}$ with $\overline{R}$ and $G_{\widehat{R}}$ with $G_R$. We set ${\rm Syn}(A,r)_n={\rm Syn}(\widehat{A},r)_n$.

 Let $i\leq r$.  We need to show that  the map 
\begin{equation}
\label{def-1}
\alpha^{\rm FM}_{r,n}:\quad H^i{\rm Syn}(A,r)_n \lomapr{\tilde{\alpha}^{\rm FM}_{r,n}} H^i C(G_{R},\Z/p^n(r)^{\prime})\to H^i\R\Gamma(Y_{\tr,\eet},\Z/p^n(r)^{\prime}),
\end{equation}
is an isomorphism (up to the wanted constants). To do that we will pass to the completion of $A$. Consider the following commutative diagram.
\begin{equation}
\label{passage}
\xymatrix{
 H^i({\rm Syn}(A,r)_n) \ar@{=}[d] \ar[r]^{\tilde{\alpha}^{\rm FM}_{r,n,A}} &  H^i(G_R,{\mathbf Z}/p^n(r)')\ar[d]_{\wr}\ar[r] &  H^i(U_{\eet},\Z/p^n(r)')\ar[d]^{g}_{\wr}\\
H^i({\rm Syn}(\widehat{A},r)_n) \ar[r]^{\tilde{\alpha}^{\rm FM}_{r,n,\widehat{A}}}_{\sim} &   H^i(G_{\widehat{R}},{\mathbf Z}/p^n(r)')\ar[r]^{f}_{\sim} &  H^i(\su_{\eet},\Z/p^n(r)'), 
}
\end{equation}
 where $U=(\Spec R)_{\tr}, \su:=(\Sp \widehat{R}[1/p])_{\tr}$. 
Now, the middle vertical arrow is an isomorphism because the two Galois groups are equal by Lemma~\ref{Abh} 
(combination of Abhyankar's Lemma and Elkik's approximation
techniques), the map $f$ is an isomorphism by a $K(\pi,1)$-Lemma (see  \ref{kpi1}),  and the map $g$ is an isomorphism by a rigid GAGA argument (see~\ref{rgaga}). All of which we prove below. 
  It thus remains  to prove that the map $$\tilde{\alpha}^{\rm FM}_{r,n}=\tilde{\alpha}^{\rm FM}_{r,n,\widehat{A}}:\quad  H^i({\rm Syn}(\widehat{A},r)_n) \to  H^i(G_{\widehat{R}},{\mathbf Z}/p^n(r)'),\quad i\leq r, $$
  is an isomorphism (up to the wanted constants). In the case that  $K$ has enough roots of unity this follows from Theorem \ref{Laz=FM} and 
   Theorem \ref{mainCN}, which proves the first claim of the theorem. 

  To prove the second claim, we pass from $R={A}$  to $R^i:=R(\zeta_{p^i})$, $i\geq c(K)+3$, so that $R^i$ has enough roots of unity. There the Fontaine-Messing period map is
  a $p^{Nr}$-quasi-isomorphism, for a universal constant $N$ (as we have just shown). To descend note that this period map is $G_i=\Gal(K_i/K)$-equivariant, i.e., that the following diagram commutes 
  (cf. Remark \ref{action5} below). 
$$
\xymatrix{
\R\Gamma(G_i,{\rm Syn}(R^i,r)_n) \ar[r]^-{\tilde{\alpha}^{\rm FM}_{r,n}} & \R\Gamma(G_i,C(G_{R^i},{\mathbf Z}/p^n(r)'))\\
    {\rm Syn}(R,r)_n \ar[u]\ar[r]^-{\tilde{\alpha}^{\rm FM}_{r,n}} & C(G_R,{\mathbf Z}/p^n(r)')\ar[u]^{\wr}
}
$$
Hence to finish the proof of our theorem it suffices to quote Lemma~\ref{Gal} below.
\end{proof}
 
\subsubsection{Comparison of Galois groups}
Recall the following theorem of L\"utkebohmert \cite{Lut}. 
\begin{theorem}{\rm (Riemann's Existence)}
\label{riemann}
Let $\sx$ be a smooth quasi compact rigid space. Let $\szz\subset \sx$ be a normal crossings divisor and set $\su=\sx\setminus \szz$. Then any finite \'etale covering of $\su$ extends (uniquely) to a finite flat normal covering of $\sx$. 
\end{theorem}
\begin{proof}
This is  the main theorem of \cite{Lut}. One uses the description of  the tubular neighborhoods of the divisor $\szz$ from \cite[Theorem 1.18]{Kih} to pass to pointed disks (of varied dimensions), which then can be treated by the extension lemma \cite[Lemma 3.3]{Lut}. 
\end{proof}

   Let $R$ be a ring as in (\ref{form}) or a henselization at $(p)$ of such a ring or a ring as in \ref{formalscheme}. 
 Let $R_m:=R(\zeta_m)[X^{m^{-1}}_{a+b+1},\cdots,X^{m^{-1}}_{d}]$, for a choice of a primitive $m$'th root of unity $\zeta_m$, and $R^{\prime}_\infty:=\cup_{m}R_m$.
Let $\overline R_\infty^{\prime}$ be the maximal extension of $R^{\prime}_\infty$ which is \'etale in characteristic zero. We define similarly $\widehat{R}^{\prime}_\infty$ and $\overline{ \widehat{R}}_\infty^{\prime}$. 
\begin{lemma}
\label{Abh}
  {\rm (i)} {\rm (Abhyankar's Lemma)}
The natural inclusion $\overline R_\infty^{\prime}\subset \overline R$ is an equality. 
In particular, we have  the natural isomorphism
$$
 G_R =\Gal(\overline R[1/p]/R[1/p])\stackrel{\sim}{\to}\Gal(\overline R^{\prime}_\infty[1/p]/R[1/p]).
$$

 {\rm (ii)} {\rm (Approximation)} Let $R$ be a henselization of a ring of the form (\ref{form}). 
The natural map 
  $G_{\widehat{R}}\to G_R
  $
  is an isomorphism.
  \end{lemma}
\begin{proof}
For the first statement, we first assume that $R$ is of the form (\ref{form}). Recall that $\overline R$ is the direct limit of a maximal chain of normal $R$-algebras, which are domains and, after inverting $p$, are finite and \'etale extensions of $R[1/p][X^{{-1}}_{a+b+1},\cdots,X^{-1}_{d}]$. Similarly,
$\overline R_\infty^\prime$ is the direct limit of a maximal chain of normal $R$-algebras, which are domains and, after inverting $p$, are finite and \'etale extensions of $R^{\prime}_\infty[1/p]$.
Let $Y\to\Spec(R[1/p])$ be a normal, connected, finite scheme such that the base change $Y_{ (\Spec R)_{\tr}}\to  (\Spec R)_{\tr}$, $(\Spec R)_{\tr}:=\Spec( R[1/p])\setminus 
\{X_{a+b+1}\cdots X_{d}=0\}$,  is a finite \'etale extension.  
Since $R[1/p]$ is regular, by Abhyankar's Lemma \cite[XIII, Proposition  5.2]{SGA1},
there exists a number $m\geq1$ such that  the finite extension $Y_m:=(Y\times_{\Spec(R[1/p])}\Spec(R(\zeta_m)[1/p][X^{m^{-1}}_{a+b+1},\cdots,X^{m^{-1}}_{d}]))^{\rm norm}$ of $\Spec(R[1/p][X^{m^{-1}}_{a+b+1},\cdots,X^{m^{-1}}_{d}])$  is \'etale. 
The first statement of the lemma for follows for $R$ and for its henselization at $(p)$. For $R$ as in \ref{formalscheme}, we argue in the same way noting that the rigid space $Y$ above exists by Theorem \ref{riemann}.  

  For the second statement,   since clearly $\Gal(\widehat{R}^{\prime}_{\infty}[1/p]/\widehat{R}[1/p]))\stackrel{\sim}{\to}\Gal({R}^{\prime}_{\infty}[1/p]/R[1/p]))$, by  (i),     it suffices to show that
  the natural map 
  $$\Gal(\overline{\widehat{R}}^{\prime}_{\infty}[1/p]/\widehat{R}_{\infty}[1/p])) \to \Gal(\overline{R}^{\prime}_{\infty}[1/p]/R_{\infty}[1/p])) $$
  is an isomorphism. 
  But
  \begin{align*}
  \Gal(\overline{\widehat{R}}^{\prime}_{\infty}[1/p]/\widehat{R}_{\infty}[1/p])) & =\lim_m \Gal(\overline{\widehat{R}}_{m}[1/p]/\widehat{R}_{m}[1/p])),\\
  \Gal(\overline{R}^{\prime}_{\infty}[1/p]/R_{\infty}[1/p]))  &=\lim_m  \Gal(\overline{{R}}_{m}[1/p]/{R}_{m}[1/p])),
    \end{align*}
    where $\overline{\widehat{R}}_{m}[1/p]$ (resp. $\overline{{R}}_{m}[1/p]$) denotes the maximal \'etale extension of $\widehat{R}_{m}[1/p]$ (resp. ${R}_{m}[1/p]$). 
Since $R_m$ is henselian at $(p)$ and $(R_m)^{\wedge}=\widehat{R}_m$, we have, by Elkik's theorem  \cite[Corollary p. 579]{Elk}, that, for $m\geq 0$, 
$$\Gal(\overline{\widehat{R}}_{m}[1/p]/\widehat{R}_{m}[1/p])) \stackrel{\sim}{\to}\Gal(\overline{{R}}_{m}[1/p]/{R}_{m}[1/p])).$$
To conclude we pass to the limit over $m$.
  \end{proof}
  
\subsubsection{$K(\pi,1)$-Lemmas}\label{kpi1}
  To show that the map  $f$ in diagram \ref{passage} is an  isomorphism, note that this is just 
 the  $K(\pi,1)$-Lemma for $\su$ and $p$-coefficients. In the case of the divisor at infinity $D$ being trivial  this is proved by Scholze \cite[Theorem 1.2]{Sch}. The general case we will treat the way 
  Faltings treated  the algebraic case.  We have to show that, for any locally constant constructible $p^n$-torsion sheaf $\sm$,
    any class $\beta\in H^i(\su_{\eet},\sm)$, $i >0$, dies in a finite \'etale extension of $\su$. During the proof we will often pass to the extension of $\su$ generated by taking $k$'th roots of all the $U_i$'s, $k\geq 0$. This is possible  since the normalization of $\su$ in this extension satisfies the same hypothesis as $\su$. 
    
    By Abhyankar's Lemma \ref{Abh} we have $G_{\widehat{R}}=\Gal(\overline{\widehat{R}}^{\prime}_{\infty}[1/p]/\widehat{R}[1/p])$. Hence, by adjoining $k$'th roots of $U_i$'s, for some $k\geq 0$,  and adding some roots of unity if necessary, we can assume that the sheaf $\sm$ is unramified, i.e., that it is a module for the fundamental group of $\Spec(\widehat{R}[1/p])$.  Let $j: \su\hookrightarrow \Sp(\widehat{R}[1/p]), f=U_1\cdots U_b$.  There is a spectral sequence
 $$
 E^{s,t}_2=H^s(\Sp(\widehat{R}[1/p])_{\eet},\R^t j_*\sm)\Rightarrow H^{s+t}(\su_{\eet},\sm).
 $$
Taking $p^k$'th roots of $f$ induces multiplication by $p^{tk}$ on $E^{s,t}_2$. This can be checked on the finite extension of $\Spec(R[1/p])$ that trivializes $\sm$ and there it follows from the purity statement \cite[Thm 7.2.2]{dJvP}, \cite[Proposition  2.0.2]{Nan}, i.e., the fact that
$$
R^tj_*\Z/p^n\simeq \wedge^t\overline{M}^{\gp}_{Y}\otimes\Z/p^n(-1),\quad Y:=\Sp(\widehat{R}[1/p]),
$$
where $M_{Y}$ is the log-structure on $Y$ associated to $D_K$: $M_{Y}=j_*\so^*_{\su}\cap\so_{Y}$, $M^{\gp}_Y$ is the induced sheaf of abelian groups, and $\overline{M}_Y:=M_Y/\so^*_Y$. 
  
    Repeating this procedure several times we reduce to the case when $\beta$ is a restriction of some cohomology class on $\Sp(\widehat{R}[1/p])$. Then we can use Scholze's $K(\pi,1)$-Lemma again. 
\subsubsection{Rigid GAGA}\label{rgaga}
 To show that the map $g$ in diagram (\ref{passage}) is an isomorphism, we will argue by induction on the number of irreducible components of the divisor $D_K$, i.e., on $s$. If $s=0$ this is a result of Gabber \cite[Thm 1]{Ga}. Assume that the isomorphism is true for $s-1$. The case of $s$ follows now easily from the following commutative diagram of localization sequences in the \'etale and rigid \'etale cohomology \cite[Thm 7.2.2]{dJvP}(we assumed $s=1$ for simplicity).
$$
\xymatrix{
\cdots \to H^i(D_{K,\eet},\Z/p^n)\ar[r]^-{i_*}\ar[d]^{\wr}& H^i(T_{\eet},\Z/p^n)\ar[r]^{j^*}\ar[d]^{\wr} & H^i(U_{\eet},\Z/p^n)\ar[r] \ar[d] & H^{i-1}(D_{K,\eet},\Z/p^n)\ar[d]^{\wr}\to\cdots \\
\cdots \to H^i(\sd_{\eet},\Z/p^n)\ar[r]^-{i_*} & H^i(\stt_{\eet},\Z/p^n)\ar[r]^{j^*} & H^i(\su_{\eet},\Z/p^n)\ar[r]  & H^{i-1}(\sd_{\eet},\Z/p^n)\to\cdots}
$$
Here we set $T:=\Spec(R[1/p])$, $\stt:=\Sp(\widehat{R}[1/p])$, and $\sd$ is the rigid analytic spaces associated to $D_K$. The vertical isomorphisms follow from the inductive assumption.

\subsubsection{Galois descent}
\begin{lemma}\label{Gal}Let $K_1$ be a Galois extension of 
 $K$, with relative ramification index $e_1$ and Galois group $G$. Set $R_1:=R\otimes_{\so_K} \so_{K_1}$. 
 The natural map
 $$\R\Gamma_{\synt}(R,r)\to  \R\Gamma(G,\R\Gamma_{\synt}(R_1,r))$$
 is a $p^{cre_1}$-quasi-isomorphism for some universal constant $c$.
\end{lemma}
\begin{remark}
\label{action5}
One can see the $G$-action on $\R\Gamma_{\synt}(R_1,r)$ and
the map $\R\Gamma_{\synt}(R,r)\to \R\Gamma_{\synt}(R_1,r)$  very explicitly. Namely, instead of the coordinate ring 
$\so_F[T_1]$ of $\so_{K_1}$, $T_1\mapsto \varpi_1$, one can take 
the coordinate ring  $\so_F[T,X_{\sigma}, \sigma\in G]$, $T\mapsto \varpi, X_{\sigma}\mapsto \sigma(\varpi_1)$. The action of 
$G$ on this coordinate ring is defined by
$g(T)=T,g(X_{\sigma})=X_{g\sigma},g\in G$. 
\end{remark}

 \begin{proof} 
 Since crystalline cohomology satisfies \'etale descent we may assume that our extension is totally ramified. Write
 \begin{align*}
 \R\Gamma_{\synt}(R,r) & =[\R\Gamma_{\crr}(R)^{\phi=p^r}\stackrel{\can}{\to}\R\Gamma_{\crr}(R)/F^r],\\
   \R\Gamma_{\crr}(R)^{\phi=p^r} & :=[\R\Gamma_{\crr}(R)\lomapr{p^r-\phi}\R\Gamma_{\crr}(R)]
 \end{align*}
  It suffices to show that the maps
 \begin{align*}
 \R\Gamma_{\crr}(R)^{\phi=p^r}  \stackrel{\sim}{\to}\R\Gamma(G,\R\Gamma_{\crr}(R_1)^{\phi=p^r}),\quad 
 \R\Gamma_{\crr}(R)/F^r  \stackrel{\sim}{\to}\R\Gamma(G,\R\Gamma_{\crr}(R_1)/F^r)
 \end{align*}
 are $p^{cre_1}$-quasi-isomorphisms. 
 
    To treat the first map consider
 the following diagram of maps of schemes 
 $$
 \xymatrix{
 R_0\ar@{^{(}->}[dr]^{i_{R}} \ar[dd]\ar@{^{(}->}[r]^{i_{R_1}} & R_1\ar[d]\\
& R\ar[d]\\
 k\ar@{^{(}->}[r] & \so_K
 }
 $$
It yields the  commutative diagram
$$
\xymatrix{
\R\Gamma_{\crr}(R)\ar[r]^{i_{R}^*}\ar[d] & \R\Gamma_{\crr}(R_0)\ar[d]\\
\R\Gamma(G,\R\Gamma_{\crr}(R_1))\ar[r]^{i_{R_1}^*} & \R\Gamma(G,\R\Gamma_{\crr}(R_0))
}
$$
Since the right vertical map is a $e_1$-quasi-isomorphism, 
it follows that 
  it suffices to show that the morphism
$$\R\Gamma_{\crr}(R)^{\phi=p^r}\stackrel{i_{R}^*}{\to}  \R\Gamma_{\crr}(R_0)^{\phi=p^r}
$$
and its analog for $R_1$ are $p^{cr}e_1$-quasi-isomorphisms. To see this for $R$ consider
the following factorization
$$
\phi^m:\quad 
 \R\Gamma_{\crr}({R/\O_F)^{\phi=p^r}}\stackrel{i_R^*}{\to} \R\Gamma_{\crr}(R_0/\O_F)^{\phi=p^r}\stackrel{j_m}{\to}
 \R\Gamma_{\crr}(R/\O_F)^{\phi=p^r}$$
of the $m$'th power of the Frobenius, where $p^m\geq e$. We also have $i_R^*j_m= \phi^m.$ Since $\phi^m$
is a $p^{2rm}$-quasi-isomorphism on $\R\Gamma_{\crr}(R)^{\phi=p^r}$ and on $\R\Gamma_{\crr}(R_0)^{\phi=p^r}$ 
both $i^*_R$ and $j_m$ are $p^{4rm}$-quasi-isomorphisms as
well. The reasoning in the case of $R_1$ is analogous.

  It remains to show that the map
  $$
  \R\Gamma_{\crr}(R)/F^r  \stackrel{\sim}{\to}\R\Gamma(G,\R\Gamma_{\crr}(R_1)/F^r)
  $$ is a $rp^{v_p(\delta_{K_1/K_0})}+e_1$-quasi-isomorphism. For that  recall that Beilinson's identification of  filtered log-crystalline cohomology with filtered derived log de Rham complex
  \cite[1.9.2]{Be} allows to prove that, for $n\geq 1$,  the natural map 
  $$\R\Gamma_{\crr}(R)_n/F^r\to \R\Gamma_{\crr}(R/\so^{\times}_K)_n/F^r\simeq \R\Gamma_{\dr}(R/\so_K)_n/F^r
  $$
  is a $rp^{v_p(\delta_{K/K_0})}$-quasi-isomorphism \cite[proof of Corollary 2.4]{NN}. Same is true for $R_1$ with a constant $rp^{v_p(\delta_{K_1/K_0})}$. It follows that it suffices to show that the map
  $$\R\Gamma_{\dr}(R/\so_K)_n/F^r\to \R\Gamma(G,\R\Gamma_{\dr}(R_1/\so_{K_1})_n/F^r)
  $$
  is a $e_1$-quasi-isomorphism. But, since, $\Omega\kr_{R_1/\so_{K_1}}=\Omega\kr_{R/\so_{K}}\otimes_{\so_K}\so_{K_1}$, this is clear.
 \end{proof}
 
   Using Corollary \ref{derham} and the proof of Theorem \ref{main},
we can relate de Rham and $p$-adic rigid \'etale cohomology. More specifically, let $\sx$ be a quasi-compact formal, semistable scheme over $\so_K$ (i.e., locally of the form described in section 2.1.2 with $h=1$). For $i\geq 0$, consider the composition
   $$
   \alpha_{r,i}: H^{i-1}_{\rm dR}(\sx_{K,\tr})\to H^i_{\synt}(\sx,r)_{\Q}\lomapr{\alpha^{FM}_r}H^i_{\eet}(\sx_{K,\tr},\Q_p(r)).
   $$
\begin{corollary}
For $i\leq r-1$, the map
$$
\alpha_{r,i}: H^{i-1}_{\rm dR}(\sx_{K,\tr}){\to} H^i_{\eet}(\sx_{K,\tr},\Q_p(r))
$$ is an isomorphism. Moreover, the map
$ \alpha_{r,r}: H^{r-1}_{\rm dR}(\sx_{K,\tr})\rightarrow  H^r_{\eet}(\sx_{K,\tr},\Q_p(r))$ is injective.
\end{corollary}
\subsubsection{Geometric syntomic cohomology}
   Recall the definition of the geometric syntomic cohomology, i.e., syntomic cohomology over $\ovk$. For a log-scheme $X$, log-smooth over $\so_K^{\times}$ and universally saturated, we have the absolute log-crystalline cohomology complexes and their completions
\begin{align*}
\R\Gamma_{\crr}(X_{{\so}_{\ovk}},\sj^{[r]})_n: &   =\R\Gamma_{\crr}(X_{\so_{\ovk},\eet},\R u_{X_{\so_{\overline{K}},n}/W_n(k)*}\sj^{[r]}_{X_{\so_{\ovk},n}/W_n(k)}),\\
\R\Gamma_{\crr}(X_{\so_{\ovk}},\sj^{[r]}): &    =
\holim_n\R\Gamma_{\crr}(X_{\so_{\ovk}},\sj^{[r]})_n,\\
\R\Gamma_{\crr}(X_{\so_{\ovk}},\sj^{[r]})_{\Q_p}: &    =\R\Gamma_{\crr}(X_{\so_{\ovk}},\sj^{[r]})\otimes{\Q_p}
\end{align*}
 By \cite[Theorem 1.18]{BE2}, if $X$ is proper, the complex $\R\Gamma_{\crr}(X_{\so_{\ovk}})$ is a perfect $\A_{\crr}$-complex and 
$$\R\Gamma_{\crr}(X_{\so_{\ovk}})_n\simeq \R\Gamma_{\crr}(X_{\so_{\ovk}})\otimes^{L}_{\A_{\crr}}{\A_{\crr}}/p^n\simeq \R\Gamma_{\crr}(X_{\so_{\ovk}})\otimes^{L}\Z/p^n.$$ In general, we have:
$$\R\Gamma_{\crr}(X_{\so_{\ovk}},\sj^{[r]})_n\simeq \R\Gamma_{\crr}(X_{\so_{\ovk}},\sj^{[r]})\otimes^{L}{\mathbf Z}/p^n.$$
Moreover,  $F^r\A_{\crr}=\R\Gamma_{\crr}(\Spec(\so_{\ovk}),\sj^{[r]})$ \cite[1.6.3,1.6.4]{Ts}. 

 For $r\geq 0$, the mod-$p^n$, completed, and rational syntomic complexes $\R\Gamma_{\synt}(X_{\so_{\ovk}},r)_n$, $\R\Gamma_{\synt}(X_{\so_{\ovk}},r)$, and $\R\Gamma_{\synt}(X_{\so_{\ovk}},r)_{\Q}$ are defined by the analogs of formulas we have used in the arithmetic case. We have $\R\Gamma_{\synt}(X_{\so_{\ovk}},r)_n\simeq \R\Gamma_{\synt}(X_{\so_{\ovk}},r)\otimes^{L}\Z/p^n$.
  Let $\sss(r)$  be the \'etale-sheafification of the presheaf sending 
\'etale map $U\to X_{\so_{\ovk}}$ to  $\R\Gamma_{\synt}(U,r)$.  Let  $\sss_n(r)$ denote
the \'etale-sheafification of
the mod-$p^n$ version of this presheaf.

 \begin{corollary}
\label{impl1}
{\rm (i)}
Let $X$ be a semistable scheme   over $\so_K$ or  a base change of a semistable scheme over the ring of integers of  a subfield of $K$. Then
for $0\leq j\leq r$, the kernel and cokernel of the map
$$
\alpha^{\rm FM}_{r,n}:\quad \sh^j(\sss_n(r)_{X_{\so_{\ovk}}})\to \overline{i}^*R^j\overline{j}_*\Z/p^n(r)^{\prime}_{X_K}
$$
is annihilated by $p^{Nr}$, for a universal constant $N$. Here $\overline{i}:X_0\hookrightarrow X_{\so_{\ovk}},\overline{j}:X_{\tr,\ovk}\hookrightarrow X_{\so_{\ovk}}$.

{\rm (ii)}
Let $X$ be a fine and saturated log-scheme log-smooth  over $\so_K^{\times}$. If $X$ is proper then the map
\begin{align*}
\alpha^{\rm FM}_r: \quad H^j_{\synt}(X_{\so_{\ovk}},r)_{\Q}\stackrel{\sim}{\to} H^j(X_{\tr,\ovk},\Q_p(r)),\quad j\leq r, 
\end{align*}
is  an isomorphism.
\end{corollary}
\begin{proof}
The first claim follows from Theorem \ref{main} by going to the limit over finite extensions of the base field. 

   It 
implies the second statement of the corollary for semistable schemes and with the group  $HN^j_{\synt}(X_{\so_{\ovk}},r)$ as the domain.  Here $HN^j_{\synt}(X_{\so_{\ovk}},r):=\invlim_nH^j\R\Gamma_{\synt}(X_{\so_{\ovk}},r)_n$, $j\geq 0$, 
  denotes the naive syntomic cohomology.  We have a natural map $H^j_{\synt}(X_{\so_{\ovk}},r)\to HN^j_{\synt}(X_{\so_{\ovk}},r)$. 
Since the groups $H^*(X_{\tr,\ovk},\Z/p^n(r))$ are finite all the relevant higher derived projective limits vanish 
and we can replace the naive syntomic groups with the ``real" ones. 

   To pass to general log-smooth schemes we use the following observation. 
We can assume that $K$ has enough roots of unity. By \cite[Theorem 2.9]{TS}, there exists a ramified extension $K_1$ of $K$ such that the base change $X_{\so_{K_1}}$ has a semistable model. That is, there exists a log-blow-up 
$\pi:Y \to X_{\so_{K_1}}$ such that $Y$ is a semistable scheme (with no multiplicities in the special fiber). We have the following commutative diagram
$$
\xymatrix{
H^j_{\synt}(Y,r)_{n}\ar[r]^-{\alpha^{\rm FM}_{r,n}} & H^j(Y_{\tr},\Z/p^n(r)^{\prime})\\
H^j_{\synt}(X_{\so_{K_1}},r)_{n}\ar[u]^{\pi^*}\ar[r]^-{\alpha^{\rm FM}_{r,n}} & H^j(X_{\tr,K_1},\Z/p^n(r)^{\prime})\ar@{=}[u]^{\pi^*}
}
$$
Since log-blow-ups do not change  syntomic cohomology \cite[Proposition 2.3]{Ni} and the top horizontal map is a $p^{Nr}$-quasi-isomorphism, for a constant $N$ independent of $n$,  so is the bottom one. 

  The final claim in the corollary is now clear.
\end{proof} 
\subsection{Semistable comparison theorems}
We show in this section how the comparison theorem for $p$-adic nearby cycles (cf. Corollary  \ref{impl1}) combined with the theory of finite dimensional Banach Spaces allows us to reprove the semistable comparison theorem for schemes and prove a semistable comparison theorem for formal schemes.
\subsubsection{Semistable conjecture for schemes}
 For $X$ proper and semistable  over $\so_K$ (with no multiplicities in the special fiber), we have \cite[3.2]{NN}
\begin{equation}
\label{presentation0}
\R\Gamma_{\synt}(X_{\so_{\ovk}},r)_{\Q}=[[\R\Gamma_{\hk}(X)\otimes_{F} \bstp]^{\phi=p^r,N=0}\stackrel{\iota_{\dr}}{\longrightarrow} (\R\Gamma_{\dr}(X_K)\otimes_K\bdrp)/F^r]
\end{equation}
Here $\R\Gamma_{\hk}(X)\simeq\R\Gamma_{\crr}(X_0/\O_F^{0})_{\Q}$ is the Hyodo-Kato cohomology of $X$ defined by Beilinson \cite{BE2} and
$\iota_{\dr}:\R\Gamma_{\hk}(X)\to  \R\Gamma_{\dr}(X_K)$ is the Hyodo-Kato map (associated to $\varpi$). The Hyodo-Kato complexes are built from finite rank $(\phi,N)$-modules. 
We set
$$
[\R\Gamma_{\hk}(X)\otimes_{F} \bstp]^{\phi=p^r,N=0}:=
\left[\begin{aligned}\xymatrix{\R\Gamma_{\hk}(X)\otimes_{F} \bstp\ar[r]^{1-\phi/p^r}\ar[d]^N & \R\Gamma_{\hk}(X)\otimes_{F} \bstp\ar[d]^{N}\\
\R\Gamma_{\hk}(X)\otimes_{F} \bstp\ar[r]^{1-\phi/p^{r-1}} &
\R\Gamma_{\hk}(X)\otimes_{F} \bstp
}\end{aligned}\right]
$$
Recall that  $H^j[\R\Gamma_{\hk}(X)\otimes_{F} \bstp]^{\phi=p^r,N=0}=(H^j_{\hk}(X)\otimes_{F}\bstp)^{\phi=p^r,N=0}$
  \cite[Corollary 3.25]{NN}. By the degeneration of the Hodge-de Rham spectral sequence \cite[Corollary 4.2.4]{DI} we also have  $ H^j((\R\Gamma_{\dr}(X_K)\otimes_K\bdrp)/F^r)=(H^j_{\dr}(X_K)\otimes_K\bdrp)/F^r$. 
It follows that we have the long exact sequence
\begin{align}
 \label{seqHK}
\to  (H^{j-1}_{\dr}(X_K)\otimes_K\bdrp)/F^r & \to H^j_{\synt}(X_{\so_{\ovk}},r)_{\Q}\to  (H^j_{\hk}(X)\otimes_{F}\bstp)^{\phi=p^r,N=0}\\
 & \to (H^j_{\dr}(X_K)\otimes_K\bdrp)/F^r\to \notag
\end{align}

\begin{corollary}\label{comp1}
{\rm (Semistable conjecture)}
Let $X$ be a proper, fine and saturated log-scheme, log-smooth over  $\so_K^{\times}$, with Cartier type reduction. 
There exists a natural $\bst $-linear Galois equivariant period isomorphism
$$
\alpha:\quad H^i(X_{\tr,\ovk},\Q_p)\otimes_{\Q_p}\bst \simeq H^i_{\hk}(X)\otimes_{F}\bst 
$$
that preserves the Frobenius and the monodromy operators, and, after extension to $\bdr$, induces a filtered  isomorphism 
$$
\alpha:\quad H^i(X_{\tr,\ovk},\Q_p)\otimes_{\Q_p}\bdr\simeq H^i_{\dr}(X_K)\otimes_{K}\bdr
$$
 \end{corollary}
\begin{proof}Take $r\geq i$. The period map is induced by the following composition
$$
 H^i(X_{tr,\ovk},\Q_p(r))\stackrel{\alpha^{\rm FM}_r}{\leftarrow} H^i_{\synt}(X_{\so_{\ovk}},r)_{\Q}\to (H^i_{\hk}(X)\otimes_{F}\bst )^{\phi=p^r,N=0}\to H^i_{\hk}(X)\otimes_{F}\bst .
$$
The first map is an isomorphism by Corollary \ref{impl1}. The period map  is clearly compatible with Galois action, Frobenius, monodromy, as well as with filtration 
(after extending to $\bdr$).
 To prove our corollary, it suffices to show that $H^i_{\dr}(X_K)$ is admissible and 
 that the above map induces an isomorphism $H^i(X_{\tr,\ovk},\Q_p)\simeq \vst(H^i_{\dr}(X_{K}))$. To do that we will use Corollary \ref{impl1}, exact sequence (\ref{seqHK}), and 
 the theory of finite dimensional Banach
 Spaces from \cite{CB} to prove Proposition \ref{VS2} below that will imply our corollary (take $D^i:=(H^i_{\hk}(X),
 H^i_{\dr}(X_{K}), \iota_{\dr}: H^i_{\hk}(X)\to H^i_{\dr}(X_{K})$). 
\end{proof}
 
\subsubsection{Finite dimensional Banach Spaces}
  Recall~\cite{CB} that a finite dimensional Banach Space $W$ is, 
morally, a finite dimensional $C$-vector space ($C:=\overline{K}^{\wedge}$) 
 up to  a finite dimensional $\Q_p$-vector space. It has a Dimension\footnote{In~\cite{CB},
the dimension is called the ``dimension principale'', noted $\dim_{\rm pr}$, and
the height is called the ``dimension r\'esiduelle'', noted $\dim_{\rm res}$, and
the Dimension is called simply the ``dimension''.}
  ${\rm Dim}\,W=(a,b)$, where $a=\dim W\in\N$, the {\it dimension of $W$}, is the dimension of the $C$-vector space 
and $b={\rm ht}\,W\in\Z$, the {\it height of $W$}, is the dimension of the $\Q_p$-vector space. 
More precisely, a {\it Banach Space} ${\mathbb W}$ is a functor
$\Lambda\mapsto {\mathbb W}(\Lambda)$, from the category of
sympathetic algebras (spectral Banach algebras, such that $x\mapsto x^p$
is surjective on $\{x,\ |x-1|<1\}$; such an algebra is, in particular, perfectoid)
to the category of $\Q_p$-Banach spaces.
Trivial examples of such objects are:

$\bullet$ 
finite dimensional $\Q_p$-vector spaces $V$,
with associated functor $\Lambda\mapsto V$ for all $\Lambda$,

$\bullet$ ${\mathbb V}^d$, for $d\in\N$, with ${\mathbb V}^d(\Lambda)=\Lambda^d$, for
all $\Lambda$.

A Banach Space ${\mathbb W}$ is {\it finite dimensional}
 if it ``is equal to ${\mathbb V}^d$, for some $d\in\N$, up to
finite dimensional $\Q_p$-vector spaces''.
More precisely, we ask that there exists finite dimensional $\Q_p$-vector spaces
$V_1,V_2$ and exact sequences\footnote{A sequence $0\to {\mathbb W}_1\to {\mathbb W}_2\to {\mathbb W}_3\to 0$
is exact if and only if 
$0\to {\mathbb W}_1(\Lambda)\to {\mathbb W}_2(\Lambda)\to {\mathbb W}_3(\Lambda)\to 0$
is exact for all $\Lambda$.}
$$0\to V_1\to {\mathbb Y}\to {\mathbb V}^d\to 0,\quad
0\to V_2\to {\mathbb Y}\to {\mathbb W}\to 0,$$
so that ${\mathbb W}$ is obtained from ${\mathbb V}^d$ by ``adding $V_1$ and
moding out by $V_2$''.
Then $\dim{\mathbb W}=d$ and ${\rm ht}\,{\mathbb W}=\dim_{\Q_p}V_1-
\dim_{\Q_p}V_2$. (We are, in general, only interested in $W={\mathbb W}(C)$ but, without
the extra structure, it would be impossible to speak of its Dimension.)
\begin{proposition}\label{BS1}
{\rm (i)} The Dimension of a finite dimensional Banach Space is independant
of the choices made in its definition.

{\rm (ii)} If $f:{\mathbb W}_1\to{\mathbb W}_2$ is a morphism of
finite dimensional Banach Spaces, then ${\rm Ker}\,f$, ${\rm Coker}\,f$ and ${\rm Im}\,f$
are finite dimensional Banach Spaces, and
we have 
$${\rm Dim}\, {\mathbb W}_1={\rm Dim}\, {\rm Ker}\,f+{\rm Dim}\,{\rm Im}\,f
\quad{\rm and}\quad
{\rm Dim}\, {\mathbb W}_2={\rm Dim}\, {\rm Coker}\,f+{\rm Dim}\,{\rm Im}\,f.$$

{\rm (iii)} If $\dim{\mathbb W}=0$, then ${\rm ht}\,{\mathbb W}\geq 0$.

{\rm (iv)} If ${\mathbb W}$ has an increasing filtration such that the successive quotients
are ${\mathbb V}^1$, and if ${\mathbb W}'$ is a sub-Banach Space of ${\mathbb W}$, then
${\rm ht}\,{\mathbb W}'\geq 0$.
\end{proposition}
\begin{proof}
The first two points are the core of the theory~\cite[Th.~0.4]{CB}.
The third point is obvious and the fourth is \cite[Lemme 2.6]{CF}.
\end{proof}
\begin{corollary}\label{BS2}
{\rm (i)} If ${\mathbb W}_1$ is a successive extension of ${\mathbb V}^1$'s, and if
${\mathbb W}_2$ is of dimension $0$, then any morphism
${\mathbb W}_1\to {\mathbb W}_2$ is the $0$-map.

{\rm (ii)} Let ${\mathbb W}_1, {\mathbb W}_2$ be finite dimensional Banach Spaces,
$W_1={\mathbb W}_1(C)$ and $W_2={\mathbb W}_2(C)$.
Suppose that ${\mathbb W}_2$ is a successive extension of ${\mathbb V}^1$'s,
and that we are given a $\Q_p$-linear map $f:W_1\to W_2$ which lifts
to a morphism
$f:{\mathbb W}_1\to {\mathbb W}_2$ of finite dimensional Banach Spaces.
Then, if we are in one of these situations:

\quad$\bullet$ $\dim_{\Q_p}{\rm Coker}\big(f:W_1\to W_2\big)<\infty$,

\quad$\bullet$ $\dim_{\Q_p}{\rm Ker}\big(f:W_1\to W_2\big)<\infty$ and $\dim{\mathbb W}_1=
\dim{\mathbb W}_2$,

\noindent
the map $f:W_1\to W_2$ is surjective.

{\rm (iii)} If $f:{\mathbb W}_1\to {\mathbb W}_2$ is a 
morphism of finite dimensional Banach Spaces, and
if the kernel and cokernel of
$f_C:{\mathbb W}_1(C)\to {\mathbb W}_2(C)$ are
finite dimensional over $\Q_p$, then
$\dim{\mathbb W}_1=\dim{\mathbb W}_2$.
\end{corollary}
\begin{proof}
Let $f:{\mathbb W}_1\to {\mathbb W}_2$ be a morphism.
We have $\dim {\rm Im}\,f=0$ since $\dim {\mathbb W}_2=0$.
Now, ${\rm ht}\,{\mathbb W}_1=0$; hence ${\rm ht}\,{\rm Im}\,f=-{\rm ht}\,{\rm Ker}\,f$, and
${\rm ht}\,{\rm Ker}\,f\geq 0$ by (iv) of Proposition~\ref{BS1}
and our assumption on ${\mathbb W}_1$.
So ${\rm ht}\,{\rm Im}\,f\leq 0$, and since
$\dim {\rm Im}\,f=0$, we obtain ${\rm Im}\,f= 0$,
as wanted.

To prove (ii) in the first situation, 
note that the Banach Space ${\rm Coker}(f:{\mathbb W}_1\to {\mathbb W}_2)$
is of dimension~$0$ since its $C$-points are finite dimensional over $\Q_p$ by assumption.
Hence we can apply (i) to deduce that this cokernel is~$0$, hence also
the cokernel of $f:W_1\to W_2$.
In the second situation, the assumption $\dim_{\Q_p}{\rm Ker}\big(f:W_1\to W_2\big)<\infty$
 imply that $\dim{\rm Ker}\big(f:{\mathbb W}_1\to {\mathbb W}_2\big)=0$
and, since $\dim{\mathbb W}_1=
\dim{\mathbb W}_2$, this implies that 
$\dim{\rm Coker}\big(f:{\mathbb W}_1\to {\mathbb W}_2\big)=0$, and we conclude as before.

Finally, (iii) is a consequence of (ii) of Proposition~\ref{BS1},
which implies (forgetting the heights), $$\dim{\mathbb W}_1-\dim{\mathbb W}_2=
\dim\,{\rm Ker}\,f-\dim\,{\rm Coker}\,f,$$
and the fact that $\dim\,{\rm Ker}\,f=\dim\,{\rm Coker}\,f=0$
since their $C$ points are of finite dimension over $\Q_p$ by assumption.
\end{proof}

\begin{example}
  Let $D=(D_{\rm st},D_{\rm dR},\lambda)$, where:
\begin{itemize}
\item $D_{\rm st}$ is a finite dimensional $F$-vector space
with a bijective semi-linear Frobenius $\varphi$, and a linear monodromy operator $N$ such that $N\varphi=p\varphi N$,
\item  $D_{\rm dR}$ is a finite dimensional $K$-vector space with a
decreasing, separated, exhausting filtration indexed by $\N$,
\item 
$\lambda:D_{\rm st}\to D_{\rm dR}$ is a $F$-linear map.
\end{itemize}
If $r\in\N$, define
$$X^r_{\rm st}(D)=(t^{-r}\bstp\otimes_F D_{\rm st})^{\varphi=1, N=0}
\quad{\rm and}\quad
X^r_{\rm dR}(D)=(t^{-r}\bdrp\otimes_K D_{\rm dR})/{F}^0.$$
These are the $C$-points of finite dimensional Banach Spaces 
${\mathbb X}^r_{\rm st}(D)$ and ${\mathbb X}^r_{\rm dR}(D)$,
and we have (cf. \cite[Proposition ~0.8]{CB}),
\begin{align*}
\dim {\mathbb X}^r_{\rm dR}(D)=&\ (r\dim_KD_{\rm dR}-\sum_{i=1}^r\dim F^iD_{\rm dR},0)\\
\dim {\mathbb X}^r_{\rm st}(D)=& \sum_{r_i\leq r}(r-r_i,1), \quad{\text{where the $r_i$'s are the
slopes of $\varphi$, repeated with multiplicity.}}
\end{align*}
In particular, if $F^{r+1}D_{\rm dR}=0$ and if all $r_i$'s are $\leq r$ (we let $r(D)$ be the smallest $r$
with these properties), then
$$\dim {\mathbb X}^r_{\rm st}(D)=
(r\dim_FD_{\rm st}-t_N(D_{\rm st}),\dim_FD_{\rm st})
\quad{\rm and}\quad
\dim {\mathbb X}^r_{\rm dR}(D)=(r\dim_KD_{\rm dR}-t_H(D_{\rm dR}),0).$$
Here $t_N(D)=v_p(\det\phi)$ and $t_H(D)=\sum_{i\geq 0}i\dim_K(F^iD_{\rm dR}/F^{i+1}D_{\rm dR})$. 

Now, the map $\lambda$ extends (using the natural map 
${\mathbb B}_{\rm st}^+\to {\mathbb B}_{\rm dR}^+$) to a map 
${\mathbb X}^r_{\rm st}(D)\to {\mathbb X}^r_{\rm dR}(D)$ of finite dimensional Banach Spaces
and, if $\lambda:K\otimes_FD_{\rm st}\to D_{\rm dR}$ is bijective
(in which case we set ${\rm rk}\, D=\dim_FD_{\rm st}=\dim_KD_{\rm dR}$),
then the kernel of the $C$ points of this map is $\vst(D)$ if $r\geq r(D)$
(\cite[Proposition ~10.14]{CB}).
\end{example}

The following result is a variant of the theorem ``weakly admissible implies
admissible''.

\begin{lemma}\label{VS1} If $\lambda:K\otimes_FD_{\rm st}{\to }D_{\rm dR}$ is an isomorphism, and
if $t_H(D_{\rm dR})=t_N(D_{\rm st})$, the following
conditions are equivalent:

{\rm (i)}  $\vst(D)$ is finite dimensional over $\Q_p$,

{\rm (ii)} the map $X^r_{\rm st}(D)\to X^r_{\rm dR}(D)$ is surjective for $r=r(D)$.

{\rm (ii$'$)} the map $X^r_{\rm st}(D)\to X^r_{\rm dR}(D)$ is surjective for all $r\geq r(D)$.

\noindent Moreover, they imply:

{\rm (iii)} $t_N(D')\geq t_H(D')$ for all $D'\subset D$
stable by $N$ and $\varphi$ (i.e.~$D$ is weakly admissible).

{\rm (iv)} $\dim_{\Q_p}\vst(D)={\rm rk}\, D$.
\end{lemma}
\begin{proof}
The hypothesis $t_H(D_{\rm dR})=t_N(D_{\rm st})$
implies that $\dim ({\mathbb X}^r_{\rm dR}(D))=\dim ({\mathbb X}^r_{\rm st}(D))$ for all $r\geq r(D)$.
Now, ${\mathbb X}^r_{\rm dR}(D)$ is a sucessive extension of ${\mathbb V}^1$'s; hence,
if $\vst(D)$ is finite dimensional over $\Q_p$,
we are in the second situation considered in Lemma~\ref{BS2}~(ii),
which proves the implications (i)$\Rightarrow$(ii) and (i)$\Rightarrow$(ii$'$).

The converse implications and (iv)
are just Dimension arguments.

Now,
if there exists $D'$ with $t_N(D')<t_H(D')$, then $\dim ({\mathbb X}^r_{\rm st}(D'))>
\dim ({\mathbb X}^r_{\rm dR}(D'))$ if $r$ is big enough, so
the dimension of ${\rm Ker}\big({\mathbb X}^r_{\rm st}(D')\to {\mathbb X}^r_{\rm dR}(D')\big)$
is~$\geq 1$ and $\vst(D')$ is infinite dimensional
over $\Q_p$.  As this contradicts our hypothesis that 
$\vst(D)$ is finite dimensional over $\Q_p$, this proves (iii).
\end{proof}

\begin{proposition}\label{VS2}
Assume that, for $r\in\N$, 
we have a set of $D^i=(D^i_{\rm st},D^i_{\rm dR},\lambda^i)$, $i\in\N$, as above but with
the additional condition $F^{i+1}D^i_{\rm dR}=0$, 
and long exact sequences
$$\cdots \to H^i(r)\to X^r_{\rm st}(D^i)\to X^r_{\rm dR}(D^i)\to H^{i+1}(r)\to\cdots$$ Assume 
that $\dim_{\Q_p}H^i(r)$ is finite if $r\geq i$.
Then:

{\rm (i)} the map $\lambda^i:K\otimes_FD_{\rm st}^i\to D_{\rm dR}^i$ is bijective, if $i\geq 0$,

{\rm (ii)} $D^i$ is weakly admissible, if $i\geq 0$,

{\rm (iii)} the map $X^r_{\rm st}(D^i)\to X^r_{\rm dR}(D^i)$ is surjective, if $r\geq i$,

{\rm (iv)} $H^i(r)=\vst(D^i)$,  if $r\geq i$.
\end{proposition}
\begin{proof}
To see (i) note that the fact that $H^i(r)$ and $H^{i+1}(r)$ are finite dimensional over
$\Q_p$ for $r$ big enough implies that
 $\dim ({\mathbb X}^r_{\rm dR}(D^i))=\dim ({\mathbb X}^r_{\rm st}(D^i))$
for all such $r$ by (iii) of Corollary~\ref{BS2}. 
This, in turn, implies that $\dim_FD^i_{\rm st}=\dim_KD^i_{\rm dR}$
and $t_H(D^i)=t_N(D^i)$.  

Now, let $D':={\rm Im}(\lambda^i)\subset D^i_{\dr}$. Then the image of ${\mathbb X}^r_{\rm st}(D^i)$ in $ {\mathbb X}^r_{\rm dR}(D^i)$
is included in ${\mathbb X}^r_{\rm dR}(D')$ and so the cokernel surjects onto
${\mathbb X}^r_{\rm dR}(D^i/D')=0$.  Since its $C$-points inject into $H^{i+1}(r)$ 
which is finite dimensional over $\Q_p$,
it follows that $\dim({\mathbb X}^r_{\rm dR}(D^i/D'))=0$ if $r\geq i+1$.  This implies that $D^i/D'=0$,
and that $\lambda^i$ is surjective, hence bijective.

Claim (ii) follows from claim (i) and Lemma \ref{VS1}. Note that it implies that $r(D^i)\leq i$ since
$F^{i+1}D^i_{\rm dR}=0$ by assumption.

If $i\leq r-1$, the map $X_{\rm dR}^r(D^i)\to H^{i+1}(r)$ is the zero map, since the source
is the $C$-points of a successive extension of 
${\mathbb V}^1$'s and the target is a finite dimensional $\Q_p$-vector space
by assumption.
This allows to split the long exact sequence into short exact sequences
$0\to H^i(r)\to X^r_{\rm st}(D^i)\to X^r_{\rm dR}(D^i)\to 0$ for $i\leq r-1$
and $0\to H^r(r)\to X^r_{\rm st}(D^r)\to X^r_{\rm dR}(D^r)$. 
We conclude that $H^i(r)=\vst(D^i)$,  if $i\leq r$, using inequality $r(D^i)\leq i$
and Lemma~\ref{VS1}.
\end{proof}
\subsubsection{Semistable conjecture for formal schemes}
   Let $\sx$ be a fine log ($p$-adic) formal scheme, log-smooth over $\so_K^{\times}$, and universally saturated. We define the arithmetic and geometric syntomic cohomology of $\sx$ by the same formulas as in
    the algebraic setting. 

 For $\sx$ semistable over $\so_K$, we have a rigid analytic version of the period map of Fontaine-Messing 
$$\alpha^{\rm FM}_{r,n}:  \quad H^j_{\synt}(\sx,r)_n\to H^j(\sx_{K,\tr},\Z/p^n(r)^{\prime}).
$$
 The definition is a straightforward analog of the algebraic definition \cite[p. 321]{Ts}. Namely, we cover $\sx$ with open affine formal schemes of the right presentation, use the Fontaine-Messing map on each open set almost verbatim to end up with a \v{C}ech complex of Galois cochains. Then we use the $K(\pi,1)$-Lemma for $p$-coefficients
 to pass to a \v{C}ech complex of \'etale cohomology. That is, in the local version of this map one combines the map $\tilde{\alpha}^{\rm FM}_{r,n}:{\rm Syn}(R,r)_n\to C(G_{R},{\mathbf Z}/p^n(r)^{\prime})$ with the quasi-isomorphism
$C(G_{R},{\mathbf Z}/p^n(r)^{\prime})\stackrel{\sim}{\to} \R\Gamma((\Sp(R[1/p]))_{\tr,\eet},{\mathbf Z}/p^n(r)^{\prime})$. This quasi-isomorphism is obtained from the Riemann's Existence Theorem \ref{riemann} by Abhyankar's Lemma argument as in Lemma \ref{Abh}.   Finally, we descend. 
\begin{corollary}
\label{impl2}
Let $\sx$ be a proper semistable formal  scheme over $\so_K$.  Then the period map
$$\alpha^{\rm FM}_r:\quad H^j_{\synt}(\sx_{\so_{\ovk}},r)_{\Q}\stackrel{\sim}{\to} H^j(\sx_{\ovk,\tr},\Q_p(r)), \quad j\leq r,$$
is an isomorphism.
\end{corollary}
\begin{proof}
The first claim follows from Theorem \ref{main} by going to the limit over finite extensions of the base field. It 
implies the second statement with the naive syntomic cohomology groups  $HN^j_{\synt}(\sx_{\so_{\ovk}},r):=\invlim_nH^j\R\Gamma_{\synt}(\sx_{\so_{\ovk}},r)_n$  in place of the syntomic cohomology groups.
Since the groups $H^*(\sx_{\ovk,\tr},\Z/p^n(r))$ are finite \cite[Theorem 1.1]{Sch}   all the relevant higher derived projective limits vanish 
and we can replace the naive syntomic groups with the ``real" ones. The final claim in the corollary is now clear.
\end{proof} 

  To proceed, we need to derive a presentation of the geometric syntomic cohomology using Hyodo-Kato and de Rham cohomologies analogous to the one we have used in the algebraic setting. Assume that the log-scheme $\sx_1$ is of Cartier type over $\Spec(\so_{K,1})^{\times}$. We have the Hyodo-Kato cohomology complex $\R\Gamma_{\hk}(\sx)\simeq\R\Gamma_{\crr}(\sx_0/\so^0_{F})_{\Q}$ and 
  the quasi-isomorphisms \cite[1.16.2]{BE2}
 $$
 (\R\Gamma_{\hk}(\sx)\otimes_{F}\bstp)^{N=0}\simeq \R\Gamma_{\hk}(\sx)\otimes_{F}{\B^+_{\crr}},\quad \iota_{\crr}: \R\Gamma_{\hk}(\sx)\otimes_{F}{\B^+_{\crr}}\stackrel{\sim}{\to}  \R\Gamma_{\crr}(\sx)_{\mathbf Q},
 $$ that are compatible with the action of $N$ and $\phi$. Recall that the complex 
 $\R\Gamma_{\hk}(\sx)$ is built from finite rank $(\phi,N)$-modules over $F$. We also have the Hyodo-Kato isomorphism \cite[1.16.2]{BE2}
 $$
\iota_{\dr}:\quad  \R\Gamma_{\hk}(\sx)\otimes_{F}K\stackrel{\sim}{\to} \R\Gamma_{\rm dR}(\sx_{K}).$$
 Since, for $\sx$ semistable over $\so_K$, we have $\R\Gamma_{\rm dR}(\sx_{K})\simeq \R\Gamma_{\rm dR}(\sx_{K,\tr})$, in that case the Hyodo-Kato isomorphism has the following form
$$
\iota_{\dr}:\quad  \R\Gamma_{\hk}(\sx)\otimes_{F}K\stackrel{\sim}{\to} \R\Gamma_{\rm dR}(\sx_{K,\tr}).
$$
 \begin{proposition}
  For $\sx$ proper and semistable  over $\so_K$ with Cartier type reduction, we have 

  {\rm (i)} the quasi-isomorphism
\begin{equation}
\label{presentation}
\R\Gamma_{\synt}(\sx_{\so_{\ovk}},r)_{\Q}\simeq [[\R\Gamma_{\hk}(\sx)\otimes_{F} \bstp]^{\phi=p^r,N=0}\lomapr{\iota_{\dr}}(\R\Gamma_{\dr}(\sx_K)\otimes_K\bdrp)/F^r]
\end{equation}

{\rm (ii)} 
for $j < r$, the following short exact sequences
\begin{equation}
\label{pierre-mod}
0\to H^j_{\synt}(\sx_{\so_{\ovk}},r)_{\Q}\to  (H^j_{\hk}(\sx)\otimes_{F}\bstp)^{\phi=p^r,N=0}\to   (H^j_{\dr}(\sx_K)\otimes_K\bdrp)/F^r \to 0
\end{equation}
\end{proposition}
\begin{proof}
Recall \cite{BE2} that we have the complexes 
\begin{align*}
\R\Gamma_{\dr}^{\natural}(\sx_{\so_{\overline{K}}})_n: = & 
\R\Gamma(\sx_{\so_{\overline{K}},\eet},\LL\Omega^{\scriptscriptstyle\bullet,\wedge}_{\sx_{\so_{\overline{K},n}}/\so_{F,n}}),\quad 
\R\Gamma_{\dr}^{\natural}(\sx_{\so_{\overline{K}}})\widehat{\otimes}{\mathbf Z}_p:=  \holim_n \R\Gamma_{\dr}^{\natural}(\sx_{\so_{\overline{K}}})_n,\\
\R\Gamma_{\dr}^{\natural}(\sx_{\so_{\overline{K}}})\widehat{\otimes}{\mathbf Q}_p:=  & (\R\Gamma_{\dr}^{\natural}(\sx_{\so_{\overline{K}}})\widehat{\otimes}{\mathbf Z}_p)\otimes {\mathbf Q}.
\end{align*}
These are objects in the filtered $\infty$-derived category.  Here $\LL\Omega^{\scriptscriptstyle\bullet,\wedge}_{\sx_{\so_{\overline{K},n}}/\so_{F,n}}$ denotes the derived log de Rham complex (see \cite{Be1} for a review). The hat refers to the completion with respect to the Hodge filtration (in the sense of prosystems). Set $\A_{\dr}:=\LL\Omega^{\scriptscriptstyle\bullet,\wedge}_{\so_{\overline{K}}/\so_K}.$  By \cite[Lemma 3.2]{BE2} we have 
 $\A_{\dr}=\R\Gamma^{\natural}_{\dr}(\so_{\overline{K}}^{\times})$. The corresponding $F$-filtered algebras $\A_{\dr,n}$, $\A_{\dr}\widehat{\otimes}{\mathbf Z}_p$, $\A_{\dr}\widehat{\otimes}{\mathbf Q}_p$ are acyclic in nonzero degrees and the projections $\cdot/F^{m+1}\to \cdot/F^m$ are surjective. 
 Thus (we set $\lim_F:=\holim_F$)
 \begin{align*}
   \lim_F\A_{\dr}\widehat{\otimes}{\mathbf Q}_p =\invlim_{m}H^0(\A_{\dr}\widehat{\otimes}{\mathbf Q}_p/F^m)=\bdrp, \quad \A_{\dr}\widehat{\otimes}{\mathbf Q}_p/F^m=\bdrp/F^m.
    \end{align*}

   Recall that we have the quasi-isomorphism \cite[Thm 2.1]{NN}
$$
\kappa^{-1}_r:\quad \R\Gamma_{\crr}(\sx_{\so_{\ovk}})_{\Q}/F^r\stackrel{\sim}{\to} \R\Gamma^{\natural}_{\dr}(\sx_{\so_{\ovk}})\widehat{\otimes}{\Q}_p /F^r.
$$
This yields the first  quasi-isomorphism in the following diagram.
\begin{align*}
\R\Gamma_{\synt}(\sx_{\so_{\ovk}},r)_{\Q}
 & \stackrel{\sim}{\to}
   \xymatrix{[\R\Gamma_{\crr}(\sx_{\so_{\ovk}})_{\Q}\ar[rr]^-{(1-\phi_r,\kappa_r^{-1})} & & \R\Gamma_{\crr}(\sx_{\so_{\ovk}})_{\Q}\oplus  (\R\Gamma^{\natural}_{\dr}(\sx_{\so_{\ovk}})\widehat{\otimes}{\Q}_p) /F^r]}\\
& \stackrel{\sim}{\to}
   \xymatrix{[\R\Gamma_{\crr}(\sx_{\so_{\ovk}})_{\Q}\ar[rr]^-{(1-\phi_r,\gamma_{r}^{-1}\kappa_r^{-1})} & & \R\Gamma_{\crr}(\sx_{\so_{\ovk}})_{\mathbf Q}\oplus 
    (\R\Gamma_{\dr}(\sx_K)\otimes_{K}\bdrp) /F^r]}
\end{align*}
To define the second quasi-isomorphism note that the natural map $\R\Gamma^{\natural}_{\dr}(\sx)\widehat{\otimes}{\mathbf Q}_p\to \R\Gamma_{\dr}(\sx_K)$ is a filtered quasi-isomorphism: 
 it suffices to show that the natural map 
$$
  \gr^i_{F}\R\Gamma(\sx_{\eet},\LL\Omega^{\scriptscriptstyle\bullet,\wedge}_{\sx/\so_F})\wh{\otimes}\Q_p \to \gr^i_F\R\Gamma (\sx_{\eet},\LL\Omega^{\scriptscriptstyle\bullet,\wedge}_{\sx/\so_K^{\times}})\wh{\otimes}\Q_p
$$
  is a quasi-isomorphism for all $i\geq 0$ and this was done in \cite[proof of Corollary 2.4]{NN}. 
It
  yields, by extension to $\A_{\dr}\widehat{\otimes}{\mathbf Q}_p$, a quasi-isomorphism  of $F$-filtered $\ovk$-algebras 
  $$
\gamma:\quad  \R\Gamma_{\dr}(\sx_K)\otimes_{K}(\A_{\dr}\widehat{\otimes}{\mathbf Q}_p)\stackrel{\sim}{\to}\R\Gamma^{\natural}_{\dr}(\sx_{\so_{\ovk}})\widehat{\otimes}{\mathbf Q}_p:
$$
one can use here the arguments of \cite[3.5]{BE2} almost verbatim. 
Its mod $F^r$-version $\gamma_r$  is the quasi-isomorphism
$$\gamma_r:  \quad (\R\Gamma_{\dr}(\sx_K)\otimes_{K}\bdrp)/F^r\stackrel{\sim}{\to}  \R\Gamma^{\natural}_{\dr}(\sx_{\so_{\ovk}})/F^r
$$

  The quasi-isomorphism (\ref{presentation}) is now defined by  the following composition of morphisms
 \begin{align*}
\R\Gamma_{\synt}(\sx_{\so_{\ovk}},r)
 & \stackrel{\sim}{\to}
  \left[\xymatrix@C=36pt{\R\Gamma_{\crr}(\sx_{\so_{\ovk}})_{\Q}\ar[rr]^-{(1-\phi_r,\gamma^{-1}_r\kappa_r^{-1})} & & \R\Gamma_{\crr}(\sx_{\so_{\ovk}})_{\Q}\oplus (\R\Gamma_{\dr}(\sx_K)\otimes_{\ovk}\bdrp) /F^r}\right]\notag\\
 &   \stackrel{\sim}{\leftarrow}
   \left[\begin{aligned}\xymatrix@C=50pt{\R\Gamma_{\hk}(\sx)\otimes_{F}\bstp\ar[r]^-{(1-\phi_r,\iota_{\dr}\otimes\iota)}\ar[d]^{N}  & \R\Gamma_{\hk}(\sx)\otimes_{F}\bstp\oplus (\R\Gamma_{\dr}(\sx_K)\otimes_{\ovk}\bdrp) /F^r\ar[d]^{(N,0)}\\
\R\Gamma_{\hk}(\sx)\otimes_{F}\bstp\ar[r]^{1-\phi_{r-1}}  & \R\Gamma_{\hk}(\sx)\otimes_{F}\bstp}\end{aligned}\right]
 \end{align*}
Here  second quasi-isomorphism 
 uses    the fact that the following diagram commutes (proof of \cite[Lemma 3.23]{NN} goes through)
 $$ \xymatrix@C=36pt{
 \R\Gamma_{\crr}(\sx_{\so_{\ovk}})_{\mathbf Q}\otimes_{\B^+_{\crr}}\bstp\ar[r]^-{\gamma^{-1}_r\kappa_r^{-1}\otimes \iota} & (\R\Gamma_{\dr}(\sx_K)\otimes_{K}\bdrp) /F^r\\
 \R\Gamma_{\hk}(\sx)\otimes_{F}\bstp\ar[ru]_-{\iota_{\dr}\otimes\iota}\ar[u]^{\iota_{cr}}_{\wr} 
  }
  $$
  This proves the first claim of our proposition. 
  
    For the second claim,  taking cohomology of the homotopy fiber (\ref{presentation}) we get the long exact sequence
\begin{align*}
\to  H^{j-1}((\R\Gamma_{\dr}(\sx_K)\otimes_K\bdrp)/F^r) & \to H^j_{\synt}(\sx_{\so_{\ovk}},r)_{\Q}\to  H^j[\R\Gamma_{\hk}(\sx)\otimes_{F} \bstp]^{\phi=p^r,N=0}\\
 &  \to     H^j((\R\Gamma_{\dr}(\sx_K)\otimes_K\bdrp)/F^r) \to \notag
\end{align*}
   But, arguing as in the algebraic case   \cite[Corollary 3.25]{NN}, we obtain the isomorphism
$$
   H^j[\R\Gamma_{\hk}(\sx)\otimes_{F} \bstp]^{\phi=p^r,N=0}  \simeq (H^j_{\hk}(\sx)\otimes_{F}\bstp)^{\phi=p^r,N=0}. 
   $$
Hence we have the long exact sequence
\begin{align*}
\to  H^{j-1}((\R\Gamma_{\dr}(\sx_K)\otimes_K\bdrp)/F^r) & \to H^j_{\synt}(\sx_{\so_{\ovk}},r)_{\Q}\to  (H^j_{\hk}(\sx)\otimes_{F}\bstp)^{\phi=p^r,N=0}\\
 & \to     H^j((\R\Gamma_{\dr}(\sx_K)\otimes_K\bdrp)/F^r) \to \notag
\end{align*}
To show that it splits into short exact sequences for $j\leq r$, we start with the observation that the cohomology groups $ H^j((\R\Gamma_{\dr}(\sx_K)\otimes_K\bdrp)/F^r)$ are 
finite dimensional Banach Spaces which are successive extensions of ${\mathbb V}^1$.
 This is because they are finite length $\bdrp$-modules: we have
$$
 H^j(\gr^i_F(\R\Gamma_{\dr}(\sx_K)\otimes_K\bdrp))\simeq \bigoplus_{k\geq 0}H^{j-k}(\Omega^k_{\sx_K})\otimes_K \gr^{i-k}_F\bdrp.
 $$ 
We get the required short exact sequences by the same argument that we
have used in the proof of Proposition ~\ref{VS2} (the key point being 
 that, by Corollary \ref{impl2},  the groups $H^j_{\synt}(\sx_{\so_{\ovk}},r)_{\Q}$, $j\leq r$,  are finite dimensional vector spaces over $\Q_p$).
 
 It remains to show that, for $j<r$ , we have an isomorphism 
 $$
 H^j((\R\Gamma_{\dr}(\sx_K)\otimes_K\bdrp)/F^r)  \simeq (H^j_{\dr}(\sx_K)\otimes_K\bdrp)/F^r.
 $$
 Just as in the algebraic case this statement (for all $j$) is an immediate result of the degeneration of the Hodge-de Rham spectral sequence
which follows from 
the de Rham comparison theorem for proper rigid analytic spaces proved by Scholze \cite{Sch}. 
It also can be proved directly via the following arguments.
 By  (\ref{pierre-mod}),  we have a surjection
$$(H^j_{\hk}(\sx)\otimes_{F}\bstp)^{\phi=p^r,N=0}\to H^j((\R\Gamma_{\dr}(\sx_K)\otimes_K\bdrp)/F^r), \quad j<r.
$$
Since the  above map 
 factors through  the natural map \begin{equation}
\label{injective}
(H^j_{\dr}(\sx_K)\otimes_K\bdrp)/F^r\to H^j((\R\Gamma_{\dr}(\sx_K)\otimes_K\bdrp)/F^r),
\end{equation}
that latter is surjective as well. But it is also injective. Indeed, we have the distinguished triangle
$$
F^r(\R\Gamma_{\dr}(\sx_K)\otimes_K\bdrp)\to \R\Gamma_{\dr}(\sx_K)\otimes_K\bdrp\to (\R\Gamma_{\dr}(\sx_K)\otimes_K\bdrp)/F^r
$$
It yields the long exact sequence of cohomology groups
$$
\to H^j F^r(\R\Gamma_{\dr}(\sx_K)\otimes_K\bdrp)\stackrel{f_j}{\to} H^j _{\dr}(\sx_K)\otimes_K\bdrp\to H^j(\R\Gamma_{\dr}(\sx_K)\otimes_K\bdrp)/F^r)\to $$
Since $F^r(H^j_{\dr}(\sx_K)\otimes_K\bdrp)=\im f_j$, the map in (\ref{injective}) is injective. We are done. 
\end{proof}

\begin{corollary}
\label{comp2}
{\rm (Semistable conjecture)}
 Let $\sx$ be a proper semistable formal scheme  over $\so_K$. 
There exists a natural $\bst$-linear Galois equivariant period isomorphism
$$
\alpha:\quad H^j(\sx_{\ovk,\tr},\Q_p)\otimes_{\Q_p}\bst \simeq H^j_{\hk}(\sx)\otimes_{F}\bst 
$$
that preserves the Frobenius and the monodromy operators, and induces a filtered  isomorphism 
$$
\alpha:\quad H^j(\sx_{\ovk,\tr},\Q_p)\otimes_{\Q_p}\bdr\simeq H^j_{\dr}(\sx_{K,\tr})\otimes_{K}\bdr.
$$
 \end{corollary}
\begin{proof}Take $r > j$. 
 The period map is induced by the following composition
$$
 H^j(\sx_{\ovk,\tr},\Q_p(r))\stackrel{\alpha^{\rm FM}_r}{\leftarrow} H^j_{\synt}(\sx_{\so_{\ovk}},r)_{\Q}\to (H^j_{\hk}(\sx)\otimes_{F}\bst )^{\phi=p^r,N=0}\to H^j_{\hk}(\sx)\otimes_{F}\bst .
$$
The first morphism is an isomorphism by Corollary \ref{impl2}. 
The proof of our corollary  proceeds now as the proof of Corollary \ref{comp1} using the short  exact sequence (\ref{pierre-mod}). 
\end{proof}
\begin{remark}
It is shown in \cite{NL} that the distinguished triangles (\ref{presentation0}) and (\ref{presentation}) lift canonically to the category of finite 
dimensional Banach Spaces.
\end{remark}
\begin{remark}It is likely that 
de Rham conjecture for smooth and proper rigid analytic spaces proved by Scholze \cite{Sch}  can be derived from the semistable comparison in Corollary \ref{comp2} using the existence of local (in the \'etale topology) semistable formal models of smooth rigid analytic spaces proved by Hartl \cite{Hrt} and the ``gluing on the generic fiber" techniques of \cite{BE2}, \cite{NN}.
\end{remark}

\end{document}